\documentclass[11pt, a4paper]{amsart}
\usepackage[all]{xy}
\usepackage{iftex}
\ifPDFTeX
\usepackage[utf8]{inputenc}
\fi
\usepackage[T1]{fontenc}
\usepackage{stix2} 
\usepackage{graphicx}
\usepackage[backend=biber, style=alphabetic, sorting=nyt, maxnames=6]{biblatex}

\usepackage{mathtools}
\numberwithin{equation}{section}

\DeclarePairedDelimiter\norm{\lVert}{\rVert}
\DeclarePairedDelimiter\floor{\lfloor}{\rfloor}

\usepackage{esint}
\usepackage{algorithm, algorithmicx}
\usepackage{algpseudocode}
\usepackage{makecell}
\allowdisplaybreaks

\usepackage{amsbsy,
	amsopn,
	amscd, 
	amsxtra, 
	amsthm,
	verbatim}
\usepackage{upref}
\usepackage{xcolor}
\usepackage{enumitem}
\usepackage[colorlinks=true, linkcolor=blue, citecolor=magenta, urlcolor=blue]{hyperref}

\usepackage{thmtools}
\usepackage{caption}
\usepackage{subcaption}
\usepackage{booktabs}
\usepackage{longtable}
\usepackage[margin=1in]{geometry}
\usepackage[capitalise]{cleveref}

\newtheorem{lemma}{Lemma}[section]
\newtheorem{theorem}[lemma]{Theorem}
\newtheorem{proposition}[lemma]{Proposition}
\newtheorem*{proposition*}{Proposition}
\newtheorem{corollary}[lemma]{Corollary}
\newtheorem*{corollary*}{Corollary}
\newtheorem{definition}[lemma]{Definition}

\theoremstyle{remark}
\newtheorem{remark}[lemma]{Remark}
\newtheorem{example}[lemma]{Example}

\Crefname{theorem}{Theorem}{Theorems}
\Crefname{proposition}{Proposition}{Propositions}
\Crefname{lemma}{Lemma}{Lemmas}
\Crefname{corollary}{Corollary}{Corollaries}
\Crefname{definition}{Definition}{Definitions}
\Crefname{assumption}{Assumption}{Assumptions}
\Crefname{remark}{Remark}{Remarks}
\Crefname{example}{Example}{Examples}
\crefname{theorem}{Theorem}{Theorems}
\crefname{proposition}{Proposition}{Propositions}
\crefname{lemma}{Lemma}{Lemmas}
\crefname{corollary}{Corollary}{Corollaries}
\crefname{definition}{Definition}{Definitions}
\crefname{assumption}{Assumption}{Assumptions}
\crefname{remark}{Remark}{Remarks}
\crefname{example}{Example}{Examples}
\crefformat{equation}{(#2#1#3)}

\newcommand{\ud}{\,\mathrm{d}}
\newcommand{\R}{\mathbb{R}}

\newcommand{\mmd}{\mathrm{MMD}}

\newcommand{\dm}{|\nabla|}
\newcommand{\nab}{\nabla}

\newcommand{\XN}{X_N}
\newcommand{\Sc}{{\mathcal{S}}}

\newcommand{\vep}{\varepsilon}

\DeclareMathOperator{\supp}{supp}

\newcommand{\indic}{\mathbf{1}}
\let\div\relax
\DeclareMathOperator{\div}{\mathrm{div}}
\DeclareMathOperator{\sgn}{sgn}
\DeclareMathOperator{\Real}{Re}
 
\renewcommand{\P}{\mathcal{P}}

\newcommand{\dist}{\mathrm{dist}}
\newcommand{\diam}{\mathrm{diam}}

\title{Wasserstein gradient flows of Maximum Mean Discrepancy with energy kernels}
\author{Matthew Rosenzweig, Dejan Slep\v{c}ev and Lihan Wang}
\address{Matthew Rosenzweig, Department of Mathematical Sciences, Carnegie Mellon University, 7127 Wean Hall, 5000 Forbes Avenue, Pittsburgh, PA 15213, USA}
\email{mrosenz2@andrew.cmu.edu}
\address{Dejan Slep\v{c}ev, Department of Mathematical Sciences, Carnegie Mellon University, 7123 Wean Hall, 5000 Forbes Avenue, Pittsburgh, PA 15213, USA}
\email{slepcev@math.cmu.edu}
\address{Lihan Wang, Department of Mathematics, National University of Singapore, 10 Lower Kent Ridge Road, Block S17, 08-14, 119076, Singapore}
\email{lihanw@nus.edu.sg}
\thanks{MR is supported by NSF grants DMS-2342349 and DMS-2441170. DS is supported by NSF grants DMS-2342349, DMS-2511684, and DMS-2407166. LW is supported by NSF grant DMS-2407166.}
\date{}

\subjclass[2020]{35Q70, 49Q22, 35B40 (Primary); 35R09,  46E22, 60B10, 82C22 (Secondary)}
\keywords{Maximum mean discrepancy, Wasserstein gradient flow, energy kernels, Riesz potentials, aggregation equation, mean-field limit, modulated energy, interacting particle system, Lagrangian critical point, long-time behavior}

\begin{document}


\begin{abstract}
We study the Wasserstein gradient flow of the squared Maximum Mean Discrepancy (MMD) generated by the nonsmooth energy kernels $K(z)=-|z|^q$, $0<q<2$.  In dimensions $d\ge2$, the corresponding energies are not displacement semiconvex, so standard Wasserstein-gradient-flow theory does not apply.  When $d+q-2>0$, we prove global well-posedness on $\R^d$ for probability densities in subcritical $L^p$ spaces, with targets in the same integrability class and with finite moments. We also include the one-dimensional Coulomb endpoint $d=q=1$.

For the associated $N$-particle system, we prove global noncollision and fixed-$N$ convergence to the collision-free critical set, a particle-to-continuum criticality principle, and a modulated-energy mean-field estimate that yields convergence of the particle dynamics to the continuum flow as $N\to\infty$ on every finite time interval.  We also construct collision-free saddle equilibria, showing that deterministic particle trajectories need not approach global empirical minimizers.

For $1\le q<2$, every continuum solution in our class has a narrowly relatively compact orbit, every $\omega$-limit point is Lagrangian critical, and the orbit approaches the Lagrangian critical set.  For $0<q<1$, the same conclusions hold under uniform-in-time moment and subcritical $L^p$ bounds.  We prove that an absolutely continuous Lagrangian critical point equals the target when the source and target have finite moments of order $q$, except when $0<q<1$ and $d\in\{1,3\}$.  In the residual three-dimensional regime, rigidity holds under a finite $(q+1)$-th moment, compactness of the positive part of the discrepancy, or radiality of the discrepancy, whereas in dimension one explicit nested-interval examples yield non-minimizing critical points.  Under the preceding uniform bounds, rigidity gives convergence of the continuum flow to the target throughout the rigid part of the well-posedness range.

Finally, we show that no initial-data-independent multiplicative MMD decay modulus exists on $\R^d$, and that global Polyak--\L ojasiewicz inequalities fail in several whole-space and periodic Riesz/Coulomb regimes.
\end{abstract}

\maketitle

\section{Introduction}

Maximum mean discrepancy rose to prominence as a practical nonparametric criterion for two-sample testing \cite{gretton2007kernel,gretton2012kernel}.  Its Wasserstein gradient flow formally transports a source distribution toward a prescribed target and provides a dynamical model for statistical machine learning, sampling, and generative learning \cite{arbel2019maximum,altekruger2023neural,hertrich2024generative,hagemann2024posterior}.  In data-driven applications, the target distribution $\mu$ is typically replaced by its empirical distribution $\mu_N$.

Let $\mu$ be a target probability measure on $\R^d$.\footnote{With a slight abuse of notation, we use the same symbols $\mu$ and $\rho$ for absolutely continuous measures and for their Lebesgue densities throughout the paper.}  We study the energy MMD generated by
\begin{equation}
K(z)=-|z|^q,\qquad 0<q<2.
\end{equation}
We refer to this negative-distance family as the Euclidean energy kernels; for $q=1$, the associated MMD is, up to normalization, the classical energy distance \cite{szekely2005new,szekely2003statistics}.
Its centered-kernel definition, moment requirements, and the significance of the threshold $q=2$ are recalled at the beginning of \Cref{sec:background}.  The formal continuum flow is
\begin{equation}\label{eq:wgfmmd}
		\partial_t \rho_t = -\nabla\cdot(\rho_t v_t), \qquad v_t=-\nabla K*(\rho_t-\mu),
\end{equation}
and, writing $[N]\coloneqq\{1,\ldots,N\}$, its (diagonal-free) particle counterpart is
\begin{equation}\label{eq:wgfmmdparticle}
		\dot{x}_i^t = -\frac1N\sum_{j\in[N]\setminus\{i\}} \nab K(x_i^t-x_j^t) + \nab(K\ast\mu)(x_i^t), \qquad i\in[N].
\end{equation}
The missing self-interaction in \eqref{eq:wgfmmdparticle} is part of the model, necessary to allow for singularities in the velocity field. We use $\nabla K(0)=0$ in this paper only as shorthand for this convention.  Since the continuum theory concerns $L^p$ densities while empirical measures are atomic, the two evolutions must be constructed separately.

 For these nonsmooth kernels, standard Wasserstein-gradient-flow and mean-field theories do not apply directly: the interaction force fails to be Lipschitz at the particle diagonal for every $0<q<2$, and in dimensions $d\ge2$ the energies are not displacement semiconvex. 
Our work touches on various aspects of both the continnum flow and the interacting particles and connects them quantitatively.  Our main contributions are in the following five directions.

\begin{enumerate}[label=\textup{\roman*.},leftmargin=*]

\item \textbf{Global well-posedness of the continuum flow.} 
In the range $d+q-2>0$ we prove existence and uniqueness of a global weak solution of \eqref{eq:wgfmmd} for probability densities in subcritical $L^p$ spaces with finite moments (that is, $p>p_c:=d/(d+q-2)$; see \eqref{eq:admissible-p-range}), with a velocity field that is Lipschitz in space uniformly on bounded time intervals, together with the associated energy--dissipation identity.  The one-dimensional Coulomb endpoint $d=q=1$ is included and follows from the stronger Cauchy theory of \cite{ChodronDeCourcelRosenzweigCoulombDiscrepancies}.  See \Cref{thm:wgfmmdwp}.

\item \textbf{Particle dynamics.}  For the system \eqref{eq:wgfmmdparticle} started from any pairwise distinct configuration, we prove global noncollision with uniform-in-time separation and boundedness, and convergence at fixed $N$ to the collision-free critical set of the particle energy (\Cref{thm:particle-global-noncollision}).  We prove a particle-to-continuum criticality principle (\Cref{prop:particle-critical-closure}) and uniform-in-time consistency for well-prepared data (\Cref{cor:well-prepared-particle-uniform-time}).  Explicit collision-free saddle equilibria (\Cref{prop:particle-saddle-configurations}) show that deterministic trajectories need not approach global empirical minimizers.

\item \textbf{Mean-field limit.}  We prove a modulated-energy estimate (\Cref{thm:wgfmmdpoc}): a quantitative comparison, in squared MMD, of the empirical measure of the particle flow with the continuum solution driven by the same target.  A probability measure is approximated by the empirical measure of $N$ independent samples with MMD error of order $N^{-1/2}$, at a rate that does not depend on the dimension.  The estimate therefore yields convergence of the particle dynamics to the continuum flow as $N\to\infty$ on every finite time interval, at the same rate $N^{-1/2}$ for such well-prepared random initial data.

\item \textbf{Stationary states and long-time behavior.}  We obtain a near-complete stationary picture among absolutely continuous states (\Cref{thm:lagrangian-rigidity-flexibility}): at the natural moment level, a Lagrangian critical point (one at which the velocity field $\nabla K*(\rho-\mu)$ vanishes $\rho$-almost everywhere; see \Cref{critical:def:Lcrit}, and \Cref{sec:stationary} for the comparison with the Wasserstein notion of \cite{boufadene2023global}) equals the target except when $0<q<1$ and $d\in\{1,3\}$; in the residual three-dimensional regime, rigidity holds under a finite moment of order $q+1$, compactness of the positive part of the discrepancy, or radiality of the discrepancy about some point; and in dimension one explicit nested-interval examples yield non-minimizing critical points.  Combined with the energy--dissipation identity, this gives asymptotic criticality without additional long-time bounds for $1\le q<2$, and conditional convergence to the target throughout the rigid part of the present well-posedness range (\Cref{thm:conditional-critical-convergence,cor:conditional-energy-convergence}).

\item \textbf{Obstructions to uniform convergence rates.}
We construct explicit examples showing that uniform convergence rates and global Polyak--\L ojasiewicz inequalities cannot hold.  Dynamically, for every $(d,q)$ in the well-posedness range and every admissible target, translating a fixed compactly supported source rules out any multiplicative MMD decay modulus independent of the initial datum (\Cref{thm:noconvmmd}).  Statically, no global PL inequality holds on $\R^d$ for any $0<q<2$ and any compactly supported target, and the same construction covers the negative-order Riesz range $-d<q<0$ (\Cref{prop:nogPL}).  On $\mathbb T^d$, where spatial escape is unavailable, a high-frequency construction gives the same failure for $2-d<q<2$ and every bounded, uniformly positive target (\Cref{prop:noPLsubCo}); this range is sharp, since at the Coulomb endpoint $q=2-d$ a PL inequality can instead hold.  These are statements about uniformity over the whole admissible class, not about any individual solution.  The mechanism is transport at finite speed: a source at distance $R$ from the target is driven by a far-field velocity of order $R^{q-1}$, so it requires a time of order $R^{2-q}$ merely to reach the target before any relaxation can begin, while its squared discrepancy remains of order $R^{q}$ throughout.  What is excluded is a single decay modulus valid simultaneously at every initial separation; rates depending on the initial support scale, or holding after a corresponding waiting time, are not excluded.  \Cref{sec:1dexample} exhibits precisely this behavior.
\end{enumerate}

Finally, in the one-dimensional $q=1$ particle model, we exhibit a two-time-scale phenomenon: particles outside the target support first travel toward it and then relax exponentially, provided the target density is bounded below on its support.  A second example shows that this positivity cannot be dropped: for a target vanishing at an interior point, a particle started inside the support relaxes only polynomially.

We also sharpen the description of the topology metrized by $\mmd_q$: it dominates every transport cost of order $r<q/2$ (\Cref{prop:mmd-topology}), and this threshold is sharp, confirming a conjecture of Modeste and Dombry \cite{modeste2023characterization} that convergence in $\mmd_q$ need not imply convergence in the $q/2$-transport topology (\Cref{ex:MMDtoWass}).

\begin{remark}[The role of the exponent $q$]\label{rem:role-of-q}
The exponent enters in two independent ways, and the two thresholds do not coincide.  Metrically, \Cref{lem:genSob} identifies $\mmd_q$ with the $\dot H^{-(d+q)/2}$ norm, so $q$ tunes the strength of the discrepancy: smaller $q$ gives a stronger norm and finer resolution, while $q=2$ is degenerate, the centered kernel reducing to $k_2(x,y)=x\cdot y$ and the discrepancy detecting only the difference of means.  Analytically, our hypotheses change at $q=1$.  For $1\le q<2$, asymptotic criticality holds with no additional long-time bounds (\Cref{thm:conditional-critical-convergence}), rigidity at the natural moment level holds in every dimension (\Cref{thm:lagrangian-rigidity-flexibility}(i)), the particle-to-continuum criticality principle requires no microscopic nonconcentration hypothesis (\Cref{prop:particle-critical-closure}), and the interaction force is bounded at the diagonal.  For $0<q<1$ each of these acquires an additional hypothesis.  The rigidity exceptions are confined to $d\in\{1,3\}$, so in even dimensions and in odd dimensions $d\ge5$ rigidity holds throughout $0<q<2$.  In applications $q$ is generally chosen on modeling grounds; the above is not a recommendation, but a record of which regimes the present guarantees cover.
\end{remark}

\subsection{Background}\label{sec:background}

For a positive-semidefinite kernel $k$ with reproducing kernel Hilbert space (RKHS) $\mathcal H_k$, the associated maximum mean discrepancy is
\begin{equation}\label{eq:rkhs-ipm-positive-kernel}
\sup_{\|f\|_{\mathcal H_k}\le1}\int f\ud(\rho-\mu)
=\left(\iint k(x,y)\ud(\rho-\mu)(x)\ud(\rho-\mu)(y)\right)^{1/2};
\end{equation}
this is the integral probability metric generated by the unit ball of $\mathcal H_k$; see \cite{muller1997integral}.  For $0<q<2$, the translation-invariant interaction profile $(x,y)\mapsto-|x-y|^q$ is not positive semidefinite but is conditionally positive semidefinite (that is, $-\iint|x-y|^q\ud\nu(x)\ud\nu(y)\ge0$ for every signed measure $\nu$ of total mass zero, such as $\nu=\rho-\mu$).  Centering it gives the positive-semidefinite kernel

\begin{equation}\label{eq:centered-positive-kernel-intro}
k_q(x,y):=\frac12\bigl(|x|^q+|y|^q-|x-y|^q\bigr).
\end{equation}
We define
\begin{align}
\mmd_q^2 (\rho,\mu) & := \frac12 \int_{\R^d} \int_{\R^d}
\bigl(|x|^q+|y|^q-|x-y|^q\bigr)\ud (\rho-\mu)(x) \ud (\rho-\mu)(y)
\label{eq:defmmd}\\
& \, = -\frac12 \int_{\R^d} \int_{\R^d} |x-y|^q \ud (\rho-\mu)(x) \ud (\rho-\mu)(y)
\qquad \text{if } \rho,\mu \in \mathcal{P}_q(\R^d).
\label{def:mmdq}
\end{align}
Here and throughout, $\mathcal L^d$ denotes $d$-dimensional Lebesgue measure.
For $r>0$ and $\lambda\in\mathcal P(\R^d)$, set
\begin{equation}\label{eq:moment-notation}
\mathcal M_r(\lambda):=\int_{\R^d}|x|^r\ud\lambda(x),
\qquad
\mathcal P_r(\R^d):=\{\lambda\in\mathcal P(\R^d):\mathcal M_r(\lambda)<\infty\}.
\end{equation}
The centered expression in \eqref{eq:defmmd} is finite for pairs of measures in $\mathcal P_{q/2}(\R^d)$.  Under the stronger hypothesis $\rho,\mu\in\mathcal P_q(\R^d)$, the translation-invariant representation \eqref{def:mmdq} is available \cite{sejdinovic2013equivalence}.  The well-posedness theory below assumes moments of order $r\ge1$; since $q/2<1$, this ensures finiteness of the centered MMD.  We use the translation-invariant representation only when the $\mathcal P_q$ hypothesis is explicitly available.

Taking $k=k_q$ in \eqref{eq:rkhs-ipm-positive-kernel} recovers the centered definition \eqref{eq:defmmd}.  The test class is then concrete: since $(d+q)/2>d/2$ for every $q>0$, the RKHS of $k_q$ is the homogeneous Sobolev space $\dot H^{(d+q)/2}(\R^d)$ modulo constants, and the constants are immaterial because $\rho-\mu$ has zero mass; see \Cref{rem:rkhs-sobolev}.  The threshold $q=2$ marks the degenerate endpoint of this family: $-|x-y|^2$ remains conditionally positive semidefinite, but $k_2(x,y)=x\cdot y$, so the resulting discrepancy detects only the difference between the means of $\rho$ and $\mu$.  Beyond $q=2$, conditional positive semidefiniteness fails in general.  Thus $0<q<2$ is the nondegenerate negative-distance MMD range; see \cite{sejdinovic2013equivalence}.  The Fourier--Sobolev representation in \Cref{lem:genSob} identifies $\mmd_q$ with a negative homogeneous Sobolev norm.

Several features make these kernels attractive in statistics and computation.  On $\mathcal P_{q/2}(\R^d)$, convergence in $\mmd_q$ implies weak convergence; indeed, it implies convergence in every transport cost of order $0<r<q/2$, as recorded in \Cref{prop:mmd-topology}; see also \cite{simon2023metrizing,sejdinovic2013equivalence}.  The restriction $0<r<q/2$ is sharp: \Cref{ex:MMDtoWass} confirms the conjecture of Modeste and Dombry \cite[Theorem~13 and the paragraph following it]{modeste2023characterization} that MMD convergence need not imply convergence in the $q/2$-transport topology.  For the unbounded energy kernel considered here and $\mu\in\mathcal P_q(\R^d)$, empirical MMD has the dimension-free $N^{-1/2}$ Monte Carlo scaling familiar from standard MMD theory \cite{gretton2007kernel,gretton2012kernel,tolstikhin2017minimax}.  More precisely, \Cref{prop:empirical_mmd_bound} gives an exact formula of order $N^{-1}$ for the expected squared MMD of the empirical measure together with a high-probability consequence.  This contrasts with the dimension-dependent empirical rates for Wasserstein distances \cite{fournier2015rate} and motivates MMD gradient flows and their particle approximations.

Unlike smooth bounded kernels, the energy kernel remains sensitive at large spatial scales: since $|z|^q$ grows at infinity, the discrepancy does not saturate when source and target mass are widely separated.  This feature contributes to the practical interest in these kernels in generative learning and sampling; see \cite{altekruger2023neural,hertrich2024generative,hagemann2024posterior}.

A further computational advantage, observed by Kolouri, Nadjahi, Shahrampour, and {\c{S}}im{\c{s}}ekli \cite{kolouri2022generalized} and exploited by Hertrich, Wald, Altekr\"uger, and Hagemann \cite{hertrich2024generative}, is the sliced identity.  For $\xi\in\mathbb S^{d-1}$, let $P_\xi(x):=\xi\cdot x$, and let $\mathcal U_{\mathbb S^{d-1}}$ denote the uniform probability measure on $\mathbb S^{d-1}$.  Then
\begin{equation}\label{eq:sliced-mmd-intro}
\mmd_q^2(\rho,\mu)=C_{d,q}\,\mathbb E_{\xi\sim\mathcal U_{\mathbb S^{d-1}}}\big[\mmd_q^2((P_\xi)_\#\rho,(P_\xi)_\#\mu)\big].
\end{equation}
In practice, one samples directions $\xi$, projects the source and target samples, and evaluates the resulting one-dimensional discrepancies and gradients.  In the energy-distance case $q=1$, sorting the one-dimensional projections gives $O((M+N)\log(M+N))$ complexity per sampled direction, instead of the naive $O(N^2+MN)$ per-step cost after precomputing the target--target term.

At short scales, however, the force creates an analytical difficulty: although the particle-particle force $-\nabla K(z)=q|z|^{q-2}z$ is repulsive, it fails to be Lipschitz at the particle diagonal throughout $0<q<2$; the force itself is singular only when $0<q<1$.  This distinction is relevant to the diagonal-free particle dynamics and the noncollision result in \Cref{thm:particle-global-noncollision}.

At the level of the continuum flow, the kernels present a second obstruction.
For interaction profiles $K$ such that $(x,y)\mapsto K(x-y)$ is jointly semiconvex in the $x$ and $y$ variables, the standard theory of Wasserstein gradient flows in the space of probability measures \cite{ambrosio2008gradient} provides well-posedness of \eqref{eq:wgfmmd}, allows for general measure-valued solutions including discrete measures, and gives stability estimates that imply mean-field convergence of \eqref{eq:wgfmmdparticle} to \eqref{eq:wgfmmd}.  These tools do not apply directly to the energy kernels above.  In one dimension, the fixed-target MMD functional associated with $K(x-y)=-|x-y|$ is displacement convex, and this flow was analyzed through quantiles by Duong, Stein, Beinert, Hertrich, and Steidl \cite{duong_steidl26}; see also the work of Bonaschi, Carrillo, Di Francesco, and Peletier \cite{bonaschi2015equivalence} who treated the pure source--source interaction energy.  This is an endpoint phenomenon: for every $0<q<1$, the corresponding fixed-target functional is not displacement convex, although its pure source--source interaction term is.  In dimensions $d\ge2$, the energies are likewise not displacement semiconvex, so well-posedness in the full space of probability measures is not available from the standard theory.  The remaining range considered here instead requires the $L^p$ approach described next.

\smallskip

\subsection{Main results}

We state the results in the order of the objects they concern rather than in the order of the contributions listed above: first the continuum flow, through its well-posedness, its stationary picture, and its long-time behavior; then the particle system and the mean-field estimate linking the two; and finally the obstructions to uniform convergence.

We begin with global well-posedness for the continuum flow \eqref{eq:wgfmmd}.  The one-dimensional Coulomb endpoint $d=q=1$ is included for completeness and follows from the stronger Cauchy theory of Chodron de Courcel and the first author \cite{ChodronDeCourcelRosenzweigCoulombDiscrepancies}.  In the remaining range $d+q-2>0$, we work in Lebesgue classes and adapt the aggregation-equation method of Bertozzi--Laurent--Rosado \cite{bertozzi2011lp} to the additional target field $\nabla K*\mu$.

For $d+q-2\ge0$, we set, here and throughout,
\begin{equation}
p_c:=
\begin{cases}
\dfrac{d}{d+q-2},&d+q-2>0,\\[4pt]
\infty,&d+q-2=0.
\end{cases}
\end{equation}
\begin{theorem}[Global well-posedness]\label{thm:wgfmmdwp}\label{thm:wgfmmdwp-linfty}

Let $r\ge1$, let $0<q<2$ satisfy $d+q-2\ge0$, and let
\begin{equation}\label{eq:admissible-p-range}
p_c\le p\le\infty,
\quad \text{and} \quad
p>p_c\quad\text{if }d+q-2>0.
\end{equation}

If $\mu,\rho_0\in\mathcal P_r(\R^d)\cap L^p(\R^d)$ are probability densities, then \eqref{eq:wgfmmd} has a unique global weak solution in the class of \Cref{def:continuum-solution-class}.  For every $T>0$, the $r$-th moment remains bounded on $[0,T]$ and
\begin{equation}\label{eq:intro-velocity-class}
v_t=-\nabla K*(\rho_t-\mu),\qquad
v\in L^\infty([0,T];W^{1,\infty}(\R^d))\cap C([0,T]\times\R^d).
\end{equation}
When $p=\infty$, one additionally has
\begin{equation}\label{eq:linfty-endpoint-bound}
\|\rho_t\|_{L^\infty}
\le \exp \Bigl(t\|(-\Delta K)*\mu\|_{L^\infty}\Bigr)\|\rho_0\|_{L^\infty},
\qquad t\ge0.
\end{equation}
\end{theorem}

\begin{remark}[Well-posedness hypotheses]

Let $p_*$ denote the H\"older conjugate of $p$.  When $d+q-2>0$, condition \eqref{eq:admissible-p-range} is equivalent to $(2-q)p_*<d$; it is exactly the subcritical integrability needed to convolve the Hessian of the kernel $|x|^{q-2}$ with $L^p$ densities.  For the target-free  equation, $p_c$ is the $L^p$-critical exponent for the time-preserving scaling
\begin{equation}
\rho(t,x)\longmapsto \lambda^{d+q-2}\rho(t,\lambda x),
\end{equation}
which leaves the $L^{p_c}$ norm invariant while rescaling the total mass.  The finite critical endpoint $p=p_c<\infty$ remains excluded; we return to the open problem of including it in \Cref{sec:further-questions}.  In this strict dimensional range, the case $p=\infty$ is obtained from any admissible finite exponent and the maximum principle.  If $d+q-2=0$, then necessarily $d=q=1$, and \eqref{eq:admissible-p-range} forces $p=p_c=\infty$; this Coulomb endpoint follows instead from \cite[Theorem~1.2 and Proposition~2.1]{ChodronDeCourcelRosenzweigCoulombDiscrepancies}.

\end{remark}

We turn to the stationary picture: a nearly complete rigidity--flexibility classification for absolutely continuous Lagrangian critical points.

\begin{theorem}[Lagrangian critical-point rigidity and flexibility]\label{thm:lagrangian-rigidity-flexibility}
Let $d\ge1$ and $0<q<2$. The following statements hold.
\begin{enumerate}[label=\textup{(\roman*)}]
\item Let $\mu,\rho\in\mathcal P_q(\R^d)$, assume $\rho\ll\mathcal{L}^d$, and suppose that $\rho$ is a Lagrangian critical point of the MMD energy with target $\mu$ in the sense of \Cref{critical:def:Lcrit}.  Then $\rho=\mu$ in any of the following cases:
\begin{enumerate}[label=\textup{(\alph*)}]
\item $d$ is even;
\item $d$ is odd and $1\le q<2$;
\item $d\ge5$ is odd and $0<q<1$.
\end{enumerate}
\item If $d=3$ and $0<q<1$, under the measure and criticality hypotheses of \textup{(i)}, the conclusion $\rho=\mu$ also holds provided that at least one of the following conditions is satisfied:
\begin{enumerate}[label=\textup{(\alph*)}]
\item $\rho,\mu\in\mathcal P_{q+1}(\R^3)$;
\item the positive part of the signed discrepancy, $(\rho-\mu)^+$, has compact support;
\item $\rho-\mu$ is invariant under rotations about some point.
\end{enumerate}
The compact-positive-part condition holds, in particular, if $\supp\rho$ is compact.
\item If $d=1$ and $0<q<1$, then for every $0<r<R$ there exist distinct compactly supported absolutely continuous probability measures $\rho_r,\mu\in\mathcal P_q(\R)$ satisfying
\begin{equation}
\supp\rho_r=[-r,r]\subsetneq[-R,R]=\supp\mu
\end{equation}
such that $\rho_r$ is a Lagrangian critical point for the target $\mu$, but is not a minimizer.
\end{enumerate}
\end{theorem}

Part~\textup{(i)} follows from \Cref{critical:thm:natural-moment-rigidity}, the three alternatives in part~\textup{(ii)} follow from \Cref{critical:thm:residual-three-rigidity}, and part~\textup{(iii)} is proved in \Cref{critical:thm:nested-intervals}.  Thus any unresolved natural-$q$-moment case in dimension three must have nonradial discrepancy and unbounded support of its positive part.

The energy--dissipation identity for the solution in \Cref{thm:wgfmmdwp} also yields long-time asymptotic criticality.  For $1\le q<2$, every sequence of times tending to infinity has a subsequence along which the flow converges narrowly to a Lagrangian critical point, and the full orbit approaches the Lagrangian critical set, without any additional long-time bound.  For $0<q<1$, the same conclusions hold under uniform moment and $L^p$ bounds.  Under those bounds, if the available moment order meets the corresponding rigidity hypothesis, every $\omega$-limit point equals the target and the full trajectory converges to it.  Within the range of \Cref{thm:wgfmmdwp}, this applies throughout part~\textup{(i)} of \Cref{thm:lagrangian-rigidity-flexibility} and also for $d=3$ and $0<q<1$ either under the additional-moment alternative in part~\textup{(ii)} or when the initial source and target are radial about a common center; see \Cref{thm:conditional-critical-convergence,cor:conditional-energy-convergence}.  These conclusions are qualitative and no rate is asserted.
\medskip

Before stating the mean-field estimate, we establish global noncollision and fixed-$N$ asymptotic criticality for the particle dynamics, including the one-dimensional range $0<q<1$, where uniqueness is not asserted.  We also prove uniform-in-time consistency for well-prepared particle data: if $\mathcal E_N(X_N^0) := \mmd_q(\frac{1}{N}\sum_{i=1}^N \delta_{x^i_0})\to0$, then the particle empirical measure $\rho_{t_N}^N\rightharpoonup\mu$ narrowly for every sequence $t_N\ge0$; see \Cref{cor:well-prepared-particle-uniform-time}.  For particle data with uniformly bounded initial energies, we show that there exist $N_k\to\infty$, $t_k\to\infty$, and a Lagrangian critical point $\bar\rho$ such that $\rho_{t_k}^{N_k}\rightharpoonup\bar\rho$ narrowly.  When $0<q<1$, this conclusion holds under the additional assumption of a uniform discrete Riesz bound with inverse-power exponent $1-q$ for the particle $\omega$-limit sets (no such nonconcentration condition is needed when $1\le q<2$).  The case $0<q<1$ includes $d=1$ under $p>1/q$; see \Cref{prop:particle-critical-closure,cor:particle-late-time-critical-limit}.  We then construct collision-free saddle equilibria showing that convergence to global empirical minimizers cannot hold for every deterministic initial configuration.
\begin{theorem}[Global noncollision \texorpdfstring{and fixed-$N$ asymptotic criticality}{and fixed-N asymptotic criticality}]\label{thm:particle-global-noncollision}\label{thm:particle-asymptotic-criticality}
Let $0<q<2$.  Assume one of the following:

\begin{enumerate}[label=\textup{(\roman*)}]
\item $d+q-2>0$ and $p_c<p\le\infty$;
\item $d=q=1$ and $p=p_c=\infty$;
\item $d=1$, $0<q<1$, and $1/q<p\le\infty$.
\end{enumerate}

Let $\mu\in\mathcal P_{\max\{1,q\}}(\R^d)\cap L^p(\R^d)$ be a probability density.  For $N\ge2$, set
\begin{equation}\label{eq:intro-collision-free-configurations}
\mathcal D_N:=\{(x_i)_{i\in[N]}\in(\R^d)^N: \, x_i\ne x_j \, \text{ for all distinct }i,j\in[N]\}.
\end{equation}
For every $X_N^0\in\mathcal D_N$, in cases \textup{(i)} and \textup{(ii)}, \eqref{eq:wgfmmdparticle} has a unique solution $X_N\in C^1([0,\infty);\mathcal D_N)$.  In case \textup{(iii)}, it has at least one solution in $C^1([0,\infty);\mathcal D_N)$, and every maximal-lifespan $C^1$ solution with values in $\mathcal D_N$ is global.

For the particle energy $\mathcal E_N$ defined in \eqref{eq:particle-energy-definition}, set
\begin{equation}\label{eq:particle-critical-set}
\operatorname{Crit}_{\mathcal D_N}\mathcal E_N
:=
\left\{X_N\in\mathcal D_N:\nabla\mathcal E_N(X_N)=0\right\}.
\end{equation}
Every such global solution satisfies
\begin{equation}\label{eq:particle-uniform-boundedness-separation}
\sup_{t\ge0}\max_{i\in[N]}|x_i^t|<\infty,
\qquad
\inf_{t\ge0}\min_{\substack{i,j\in[N]\\i\ne j}}|x_i^t-x_j^t|>0.
\end{equation}
Moreover, its $\omega$-limit set
\begin{equation}\label{eq:particle-omega-limit}
\omega(X_N)
:=
\bigcap_{T\ge0}\overline{\{X_N^t:t\ge T\}}
\end{equation}
is nonempty, compact, and connected, and
\begin{equation}\label{eq:particle-omega-critical-level}
\omega(X_N)
\subset
\operatorname{Crit}_{\mathcal D_N}\mathcal E_N
\cap\left\{\mathcal E_N=\mathcal E_N^\infty\right\},
\qquad
\mathcal E_N^\infty:=\lim_{t\to\infty}\mathcal E_N(X_N^t).
\end{equation}
Consequently,
\begin{equation}\label{eq:particle-distance-critical-set}
\dist\!\left(X_N^t,\operatorname{Crit}_{\mathcal D_N}\mathcal E_N\right)
\longrightarrow0.
\end{equation}

\end{theorem}

\begin{proposition}[Collision-free saddle configurations]\label{prop:particle-saddle-configurations}
Let $q=1$ and $d\ge2$.  Let $B_1\subset \R^d$ denote the unit ball, and let $\mu$ denote the uniform probability measure on $B_1$.
For the particle energy $\mathcal E_N$ with target $\mu$, the following statements hold.
\begin{enumerate}[label=\textup{(\roman*)}]
\item For every $N\ge4$, there is an explicit configuration $X_N^{\mathrm{line}}\in\operatorname{Crit}_{\mathcal D_N}\mathcal E_N$ such that $D^2\mathcal E_N(X_N^{\mathrm{line}})$ has both a positive and a negative eigenvalue.  The associated empirical measures converge narrowly, as $N\to\infty$, to a nonatomic probability measure supported on a line.
\item For every integer $m\ge3$, there is a configuration $X_{m^2}^{\mathrm{bulk}}=(x_i^{(m),\mathrm{bulk}})_{i\in[m^2]}\in\operatorname{Crit}_{\mathcal D_{m^2}}\mathcal E_{m^2}$ such that $D^2\mathcal E_{m^2}(X_{m^2}^{\mathrm{bulk}})$ has both a positive and a negative eigenvalue, and
\begin{equation}\label{eq:bulk-saddle-empirical-convergence}
\frac1{m^2}\sum_{i=1}^{m^2}\delta_{x_i^{(m),\mathrm{bulk}}}
\rightharpoonup\mu
\qquad\text{as }m\to\infty.
\end{equation}
\end{enumerate}
\end{proposition}

\begin{remark}[Uniqueness for regular targets]
In case~\textup{(iii)} of \Cref{thm:particle-global-noncollision}, uniqueness is not asserted because the assumption $p>1/q$ guarantees continuity, but not an Osgood modulus of continuity, for $K'*\mu$.  If, in addition, $\mu\in C_b^\alpha(\R)$ for some $1-q\le\alpha<1$, then uniqueness does hold.  Indeed, the standard H\"older--Zygmund estimates for fractional potentials \cite[Chapter~V]{Stein1970singular} show that $K'*\mu$ is locally Lipschitz when $\alpha>1-q$, while for $\alpha=1-q$ it belongs to the Zygmund class and is therefore locally log-Lipschitz.  Since the pair-interaction field is smooth on $\mathcal D_N$, the full particle vector field has an Osgood modulus on compact subsets of $\mathcal D_N$.  The Osgood uniqueness criterion \cite[Problem~2.25]{Teschl2012ODE} gives uniqueness up to the first collision, and the noncollision argument in the proof below then yields a unique global solution.
\end{remark}

Because the standard Dobrushin--Sznitman coupling argument relies on a uniform Lipschitz bound for the interaction force, it does not apply directly here, where $|\nabla^{\otimes 2}K(x)|\sim |x|^{q-2}$ at the particle diagonal \cite{dobrushin,sznitman1991topics}.  Adapted Wasserstein methods can nevertheless treat certain singular interactions: for divergence-free Euler-type kernels, Hauray \cite{hauray_wasserstein_2009} obtains a quantitative $W_\infty$ estimate under a coupled initial $W_\infty$-approximation and minimum-separation condition.  Those structural and configuration hypotheses do not directly furnish the deterministic MMD estimate proved here for arbitrary pairwise-distinct configurations.  The relative-entropy framework of Jabin--Wang \cite{jabin2018quantitative} is likewise not directly applicable: it compares an $N$-particle law with a tensorized continuum law rather than a single deterministic empirical trajectory, and its main estimates are formulated in the periodic setting.  The whole-space extensions discussed there do not cover the growth $|\nabla K(x)|\sim |x|^{q-1}$ of the interaction field at infinity when $q>1$.

We instead follow the modulated-energy method introduced by Duerinckx and Serfaty \cite{duerinckx_mean-field_2016,Serfaty2020} and further developed by Q. H. Nguyen, Serfaty, and the first author \cite{NRS2021}.  This latter work is closest to the present setting, and the method entails differentiating the modulated MMD energy and controlling the resulting transport commutator through the homogeneous Sobolev representation of MMD.  For overviews of the method and its commutator formulation, see \cite{SerfatyLN,Rosenzweig2026Commutators}.  Ultimately, we obtain the following mean-field estimate.

\begin{theorem}[Mean-field MMD estimate]\label{thm:wgfmmdpoc}\label{cor:wgfmmdpoc-linfty}
Assume the hypotheses of \Cref{thm:wgfmmdwp} with $r\ge\max\{1,q\}$.  Fix $T>0$ and $N\ge2$.  Let $\rho_t$ be the corresponding continuum solution, let $X_N^t$ be the particle solution from \Cref{thm:particle-global-noncollision} with pairwise distinct initial configuration, and set
\begin{equation}\label{eq:intro-empirical-error}
\rho_t^N:=\frac1N\sum_{i=1}^N\delta_{x_i^t}.
\end{equation}
Then, for $0\le t\le T$,
\begin{equation}\label{eq:main}
\mmd_q^2(\rho_t^N,\rho_t)
\le \exp\left(C_{d,q}\int_0^t\left(\|\nabla v_s\|_{L^\infty}+\|\dm^{\frac{d+q}{2}}v_s\|_{L^{\frac{2d}{d+q-2}}}\indic_{d+q>2}\right)\ud s\right)
\mmd_q^2(\rho_0^N,\rho_0).
\end{equation}
\end{theorem}

\begin{remark}[One-dimensional super-Coulomb range]
The restriction above excludes the range $d=1$ and $0<q<1$ only in the presence of a nonzero target.  When $\mu=0$, the mean-field conclusion continues to hold by combining the preceding commutator estimate with the local classical well-posedness theory of Choi--Jeong \cite{ChoiJeong2021FractionalPorousMedium}; foundational weak and regularity theories for the target-free fractional porous medium equation were developed in \cite{CaffarelliVazquez2011FractionalPressure,CaffarelliSoriaVazquez2013Regularity,BilerImbertKarch2015NonlocalPorousMedium}.

These target-free results also suggest an extension to a nonzero sufficiently regular target.  Indeed, writing $f_t:=\rho_t-\mu$, one has, in the finite-part sense,
\begin{equation}
\partial_x v_t(x)=q(1-q)\int_{\R}\frac{f_t(x)-f_t(y)}{|x-y|^{2-q}}\ud y.
\end{equation}
Consequently, $f_t\in C^\alpha(\R)$ with $\alpha>1-q$ yields the $L^\infty$ bound on $\partial_x v_t$ required by \Cref{prop:comm}.  Treating the MMD flow in this range would therefore require propagation of such regularity for the target-dependent equation.  We do not pursue this extension here.
\end{remark}

\begin{remark}
If $(x_i^0)_{i\in[N]}$ are independent with law $\rho_0\in\mathcal P_q(\R^d)$, then
\begin{equation}\label{eq:intro-monte-carlo}
\mathbb E\big[\mmd_q^2(\rho_0^N,\rho_0)\big]
=-\frac1{2N}\int_{(\R^d)^2}K(x-y)\ud\rho_0(x)\ud\rho_0(y)=O(N^{-1}).
\end{equation}
The sampled points are distinct almost surely, so \Cref{thm:particle-global-noncollision,thm:wgfmmdpoc} propagate this scaling on every fixed time interval; the velocity factor is finite by \Cref{lem:mf-velocity-factor}.
\end{remark}

A central long-time question is whether $\rho_t\to\mu$ and at what rate.  For $1\le q<2$, \Cref{thm:conditional-critical-convergence} gives qualitative convergence to the Lagrangian critical point set without additional long-time bounds, but it does not identify a possibly singular limit point with the target.  The gap is a single estimate: under the uniform-in-time bound \eqref{eq:conditional-uniform-bounds}, every $\omega$-limit point is an $L^p$ density with finite $r$-th moment, and rigidity then identifies it with $\mu$ (\Cref{cor:conditional-energy-convergence}).  Whether that bound propagates along the flow is the first question raised in \Cref{sec:further-questions}.  When $0<q<1$, asymptotic criticality holds under uniform moment and $L^p$ bounds.  We also show that, for every target covered by the well-posedness theorem, no multiplicative MMD decay modulus can be uniform over the full admissible class of initial data.
\begin{proposition}[No universal global MMD convergence rate]\label{thm:noconvmmd}
Let $r\ge1$, let $0<q<2$ satisfy $d+q-2\ge0$, let $p$ satisfy \eqref{eq:admissible-p-range}, and let $\mu\in\mathcal P_r(\R^d)\cap L^p(\R^d)$ be a probability density.  There is no function $\gamma:[0,\infty)\to(0,\infty)$, independent of the initial datum and satisfying $\gamma(t)\to0$, such that, for every probability density $\rho_0\in\mathcal P_r(\R^d)\cap L^p(\R^d)$, the corresponding solution satisfies
\begin{equation}\label{eq:impossibledecay}
\mmd_q^2(\rho_t,\mu)\le\gamma(t)\mmd_q^2(\rho_0,\mu),\qquad t\ge0.
\end{equation}
\end{proposition}
The obstruction already appears along translations of a fixed compactly supported probability density: the solution remains in a far-field neighborhood of its initial location for a time $T_R\to\infty$, while its squared MMD from the fixed target remains comparable to $R^q$.\par
It therefore rules out rates uniform over the initial support scale, but does not rule out qualitative convergence or estimates after a datum-dependent waiting time.  Consequently, the standard global Polyak--\L ojasiewicz inequality
\begin{equation}\label{eq:intro-global-PL}
\mmd_q^2(\rho,\mu)\le C\int_{\R^d}|\nabla K*(\rho-\mu)|^2\ud\rho
\end{equation}
cannot hold uniformly on this class.

The remaining results of this paper describe other mechanisms that can and cannot drive convergence.  \Cref{sec:stationary} develops the rigidity mechanisms and the nested-interval construction underlying \Cref{thm:lagrangian-rigidity-flexibility}, and then proves asymptotic criticality, unconditionally for $1\le q<2$ and under uniform bounds for $0<q<1$, together with conditional convergence to the target in the rigid part of the present well-posedness range.  \Cref{sec:PL} proves further whole-space and periodic failures of global PL inequalities.  Further metric and completion results for $\mmd_q$ are collected in Appendix~\ref{app:completion-proof}.


\subsection{\texorpdfstring{Prior and companion work}{Prior and companion work}}

For smooth kernels, Arbel, Korba, Salim, and Gretton \cite{arbel2019maximum} proved well-posedness and mean-field approximation, and discussed conditions under which long-time convergence can hold, while Chen, Mustafi, Glaser, Korba, Gretton, and Sriperumbudur \cite{chen2024regularized} studied a regularized variant.  Galashov, De Bortoli, and Gretton \cite{galashov2024deep} developed learned-kernel MMD flows.  Related Wasserstein mean-field dynamics arise in neural-network training \cite{mei2019mean,rotskoff2018neural,sirignano2018meanfield,chizat2018global,takakura2024kernel}.

For energy kernels, Kolouri, Nadjahi, Shahrampour, and {\c{S}}im{\c{s}}ekli \cite{kolouri2022generalized} and Hertrich, Wald, Altekr\"uger, and Hagemann \cite{hertrich2024generative} developed the sliced approach reviewed above.  Hertrich, Gr\"af, Beinert, and Steidl \cite{hertrich2024wasserstein} studied Jordan--Kinderlehrer--Otto (JKO) \cite{jordan1998variational} steps for Riesz discrepancies, and Duong, Stein, Beinert, Hertrich, and Steidl \cite{duong_steidl26} analyzed the one-dimensional distance-kernel flow through quantiles.  The theoretical difficulties caused by the nonsmoothness of the kernel at the diagonal led Rux, Quellmalz, and Steidl \cite{ruxQuellmalzSteidl2026Smoothed} to mollify the negative distance kernel to make it Lipschitz differentiable, and therefore accessible to the standard Wasserstein-gradient-flow theory, while retaining conditional positive semidefiniteness, nearly linear growth at infinity, and the slicing structure.  The present work is complementary: rather than regularizing the kernel, we establish the corresponding properties of the flow for the nonsmooth energy kernels themselves.

Boufad\`ene and Vialard \cite{boufadene2023global} study the Wasserstein gradient flow of the Coulomb discrepancy and emphasize that the Coulomb energy is not geodesically convex in Wasserstein geometry, so global convergence cannot be read off from the standard convex theory.  Their notions of Wasserstein and Lagrangian criticality, in particular their Definitions~3.1--3.2 and Theorem~3.4, are used below as a reference point for our stationary-state discussion in \Cref{sec:stationary}.  Building on their odd-dimensional energy-distance rigidity result, we establish rigidity for absolutely continuous Lagrangian critical points in every natural-moment regime except $d\in\{1,3\}$ and $0<q<1$.  The three-dimensional exception is covered under an additional moment hypothesis, when the positive part is compactly supported, or when the discrepancy is radial, while the one-dimensional nested-interval construction shows that uniqueness fails.  Our PL counterexamples are complementary: uniqueness of absolutely continuous critical points does not by itself furnish a global functional inequality or a uniform convergence rate.

A closely related work by Chizat, Colombo, Colombo, and Fern\'andez-Real
\cite{chizat2026quantitativeconvergencewassersteingradient} studied Wasserstein gradient flows of kernel mean discrepancies on the torus for Sobolev/Riesz kernels,
\begin{equation}\label{eq:companion-torus-energy}
\mathscr E_s^\nu(\rho)=\frac12\|\rho-\nu\|_{\dot H^{-s}(\mathbb T^d)}^2.
\end{equation}
Under the identification \(s=(d+q)/2\), the negative-distance kernel considered here corresponds to the range \(s\in(d/2,d/2+1)\).  Their well-posedness theory is formulated in a critical Lorentz class on \(\mathbb T^d\), whereas \Cref{thm:wgfmmdwp} gives global Euclidean solutions in finite subcritical \(L^p\) classes and at the endpoint \(L^\infty\), with moment assumptions used to control the target field at infinity.  Their main quantitative results concern convergence to the target on compact geometries: exponential convergence in the Coulomb case \(s=1\) and local polynomial convergence for \(s>1\) under positivity, regularity, and small-discrepancy assumptions.  
  The two sets of results meet exactly at \(s=1\).  Indeed, \(s=1\) is the condition \(d+q-2=0\), so their Coulomb case meets the present family only at the endpoint \(d=q=1\), whereas \(s>1\) is exactly the strict well-posedness range \(d+q-2>0\) of \Cref{thm:wgfmmdwp}, and exactly the range \(2-d<q<2\) of \Cref{prop:noPLsubCo}.  This explains why their exponential rate does not extend past the endpoint: for every \(s>1\), which is precisely the regime in which they obtain only local polynomial convergence, \Cref{prop:noPLsubCo} rules out a global PL inequality on \(\mathbb T^d\), so the standard route to exponential decay is unavailable there.  On \(\mathbb R^d\), \Cref{thm:noconvmmd,prop:nogPL} likewise exclude any initial-data-independent MMD decay modulus and any global PL inequality.  In addition to these continuum questions, the present paper treats diagonal-free particle dynamics; the noncollision and modulated-energy propagation estimates in \Cref{thm:particle-global-noncollision,thm:wgfmmdpoc} are separate from the torus convergence analysis in \cite{chizat2026quantitativeconvergencewassersteingradient}.

Another closely related work, by Chodron de Courcel and the first author \cite{ChodronDeCourcelRosenzweigCoulombDiscrepancies}, studies Wasserstein gradient flows for the Coulomb discrepancy on $\R^d$ and $\mathbb T^d$.  It establishes global weak solutions from arbitrary Borel initial probability measures for bounded target densities, together with instantaneous $L^\infty$ regularization.  On the torus, it proves a global PL inequality for the near-uniform targets and obtains exponential decay of the squared MMD for arbitrary bounded, uniformly positive targets through a defective PL inequality that allows vacuum in the evolving density.  In the Euclidean setting, it establishes radial PL coercivity and exponential convergence under source-support inclusion and target-positivity assumptions, while dynamical persistence and static constructions rule out any initial-data-uniform multiplicative decay modulus or global PL inequality on the unrestricted whole-space class.  It also proves rigidity of Lagrangian critical points whenever $(\rho-\mu)^+$ is absolutely continuous and, in dimension two, qualitative convergence to the target in narrow, weak-$*$ $L^\infty$, and negative-Sobolev topologies.  These Coulomb-specific results overlap with the present work only for $d=q=1$, and their Cauchy theory supplies that endpoint case in \Cref{thm:wgfmmdwp}. 

We also take this opportunity to advertise two closely related works in preparation.  Together with the preceding Coulomb paper, these projects address overlapping but complementary aspects of the questions considered here.  The two works in preparation are not used in the proofs of the main results of the present paper.

The first, by Chodron de Courcel and the first author, studies the unconfined, target-free quadratic Wasserstein gradient flow of sub-Coulomb Riesz interaction energies in the range $-2<s<d-2$ \cite{ChodronDeCourcelRosenzweig2026SubCoulomb}.  This work constructs canonical global weak solutions in the scaling-critical and subcritical Lebesgue classes, proves uniqueness in the corresponding Eulerian class and monotonicity of lower Lebesgue norms, and shows that higher Lebesgue integrability is not created.  This absence of any hypercontractivity is in stark contrast to the Coulomb/super-Coulomb case. It also develops a radial measure-valued well-posedness theory on annuli.  The aforementioned results also apply to the confined setting, including the MMD target case, under suitable assumptions on the confinement.
The main contrast with the companion Coulomb case is that the sub-Coulomb evolution is transport-dominated and preserves, rather than improves, Lebesgue singularities.

The second, joint work of the first author with Hess-Childs and Serfaty, studies the optimal empirical quantization problem for homogeneous and screened Riesz MMDs and diagonal-excluded modulated energies \cite{HessChildsRosenzweigSerfaty2026Quantization}.  Given a target probability measure $\mu$ and a sample size $N$, the empirical quantization problem asks for points $(x_i)_{i\in[N]}$ minimizing the prescribed discrepancy or energy between $\mu$ and the equal-weight empirical measure $\frac{1}{N}\sum_{i=1}^N\delta_{x_i}$; see \cite{GrafLuschgy2000} for the classical theory.  The work \cite{HessChildsRosenzweigSerfaty2026Quantization} proves fixed-cardinality existence throughout the homogeneous Riesz and logarithmic range, upper and lower bounds for the quantization rate which match in their dependence on the number of samples, and separation estimates for empirical minimizers.  It also treats a larger class of (screened) Riesz-type kernels.  This work is a static/equilibrium complement to the present dynamical analysis.  We mention that, for energy kernels, overlapping results were recently obtained independently by Colasanto, Focardi, Fornasier, and Mattesini \cite{colasantoSharpRatesMMD2026}.


\subsection{Organization of the paper}\label{ssec:intro-organization}
\Cref{sec:mmd} gives the Fourier--Sobolev representation used throughout the paper.  \Cref{sec:LPwp} proves the unified continuum well-posedness theorem and the energy--dissipation identity.  \Cref{sec:meanfield} establishes particle noncollision, fixed-$N$ asymptotic criticality, particle-to-continuum criticality with discrete nonconcentration required only when $0<q<1$, collision-free saddle equilibria, and the mean-field estimate.  \Cref{sec:stationary} proves natural-moment rigidity across the energy-kernel family, treats compact-positive-part and radial rigidity in the residual regime $d=3$, $0<q<1$, gives one-dimensional counterexamples, and then proves asymptotic criticality and conditional target convergence.  \Cref{sec:PL,sec:LT} give obstructions to global convergence, while \Cref{sec:1dexample} gives explicit one-dimensional particle dynamics.  Further questions appear in \Cref{sec:further-questions}; Appendix~\ref{app:completion-proof} contains the auxiliary MMD results cited above.

\subsection*{Acknowledgments}\leavevmode\par
The authors thank Antonin Chodron de Courcel, Borjan Geshkovski, Sylvia Serfaty, and Fran{\c{c}}ois-Xavier Vialard for stimulating discussions regarding this work.

\subsection*{Use of AI}\leavevmode\par
The authors used generative-AI tools (OpenAI’s ChatGPT and Anthropic’s Claude) during the development and preparation of this paper to identify potentially relevant literature, explore possible proof strategies, test draft arguments for errors or gaps, improve exposition, and check internal consistency and cross-references. Outputs from these tools were treated as unverified suggestions. The authors followed up on literature suggestions by consulting relevant sources and worked through mathematical suggestions before deciding whether to incorporate them into the paper. Recognizing that these tools may suggest ideas without reliably indicating their provenance, the authors made particular efforts to identify and cite relevant antecedents in the literature. The authors made all final decisions concerning the manuscript and take full responsibility for its contents.

\section{Fourier--Sobolev representation of the energy MMD}\label{sec:mmd}
For a finite Borel measure $\mu$ on $\R^d$, we use the Fourier transform convention
\begin{equation}\label{eq:ft}
\widehat\mu(\xi):=\int_{\R^d}e^{-2\pi i x\cdot\xi}\ud\mu(x).
\end{equation}
The identity
\begin{equation}\label{eq:fourierkernel}
|z|^q=c_{d,q}\int_{\R^d}\frac{1-\cos(2\pi\xi\cdot z)}{|2\pi\xi|^{d+q}}\ud\xi,
\qquad 0<q<2,
\end{equation}
is the Schoenberg--L\'evy representation; see, for example, \cite[(3.2) and (3.12)]{dNPV2012}.

For two probability measures, their difference has zero total mass, so \eqref{eq:fourierkernel} converts the centered energy MMD into a positive weighted $L^2$-distance between their Fourier transforms. This is the Fourier--Sobolev representation recorded in the following lemma.

\begin{lemma}[Fourier--Sobolev representation of the energy MMD]\label{lem:genSob}
Let $q\in(0,2)$ and $\mu_1,\mu_2\in\mathcal P_{q/2}(\R^d)$.  With $\mmd_q$ defined by the centered kernel in \eqref{eq:defmmd},
\begin{equation}\label{eq:genSob}
\mmd_q^2(\mu_1,\mu_2)
=C_{d,q}\int_{\R^d}|2\pi\xi|^{-(d+q)}|\widehat\mu_1(\xi)-\widehat\mu_2(\xi)|^2\ud\xi
=C_{d,q}\|\mu_1-\mu_2\|_{\dot H^{-\frac{d+q}{2}}}^2.
\end{equation}
If $\mu_1,\mu_2\in\mathcal P_q(\R^d)$, this is equivalently the Fourier representation of the uncentered energy-kernel form \eqref{def:mmdq}.
\end{lemma}

\begin{proof}
Set
\begin{equation}\label{eq:feature-space-main}
H:=L^2(\R^d,|2\pi\xi|^{-(d+q)}\ud\xi;\mathbb C),
\qquad \Phi_x(\xi):=e^{-2\pi i\xi\cdot x}-1.
\end{equation}
By \eqref{eq:fourierkernel}, $\|\Phi_x-\Phi_y\|_H^2=C_{d,q}|x-y|^q$ after changing the positive dimensional constant.  Taking $y=0$ and using $\Phi_0=0$, we obtain $\|\Phi_x\|_H=C_{d,q}^{1/2}|x|^{q/2}$.  Thus the assumption $\mu_j\in\mathcal P_{q/2}(\R^d)$ ensures that $M_j:=\int\Phi_x\ud\mu_j(x)$ is a well-defined Bochner integral in $H$.  If $\nu:=\mu_1-\mu_2$, the centered kernel identity and bilinearity give
\begin{equation}\label{eq:feature-energy-main}
\mmd_q^2(\mu_1,\mu_2)=C_{d,q}\left\|\int_{\R^d}\Phi_x\ud\nu(x)\right\|_H^2.
\end{equation}
Since $\nu(\R^d)=0$, the Bochner--Fubini theorem shows that the $H$-valued integral in \eqref{eq:feature-energy-main} has a representative equal to $\widehat\nu(\xi)$ for $\mathcal L^d$-a.e. $\xi$.  Since $\widehat\nu$ is continuous, we use it as the continuous representative of the Bochner integral.  This proves \eqref{eq:genSob}.  For measures in $\mathcal P_q$, the one-body terms in \eqref{eq:defmmd} vanish against $\nu\otimes\nu$.
\end{proof}

\begin{remark}[RKHS interpretation]\label{rem:rkhs-sobolev}
The feature-space construction in the proof identifies the RKHS of $k_q$ with an anchored realization of $\dot H^{(d+q)/2}(\R^d)$, equivalently with the homogeneous Sobolev space modulo constants.
\end{remark}

\begin{remark}[Fourier regularity and integer moments]\label{lem:mmd-PsFourier}
If $\int|x|^s\ud\mu(x)<\infty$, then for every multi-index $\beta$ and $0<\gamma\le1$ with $|\beta|+\gamma\le s$,
\begin{equation}\label{eq:HolderFourier}
[\partial^\beta\widehat\mu]_{C^{0,\gamma}}
\le C_{\beta,\gamma}\int_{\R^d}|x|^{|\beta|+\gamma}\ud\mu(x).
\end{equation}
For the sharp noninteger characterization in terms of an integral modulus of the highest Fourier derivative, see Cho \cite{Cho}.  At an integer order $m\ge1$, bounded $m$-th moments do not give a uniform modulus for the $m$-th derivative.  Indeed, in one dimension, for
\begin{equation}\label{eq:oscillatory-fourier-measures}
\mu_n=(1-n^{-m})\delta_0+n^{-m}\delta_n,
\end{equation}
one has $\mathcal M_m(\mu_n)=1$ and
\begin{equation}\label{eq:oscillatory-fourier-derivative}
\widehat\mu_n^{(m)}(\xi)=(-2\pi i)^m e^{-2\pi i n\xi},\qquad
\left|\widehat\mu_n^{(m)}\left(\frac1{2n}\right)-\widehat\mu_n^{(m)}(0)\right|=2(2\pi)^m.
\end{equation}
Thus oscillations occur on scales tending to zero despite the uniform moment bound.
\end{remark}

The moment, transport, sampling, topology, and completion properties not needed before the PDE analysis are collected in Appendix~\ref{app:completion-proof}.

\section{Well-posedness and properties of solutions}\label{sec:LPwp}
We prove \Cref{thm:wgfmmdwp}.  We first reduce the endpoint $d=q=1$ to the Coulomb Cauchy theory of \cite{ChodronDeCourcelRosenzweigCoulombDiscrepancies}.  In the remaining range $d+q-2>0$, finite-$p$ solutions are constructed by adapting the aggregation-equation strategy of \cite{bertozzi2011lp,laurent2007local}; the case $p=\infty$ is then obtained by applying the finite-$p$ theory at an auxiliary exponent and propagating an $L^\infty$ bound along characteristics.

\begin{definition}[Continuum solution class]\label{def:continuum-solution-class}
Fix $T>0$, let $d\ge1$ and $r\ge1$, let $0<q<2$ satisfy $d+q-2\ge0$, and let $p$ satisfy \eqref{eq:admissible-p-range}.  Let $\rho_0,\mu\in\mathcal P_r(\R^d)\cap L^p(\R^d)$ be probability densities.  A weak solution of \eqref{eq:wgfmmd} on $[0,T]$ has velocity
\begin{equation}\label{eq:continuum-solution-velocity}
v_t=-\nabla K*(\rho_t-\mu),\qquad
v\in L^\infty([0,T];W^{1,\infty}(\R^d))\cap C([0,T]\times\R^d),
\end{equation}
satisfies
\begin{equation}\label{eq:wp-weak-formulation-intro}
\int_0^T\int_{\R^d}(\partial_t\varphi+v_t\cdot\nabla\varphi)\rho_t\ud x\ud t
+\int_{\R^d}\varphi(0,x)\rho_0(x)\ud x=0
\end{equation}
for every $\varphi\in C_c^\infty([0,T)\times\R^d)$, and has $\sup_{0\le t\le T}\mathcal M_r(\rho_t)<\infty$.  When $p<\infty$, we require
\begin{equation}\label{eq:wp-solution-class}
\rho\in C([0,T];\mathcal P_1(\R^d))\cap C([0,T];L^p(\R^d))\cap C^1([0,T];W^{-1,p}(\R^d)),
\end{equation}
where $\mathcal P_1$ carries the $W_1$ topology.  When $p=\infty$, we instead require
\begin{equation}\label{eq:wp-linfty-solution-class}
\rho\in C([0,T];\mathcal P_1(\R^d))\cap C_{w^*}([0,T];L^\infty(\R^d))
\cap\bigcap_{1\le s<\infty}C([0,T];L^s(\R^d)),
\end{equation}
together with $\rho\in L^\infty((0,T)\times\R^d)$.
\end{definition}

Suppose $d=q=1$.  Let $\mathsf g(x):=-|x|/2$ be the one-dimensional Coulomb kernel, so that $K=2\mathsf g$.  After the corresponding linear rescaling of time, global existence, uniqueness in the bounded-density class, and the $L^\infty$ estimate follow from \cite[Theorem~1.2 and Proposition~2.1]{ChodronDeCourcelRosenzweigCoulombDiscrepancies}.  For the resulting solution,
\begin{equation}
|v_t|\le2,
\qquad
\partial_x v_t=2(\rho_t-\mu)
\quad\text{in }\mathcal D'(\R).
\end{equation}
Thus the velocity is bounded and spatially Lipschitz on every finite time interval.  Its flow transports $\rho_0$ to $\rho_t$ and gives the temporal continuity properties in \eqref{eq:wp-linfty-solution-class}, while
\begin{equation}
\mathcal M_r(\rho_t)^{1/r}
\le \mathcal M_r(\rho_0)^{1/r}+2t
\end{equation}
gives the required moment bound.  Moreover, $(-\Delta K)*\mu=2\mu$, and the cited $L^\infty$ estimate implies \eqref{eq:linfty-endpoint-bound}.  Hence the cited solution belongs to the class of \Cref{def:continuum-solution-class} and proves \Cref{thm:wgfmmdwp} when $d=q=1$.  For the remainder of this section, assume $d+q-2>0$.

For finite $p$, the proof has four steps:
\begin{enumerate}[label=\textup{(\arabic*)}]
\item solve a smooth regularization;
\item derive estimates uniform in the regularization parameter;
\item pass to the singular limit and identify the velocity;
\item prove stability and uniqueness.
\end{enumerate}
Throughout the four-step argument below, $p<\infty$.  We then treat the case $p=\infty$ in the strict range $d+q-2>0$ by refining the finite-$p$ solution.  We conclude with a remark on a Coulomb-specific obstruction to uniform higher-regularity bounds.

Set $p_*=p/(p-1)$ for finite $p$ and $p_*=1$ for $p=\infty$.  Under the standing assumption $d+q-2>0$, the condition $p>p_c$ is equivalent to $(2-q)p_*<d$, so
\begin{equation}\label{eq:kernel-local-integrability}
\nabla K,\nabla^{\otimes 2}K\in L^{p_*}_{\mathrm{loc}}(\R^d),\qquad\text{and hence }\Delta K\in L^{p_*}_{\mathrm{loc}}(\R^d).
\end{equation}
The following elementary estimate will be used repeatedly.

\begin{lemma}[Singular-convolution bound]\label{lem:singular-convolution-bound}
Let $1<p\le \infty$ and $0<a<d/p_*$.  Put $\theta:=a p_*/d$.  Let $\eta_\vep$ be a nonnegative unit-mass mollifier for $\vep>0$, and set $\eta_0:=\delta_0$.  Then for every $f\in L^1(\R^d)\cap L^p(\R^d)$ and every $\vep\ge0$,
\begin{equation}\label{eq:singintbd-mollified}
\sup_{x\in\R^d}\int_{\R^d}(\eta_\vep*|\cdot|^{-a})(x-y)|f(y)|\ud y
\le C_{d,a,p}\|f\|_{L^1}^{1-\theta}\|f\|_{L^p}^{\theta},
\end{equation}
with the constant independent of $\vep$.
\end{lemma}

\begin{proof}
For $\vep\ge0$, set $f_\vep:=\eta_\vep*|f|$, with the convention $f_0=|f|$.  By associativity of convolution and Tonelli's theorem,
\begin{equation}
\int_{\R^d}(\eta_\vep*|\cdot|^{-a})(x-y)|f(y)|\ud y
=\int_{\R^d}|x-z|^{-a}f_\vep(z)\ud z.
\end{equation}
Young's inequality gives $\|f_\vep\|_{L^1}=\|f\|_{L^1}$ and $\|f_\vep\|_{L^p}\le\|f\|_{L^p}$.  Hence, for every $x\in\R^d$ and $R>0$, H\"older's inequality on $B_R(x)$ and the trivial far-field bound give
\begin{align}
\int_{\R^d}|x-z|^{-a}f_\vep(z)\ud z
&\le \|f_\vep\|_{L^p}\left(\int_{|z|\le R}|z|^{-ap_*}\ud z\right)^{1/p_*}+R^{-a}\|f_\vep\|_{L^1}\notag\\
&\le C_{d,a,p}R^{d/p_*-a}\|f\|_{L^p}+R^{-a}\|f\|_{L^1}.\label{eq:singular-convolution-split}
\end{align}
The same formula covers $p=\infty$.  If $f\ne0$, choosing
\begin{equation}\label{eq:singular-convolution-optimal-radius}
R=\left(\frac{\|f\|_{L^1}}{\|f\|_{L^p}}\right)^{p_*/d}
\end{equation}
balances the two terms and proves \eqref{eq:singintbd-mollified}, uniformly in $\vep$; the case $f=0$ is immediate.
\end{proof}

In particular, with $a=2-q$, the assumption $p>p_c$ gives the following Lipschitz bound for the velocity field $v=-\nabla K\ast (\rho-\mu)$:
\begin{align}\label{eq:velbd}
	\|\nab v\|_{L^\infty} \lesssim \sup_{x\in\R^d} \int_{\R^d} |x-y|^{q-2}\ud |\rho_t-\mu|(y) \lesssim \|\rho_t-\mu\|_{L^1}^{1-\frac{(2-q)p_*}{d}} \|\rho_t-\mu\|_{L^p}^{\frac{(2-q)p_*}{d}}.
\end{align}

Fix a nonnegative mollifier $\eta \in C_c^\infty(B_1)$ with $\int_{\R^d}\eta(x)\ud x=1$, and set
$\eta_\vep(x)=\vep^{-d}\eta(x/\vep)$. We regularize both the kernel and the data by
\begin{equation}
K_\vep = K*\eta_\vep, \qquad \rho_{0,\vep}=\rho_0*\eta_\vep, \qquad \mu_\vep = \mu*\eta_\vep.
\end{equation}
Since $\rho_0,\mu \in \P_r(\R^d)\cap L^p(\R^d)$, the functions
$\rho_{0,\vep},\mu_\vep \in C_b^\infty(\R^d)\cap \P_r(\R^d)$. More generally, let $f\in L^{s_2}(\R^d)$ and suppose, for some integers $0\le j\le k$, that $\nab^{\otimes j}f\in L^{s_2}(\R^d)$. If $1\le s,s_1,s_2\le\infty$ satisfy $1+\frac1s=\frac1{s_1}+\frac1{s_2}$, then Young's inequality gives
\begin{equation}
\|\nab^{\otimes k}(f\ast\eta_\vep)\|_{L^s}
\le C_{k-j,\eta,s_1}\,\vep^{-(k-j)-d(1-1/s_1)}
\|\nab^{\otimes j}f\|_{L^{s_2}}.
\end{equation}
Indeed, componentwise, $\nab^{\otimes k}(f\ast\eta_\vep)=(\nab^{\otimes j}f)\ast\nab^{\otimes(k-j)}\eta_\vep$. Under the hypotheses used here only the case $j=0$ is required, and no derivative of the data is assumed; in particular,
\begin{equation}
\|\nab^{\otimes k}\rho_{0,\vep}\|_{L^s}
\le C_{k,\eta,s_1}\,\vep^{-k-d(1-1/s_1)}\|\rho_0\|_{L^{s_2}},
\qquad
\|\nab^{\otimes k}\mu_\vep\|_{L^s}
\le C_{k,\eta,s_1}\,\vep^{-k-d(1-1/s_1)}\|\mu\|_{L^{s_2}}.
\end{equation}
Additionally, by Minkowski's inequality,
\begin{equation}
\mathcal M_r(\rho_{0,\vep})^{\frac1r}  \le \mathcal M_r(\rho_0)^{\frac1r} + \vep,
\qquad
\mathcal M_r(\mu_\vep)^{\frac1r} \le \mathcal M_r(\mu)^{\frac1r} + \vep.
\end{equation}
We next record the regularized-kernel estimates used in the uniform a priori bounds below.

\begin{lemma}[Regularized-kernel bounds]\label{lem:regularized-kernel-bounds}
Let $q\in(0,2)$ satisfy $d+q-2>0$, and let $K_\vep=K*\eta_\vep$ with $\eta_\vep\ge0$ and $\int_{\R^d}\eta_\vep(x)\ud x=1$.  Then, in the sense of distributions,
\begin{equation}
\Delta K=-q(d+q-2)|x|^{q-2}\le0,
\qquad
\Delta K_\vep=(\Delta K)*\eta_\vep\le0.
\end{equation}
Moreover, for all $f\in L^1(\R^d)\cap L^p(\R^d)$,
\begin{equation}\label{eq:reg-hessian-conv}
\sup_{x\in\R^d}\int_{\R^d}|\nabla^{\otimes 2}K_\vep(x-y)|\,|f(y)|\ud y
\le C\|f\|_{L^1}^{1-\frac{(2-q)p_*}{d}}\|f\|_{L^p}^{\frac{(2-q)p_*}{d}}.
\end{equation}
If $0<q\le1$, then
\begin{equation}\label{eq:reg-gradient-conv}
\sup_{x\in\R^d}\int_{\R^d}|\nabla K_\vep(x-y)|\,|f(y)|\ud y
\le C\|f\|_{L^1}^{1-\frac{(1-q)p_*}{d}}\|f\|_{L^p}^{\frac{(1-q)p_*}{d}},
\end{equation}
where for $q=1$ the right-hand side is interpreted as $C\|f\|_{L^1}$.  If $1<q<2$, then
\begin{equation}\label{eq:kernel-holder}
|\nabla K_\vep(x-z)-\nabla K_\vep(x)| \le C|z|^{q-1}
\qquad \text{for all } x,z\in \R^d.
\end{equation}
Finally, the bounds
\begin{equation}\label{eq:kernel-local-bounds}
\|\nabla K_\vep\|_{L^{p_*}(B_1)} + \|\nabla^{\otimes 2}K_\vep\|_{L^{p_*}(B_1)} \le C,
\end{equation}
\begin{equation}\label{eq:kernel-far-bounds}
\sup_{|x|\ge 2}\Big(\indic_{q\in(0,1]}|\nabla K_\vep(x)|
+|\nabla^{\otimes 2}K_\vep(x)|+|\nabla^{\otimes 3}K_\vep(x)|\Big) \le C,
\end{equation}
hold uniformly in $\vep\in(0,1)$.
\end{lemma}

\begin{proof}
Since $d+q-2>0$, $|x|^{q-2}$ is locally integrable and
\begin{equation}\label{eq:regularized-kernel-pointwise-bounds}
\Delta K=-q(d+q-2)|x|^{q-2},\qquad
|\nabla K(x)|\le C|x|^{q-1},\qquad
|\nabla^{\otimes 2}K(x)|+|\Delta K(x)|\le C|x|^{q-2}
\end{equation}
for $x\ne0$.  Convolution with the nonnegative mollifier preserves the sign of $\Delta K$ and gives
\begin{equation}\label{eq:regularized-kernel-convolution-bounds}
|\nabla^{\otimes 2}K_\vep|+|\Delta K_\vep|\le C\eta_\vep*|\cdot|^{q-2},
\qquad
|\nabla K_\vep|\le C\eta_\vep*|\cdot|^{q-1}\quad(0<q<1).
\end{equation}
Thus \eqref{eq:reg-hessian-conv} follows from \eqref{eq:singintbd-mollified} with $a=2-q$, and \eqref{eq:reg-gradient-conv} follows with $a=1-q$.  When $q=1$, $|\nabla K|$ is bounded, which gives the stated $L^1$ estimate directly.

For $1<q<2$, the homogeneity of $\nabla K$ and the mean-value theorem, according as $|x-y|$ is or is not comparable to $\max\{|x|,|y|\}$, yield
\begin{equation}\label{eq:kernel-gradient-global-holder}
|\nabla K(x)-\nabla K(y)|\le C|x-y|^{q-1},
\qquad x,y\in\R^d.
\end{equation}
Averaging \eqref{eq:kernel-gradient-global-holder} against $\eta_\vep$ proves \eqref{eq:kernel-holder}.  Finally, \eqref{eq:kernel-local-bounds} follows from Young's inequality and the local $L^{p_*}$ integrability of $\nabla K,\nabla^{\otimes 2}K$.  The estimate \eqref{eq:kernel-far-bounds} follows from the pointwise derivative bounds away from the origin, since $|x|\ge2$ and $z\in\operatorname{supp}\eta_\vep\subset B_1$ imply $|x-z|\ge1$.
\end{proof}

\subsection{Regularized problem}\label{ssec:LPwp-regularized}
With the regularized kernel and data fixed above, we implement the first step of the construction through the mollified continuity equation
\begin{equation}\label{eq:regeq}
\partial_t \rho^\vep + \div(\rho^\vep v^\vep)=0,
\qquad
v^\vep = -\nabla K_\vep*(\rho^\vep-\mu_\vep),
\qquad
\rho^\vep\big|_{t=0}=\rho_{0,\vep}.
\end{equation}

\begin{lemma}\label{lem:regularized-problem}
For every $\vep\in (0,1)$, \eqref{eq:regeq} has a unique global classical solution
$\rho^\vep \in C^1([0,\infty)\times \R^d)$. In addition, $\rho_t^\vep$ is a probability density
for every $t\ge 0$.
\end{lemma}

\begin{proof}

Fix $\vep\in(0,1)$ and set $L_\vep:=\|\nabla^{\otimes 2}K_\vep\|_{L^\infty}$.  Laurent's existence theorem is formulated under global integrability assumptions not satisfied by the present growing kernel, so we adapt only the characteristic fixed-point scheme from its proof \cite{laurent2007local}.  For $\sigma\in C([0,T];\mathcal P_1(\R^d))$, define
\begin{equation}\label{eq:regularized-trial-field}
b_\vep[\sigma_t](x):=-\int_{\R^d}\nabla K_\vep(x-y)\ud(\sigma_t-\mu_\vep)(y).
\end{equation}
For probability measures $\sigma,\tau\in\mathcal P_1(\R^d)$, coupling the relevant measures and using the equality of the source and target masses give
\begin{align}
\operatorname{Lip}_x b_\vep[\sigma]&\le2L_\vep,\label{eq:regularized-field-spatial-lipschitz}\\
\|b_\vep[\sigma]-b_\vep[\tau]\|_{L^\infty}&\le L_\vep W_1(\sigma,\tau),\label{eq:regularized-field-measure-lipschitz}\\
\|b_\vep[\sigma]\|_{L^\infty}&\le L_\vep\bigl(\mathcal M_1(\sigma)+\mathcal M_1(\mu_\vep)\bigr).\label{eq:regularized-field-bounded}
\end{align}
Indeed, the last estimate follows by writing the field against any coupling of $\sigma$ and $\mu_\vep$ and subtracting the two kernel gradients inside the integral.

Let $X^\sigma$ be the flow of $b_\vep[\sigma]$ and set
\begin{equation}\label{eq:regularized-fixed-point-map}
(\mathcal T\sigma)_t:=X^\sigma(t,0,\cdot)_\#\rho_{0,\vep}.
\end{equation}
On a closed ball of $C([0,T];\mathcal P_1)$ with uniformly bounded first moment, \eqref{eq:regularized-field-bounded} shows that $\mathcal T$ maps the ball into itself when $T$ is sufficiently small.  Moreover, \eqref{eq:regularized-field-spatial-lipschitz}--\eqref{eq:regularized-field-measure-lipschitz} and Gr\"onwall's inequality give
\begin{equation}\label{eq:regularized-fixed-point-contraction}
\sup_{0\le t\le T}W_1\bigl((\mathcal T\sigma)_t,(\mathcal T\tau)_t\bigr)
\le L_\vep T e^{2L_\vep T}\sup_{0\le t\le T}W_1(\sigma_t,\tau_t).
\end{equation}
Thus $\mathcal T$ is a contraction for smaller $T$, yielding a unique local solution transported by its characteristic flow.  Positivity and conservation of mass are automatic from the pushforward construction.

To continue the local solution globally, we control its first moment.  For the fixed point, \eqref{eq:regularized-field-bounded} yields
\begin{equation}\label{eq:regularized-first-moment-continuation}
\mathcal M_1(\rho_t^\vep)
\le \mathcal M_1(\rho_{0,\vep})
+L_\vep\int_0^t\bigl(\mathcal M_1(\rho_s^\vep)+\mathcal M_1(\mu_\vep)\bigr)\ud s.
\end{equation}
Gr\"onwall's inequality bounds the first moment on every finite interval, so the local construction may be iterated globally.  Finally, $\nabla^{\otimes m}K_\vep\in L^\infty$ for every integer $m\ge2$; differentiating the characteristic system and using the smooth data therefore gives $\rho^\vep\in C^1([0,\infty)\times\R^d)$.

\end{proof}

\subsection{Uniform estimates}\label{ssec:LPwp-uniform}

To pass to the singular equation, we need estimates on the smooth approximations that are uniform in $\vep$.  The $L^p$ and moment bounds below control the velocity, while the time estimates provide the compactness used in \Cref{ssec:LPwp-limit}.

\begin{lemma}\label{lem:uniform-estimates}
For every $T>0$ there exists a constant $C_T>0$, independent of $\vep$, such that
\begin{equation}\label{eq:uniform-estimates-approx}
\sup_{0\le t\le T}\Big(
\|\rho_t^\vep\|_{L^p}
+ \mathcal M_r(\rho_t^\vep)
+ \|v_t^\vep\|_{L^\infty}
+ \|\nabla v_t^\vep\|_{L^\infty}
+ \|\partial_t v_t^\vep\|_{L^\infty}
\Big)
\le C_T.
\end{equation}
Finally,
\begin{equation}\label{eq:W1-equi}
W_1(\rho_t^\vep,\rho_s^\vep) \le C_T |t-s|
\qquad \text{for all } s,t\in [0,T].
\end{equation}
\end{lemma}

\begin{proof}
We begin with the $L^p$ estimate. Since $\rho^\vep$ is smooth, multiplying \eqref{eq:regeq} by
$(\rho^\vep)^{p-1}$ and integrating by parts gives
\begin{equation}\label{eq:approx-lp-identity}
\frac{\ud}{\ud t}\|\rho_t^\vep\|_{L^p}^p
= (p-1)\int_{\R^d} (\rho_t^\vep(x))^p\,\Delta K_\vep*(\rho_t^\vep-\mu_\vep)(x)\ \ud x.
\end{equation}
Because $\Delta K_\vep\le 0$ and $\rho_t^\vep\ge 0$, the self-interaction contribution is nonpositive.
Thus,
\begin{equation}
\frac{\ud}{\ud t}\|\rho_t^\vep\|_{L^p}^p
\le (p-1)\|\Delta K_\vep*\mu_\vep\|_{L^\infty}\, \|\rho_t^\vep\|_{L^p}^p.
\end{equation}
Using \eqref{eq:reg-hessian-conv} with $f=\mu_\vep$, together with $\|\mu_\vep\|_{L^1}=1$ and $\|\mu_\vep\|_{L^p}\le \|\mu\|_{L^p}$, we obtain
\begin{align}\label{eq:laplaceKbd}
\|\Delta K_\vep*\mu_\vep\|_{L^\infty}
\lesssim \|\mu\|_{L^p}^{\frac{(2-q)p_*}{d}}.
\end{align}
Gr\"onwall's inequality therefore yields
\begin{equation}\label{eq:approxLpest}
\|\rho_t^\vep\|_{L^p} \le e^{Ct}\|\rho_0\|_{L^p}
\qquad \text{for all } t\in [0,T],
\end{equation}
with $C$ depending only on $d,q,p$ and $\|\mu\|_{L^p}$.

We next estimate the velocity. If $0<q\le 1$, then applying \eqref{eq:reg-gradient-conv} to $f=\rho_t^\vep-\mu_\vep$ gives
\begin{equation}
\begin{aligned}
|v_t^\vep(x)| \lesssim \|\rho_t^\vep - \mu_\vep\|_{L^1}^{1-\frac{(1-q)p_*}{d}} \|\rho_t^\vep - \mu_\vep\|_{L^p}^{\frac{(1-q)p_*}{d}} \lesssim \|\rho_t^\vep\|_{L^p}^{\frac{(1-q)p_*}{d}} + \|\mu\|_{L^p}^{\frac{(1-q)p_*}{d}} \le C_T.
\end{aligned}
\end{equation}
If $1<q<2$, then $\rho_t^\vep$ and $\mu_\vep$ have the same mass, so
\begin{equation}
v_t^\vep(x)
= -\int_{\R^d}\big(\nabla K_\vep(x-y)-\nabla K_\vep(x)\big)\ud(\rho_t^\vep-\mu_\vep)(y).
\end{equation}
Using \eqref{eq:kernel-holder},
\begin{align}\label{eq:velocity-holder}
|v_t^\vep(x)|\le C\int_{\R^d}|y|^{q-1}\ud(\rho_t^\vep+\mu_\vep)(y) 
&= C\big(\mathcal M_{q-1}(\rho_t^\vep)+\mathcal M_{q-1}(\mu_\vep)\big).
\end{align}

We next propagate the $r$-th moment; in the regime $1<q<2$, this also closes the moment-dependent velocity estimate \eqref{eq:velocity-holder}. When $r=1$, the following calculation is performed with smooth convex approximations of $|x|$ and then passed to the limit. Observe from  \eqref{eq:regeq},
\begin{equation}
\frac{\ud}{\ud t}\mathcal M_r(\rho_t^\vep)
= r\int_{\R^d} |x|^{r-2}x\cdot v_t^\vep(x)\,\rho_t^\vep(x)\ud x
\le r\|v_t^\vep\|_{L^\infty}\,\mathcal M_r(\rho_t^\vep)^{\frac{r-1}{r}}.
\end{equation}
Set $Y_\vep(t):=\mathcal M_r(\rho_t^\vep)^{1/r}$. The preceding differential inequality gives
\begin{equation}
Y_\vep(t)\le Y_\vep(0)+\int_0^t\|v_s^\vep\|_{L^\infty}\ud s
=\mathcal M_r(\rho_{0,\vep})^{1/r}+\int_0^t\|v_s^\vep\|_{L^\infty}\ud s.
\end{equation}
If $0<q\le1$, the already established bound $\|v_t^\vep\|_{L^\infty}\le C_T$ and the preceding integral inequality directly bound $Y_\vep$ on $[0,T]$.  If $1<q<2$, the assumption $r\ge1$ gives $r>q-1$, and H\"older's inequality yields $\mathcal M_{q-1}(\nu)\le\mathcal M_r(\nu)^{(q-1)/r}$ for every $\nu\in\mathcal P(\R^d)$. Hence \eqref{eq:velocity-holder} implies
\begin{equation}
Y_\vep(t)
\le Y_\vep(0)+C\int_0^t\left(Y_\vep(s)^{q-1}+\mathcal M_r(\mu_\vep)^{(q-1)/r}\right)\ud s.
\end{equation}
The initial values $Y_\vep(0)$ and the quantities $\mathcal M_r(\mu_\vep)^{1/r}$ are bounded uniformly in $\vep\in(0,1)$ by the mollified moment estimates above. Since $0<q-1<1$, the elementary bound $y^{q-1}\le1+y$ for $y\ge0$ and Gr\"onwall's inequality bound $Y_\vep$ on $[0,T]$.  Thus, in either case,
\begin{equation}\label{eq:momentr-uniform-approx}
\sup_{0\le t\le T}\mathcal M_r(\rho_t^\vep) \le C_T.
\end{equation}
This in turn implies that $\|v_t^\vep\|_{L^\infty}\le C_T$ for all $t\in[0,T]$.

With the velocity and moment bounds now closed, it remains to control the spatial and temporal derivatives of $v^\vep$ appearing in \eqref{eq:uniform-estimates-approx}. The estimate \eqref{eq:reg-hessian-conv} and the $L^p$ bound on $\rho_t^\vep$ give 
\begin{equation}
\begin{aligned}
\|\nab v_t^\vep\|_{L^\infty} \le \|\rho_t^\vep-\mu_\vep\|_{L^1}^{1-\frac{(2-q)p_*}{d}} \|\rho_t^\vep-\mu_\vep\|_{L^p}^{\frac{(2-q)p_*}{d}} \le C_T.
\end{aligned}
\end{equation}
Similarly, since $\partial_t\rho^\vep = -\div(\rho^\vep v^\vep)$,
\begin{equation}
\partial_t v_t^\vep = -\nabla K_\vep * \partial_t\rho_t^\vep = \nabla^{\otimes 2}K_\vep*(\rho_t^\vep v_t^\vep),
\end{equation}
and \eqref{eq:reg-hessian-conv} gives
\begin{equation}
\|\partial_t v_t^\vep\|_{L^\infty}
\le \|\rho_t^\vep v_t^\vep\|_{L^1}^{1-\frac{(2-q)p_*}{d}} \|\rho_t^\vep v_t^\vep\|_{L^p}^{\frac{(2-q)p_*}{d}}
\le C_T.
\end{equation}
This proves \eqref{eq:uniform-estimates-approx}.

Finally, let $X^\vep(t,s,\cdot)$ denote the characteristic flow of $v^\vep$. Since
$\rho_t^\vep = X^\vep(t,s,\cdot)_\#\rho_s^\vep$, we have
\begin{equation}
W_1(\rho_t^\vep,\rho_s^\vep)
\le \int_{\R^d}|X^\vep(t,s,x)-x|\,\rho_s^\vep(x)\ud x
\le \int_s^t \|v_\tau^\vep\|_{L^\infty}\ud\tau
\le C_T|t-s|,
\end{equation}
which is exactly \eqref{eq:W1-equi}.

\end{proof}

\subsection{Passage to the limit and verification}\label{ssec:LPwp-limit}

Fix $T>0$. By \eqref{eq:uniform-estimates-approx} and Arzel\`a--Ascoli, after passing to a subsequence,
\begin{equation}\label{eq:veps-to-v}
v^\vep \to v
\qquad \text{locally uniformly on } [0,T]\times \R^d.
\end{equation}
By \eqref{eq:W1-equi}, the family $\{\rho^\vep\}_{\vep\in(0,1)}$ is equicontinuous in $W_1$.  Moreover, the collection of time slices
$\{\rho_t^\vep:\vep\in(0,1),\ t\in[0,T]\}$ is relatively compact in $(\mathcal P_1(\R^d),W_1)$.  Indeed, the velocity bound gives $|X^\vep(t,0,x)-x|\le C_T$, while the mollifications $\rho_{0,\vep}$ have uniformly integrable first moments because $\rho_0\in\mathcal P_1$ and $\operatorname{supp}\eta_\vep\subset B_1$.  The pushforward representation therefore gives uniform integrability of the first moments of $\rho_t^\vep$ on $[0,T]$.  Hence, after passing to a further subsequence,
\begin{equation}\label{eq:rhoeps-to-rho-W1}
\rho^\vep \to \rho
\qquad \text{in } C([0,T];\P_1(\R^d))
\end{equation}
with respect to the $W_1$ metric. In particular, $\rho_t^\vep \rightharpoonup \rho_t$ narrowly for every $t$.
The limiting curve also satisfies the $L^p$ and moment bounds required by \Cref{def:continuum-solution-class}. Indeed, by \eqref{eq:approxLpest} and Banach--Alaoglu,
$\rho^\vep \rightharpoonup^* \rho$ in $L^\infty(0,T;L^p(\R^d))$; thus
$\rho \in L^\infty(0,T;L^p(\R^d))$. By lower semicontinuity of moments and
\eqref{eq:momentr-uniform-approx},
\begin{equation}\label{eq:limit-r-moment}
\sup_{0\le t\le T}\mathcal M_r(\rho_t) \le C_T,
\end{equation}
so in fact $\rho_t\in \P_r(\R^d)$ for all $t\in[0,T]$.

We next identify the limit velocity. If $0<q\le 1$, fix $(t,x)\in [0,T]\times \R^d$ and $\delta\in (0,1)$.  Let $\chi_\delta\in C_c^\infty(B_{2\delta})$ equal one on $B_\delta$.  For $0<\vep<\delta$, the support of the mollifier and Young's inequality give
\begin{equation}\label{eq:near-field-local-uniform-integrability}
\|\chi_\delta\nabla K_\vep\|_{L^{p_*}}
\le \|\nabla K\|_{L^{p_*}(B_{3\delta})}.
\end{equation}
Consequently, the uniform $L^p$ bounds imply
\begin{equation}\label{eq:near-field-limsup}
\limsup_{\vep\downarrow0}
\left|\int_{\R^d}\chi_\delta(x-y)\nabla K_\vep(x-y)\ud(\rho_t^\vep-\mu_\vep)(y)\right|
\le C_T\|\nabla K\|_{L^{p_*}(B_{3\delta})},
\end{equation}
which tends to zero as $\delta\downarrow0$.  On the complementary region, $(1-\chi_\delta)\nabla K_\vep$ converges uniformly to $(1-\chi_\delta)\nabla K$ and is bounded and continuous.  Using \eqref{eq:rhoeps-to-rho-W1}, we conclude that
\begin{equation}
v(t,x) = -\int_{\R^d}\nabla K(x-y)\ud(\rho_t-\mu)(y).
\end{equation}
If $1<q<2$, then we use cancellation of the total mass:
\begin{equation}
v_t^\vep(x)
= -\int_{\R^d}\big(\nabla K_\vep(x-y)-\nabla K_\vep(x)\big)\ud(\rho_t^\vep-\mu_\vep)(y).
\end{equation}
For fixed $x$, the functions
$G_{\vep,x}(y):=\nabla K_\vep(x-y)-\nabla K_\vep(x)$ converge locally uniformly to
$G_x(y):=\nabla K(x-y)-\nabla K(x)$, and by \eqref{eq:kernel-holder},
$|G_{\vep,x}(y)|\le C|y|^{q-1}$. Since $q-1<1$, the function $G_x$ is continuous and has at most linear growth.
Therefore \eqref{eq:rhoeps-to-rho-W1} implies
\begin{equation}
v(t,x)
= -\int_{\R^d}\big(\nabla K(x-y)-\nabla K(x)\big)\ud(\rho_t-\mu)(y)
= -\int_{\R^d}\nabla K(x-y)\ud(\rho_t-\mu)(y).
\end{equation}
Thus, in all cases,
\begin{equation}\label{eq:limit-velocity-formula}
v_t = -\nabla K*(\rho_t-\mu).
\end{equation}
In particular, the estimates proved in \Cref{lem:uniform-estimates} pass to the limit and show that
\begin{equation}
v\in L^\infty([0,T];W^{1,\infty}(\R^d))
\cap C([0,T]\times \R^d).
\end{equation}

We now pass to the weak formulation. Let $\varphi\in C_c^\infty([0,T)\times \R^d)$. Since
$\rho^\vep$ is a classical solution of \eqref{eq:regeq},
\begin{equation}
\int_0^T\!\!\int_{\R^d} \big(\partial_t\varphi + v_t^\vep\cdot \nabla_x\varphi\big)\rho_t^\vep\ud x\ud t
+ \int_{\R^d}\varphi(0,x)\rho_{0,\vep}(x)\ud x =0.
\end{equation}
Using \eqref{eq:veps-to-v}, \eqref{eq:rhoeps-to-rho-W1}, and the fact that
$\rho_{0,\vep}\to \rho_0$ in $L^1(\R^d)$, we may pass to the limit and obtain
\begin{equation}
\int_0^T\!\!\int_{\R^d} \big(\partial_t\varphi + v_t\cdot \nabla_x\varphi\big)\rho_t\ud x\ud t
+ \int_{\R^d}\varphi(0,x)\rho_0(x)\ud x =0.
\end{equation}
Therefore $\rho$ is a weak solution of \eqref{eq:wgfmmd} on $[0,T]$.

We next record the flow representation and the asserted time regularity.  The uniform time-derivative estimate passes to the limit and gives
\begin{equation}\label{eq:limit-v-time-lip}
\|v_t-v_s\|_{L^\infty(\R^d)}\le C_T|t-s|,
\qquad s,t\in[0,T].
\end{equation}
Together with $v\in L^\infty_tW^{1,\infty}_x$, this yields a global bi-Lipschitz flow $X(t,s,\cdot)$ and the standard bounds
\begin{align}
|X(t,s,x)-X(t,s,y)|&\le e^{C_T|t-s|}|x-y|,
\qquad |X(t,s,x)-x|\le C_T|t-s|,\notag\\
e^{-C_T|t-s|}&\le \det\nabla_xX(t,s,x)\le e^{C_T|t-s|}
\quad\text{for $\mathcal L^d$-a.e. }x.\label{eq:limit-flow-jacobian}
\end{align}
For the bi-Lipschitz flow estimates and boundedness of composition by the inverse flow on $L^p$, see \cite[Lemmas~2.7--2.8]{bertozzi2011lp}.  The representation and backward-adjoint uniqueness theory for the linear continuity equation, in the stronger regularity class used here, is given in \cite[Remark~8.1.5 and Propositions~8.1.7--8.1.8]{ambrosio2008gradient}.  It yields
\begin{equation}\label{eq:limit-characteristic-density}
\rho_t=X(t,0,\cdot)_\#\rho_0.
\end{equation}
Writing $J(t,s,a):=\det\nabla_aX(t,s,a)$, the corresponding density formula is
\begin{align}
\rho_t(X(t,0,a))J(t,0,a)&=\rho_0(a)\quad\text{for $\mathcal L^d$-a.e. }a,\notag\\
\rho_t(x)&=\rho_0(X(0,t,x))
\exp\left(-\int_0^t\div v_\tau(X(\tau,t,x))\ud\tau\right)
\quad\text{for $\mathcal L^d$-a.e. }x.\label{eq:limit-density-formula}
\end{align}
The density-approximation and composition argument in the proof of \cite[Proposition~2.11]{bertozzi2011lp}, including the appendix argument used there, applied to \eqref{eq:limit-characteristic-density} yields
\begin{equation}\label{eq:limit-rho-C-Lp}
\rho\in C([0,T];L^p(\R^d)),
\qquad
\|\rho_t\|_{L^p}\le e^{C_Tt}\|\rho_0\|_{L^p}.
\end{equation}
Finally, \eqref{eq:limit-v-time-lip} and \eqref{eq:limit-rho-C-Lp} imply $\rho v\in C([0,T];L^p)$; hence the weak equation gives
\begin{equation}\label{eq:limit-rho-C1-Wminus}
\partial_t\rho=-\div(\rho v)\in C([0,T];W^{-1,p}(\R^d)),
\qquad
\rho\in C^1([0,T];W^{-1,p}(\R^d)).
\end{equation}
Since $T$ is arbitrary, the finite-$p$ solution is global and belongs to the class of \Cref{def:continuum-solution-class}.

\subsection{Uniqueness}\label{ssec:LPwp-uniqueness}
Uniqueness is a consequence of the following Dobrushin-type stability estimate.  The argument only uses the solution class already obtained above: the associated velocities are bounded and globally Lipschitz in space on each finite time interval, and the solutions are represented by their characteristic flows.

The stability argument requires pointwise control of kernel-gradient differences, which we record first.

\begin{lemma}[Kernel-gradient difference estimate]\label{lem:kernel-gradient-difference}
Let $0<q<2$.  There exists a constant $C_q<\infty$ such that, for all $u,v\in\R^d\setminus\{0\}$,
\begin{equation}\label{eq:kernel-difference}
|\nabla K(u)-\nabla K(v)|
\le C_q |u-v|\bigl(|u|^{q-2}+|v|^{q-2}\bigr).
\end{equation}
The same estimate holds with the convention that the right-hand side is $+\infty$ if $u=0$ or $v=0$.
\end{lemma}

\begin{proof}
It suffices to prove the estimate for nonzero $u,v$.  If the segment $[u,v]$ stays at distance at least $\frac12\min\{|u|,|v|\}$ from the origin, then the mean-value theorem and the bound $|\nabla^{\otimes 2}K(w)|\le C_q|w|^{q-2}$ give
\begin{equation}
|\nabla K(u)-\nabla K(v)|
\le C_q |u-v|\min\{|u|,|v|\}^{q-2}
\le C_q |u-v|\bigl(|u|^{q-2}+|v|^{q-2}\bigr),
\end{equation}
because $q-2<0$.  Otherwise the segment passes within distance $\frac12\min\{|u|,|v|\}$ of the origin.  In that case $|u-v|\ge \frac12\min\{|u|,|v|\}$, and the triangle inequality also gives $\max\{|u|,|v|\}\le 3|u-v|$.  Hence, using $|\nabla K(w)|\le C_q|w|^{q-1}$,
\begin{equation}
|\nabla K(u)-\nabla K(v)|
\le C_q\bigl(|u|^{q-1}+|v|^{q-1}\bigr)
\le C_q |u-v|\bigl(|u|^{q-2}+|v|^{q-2}\bigr).
\end{equation}
This proves the lemma.
\end{proof}

\begin{proposition}[$W_1$-stability and uniqueness]\label{prop:W1-stability}
Let $T>0$, let $0<q<2$ satisfy $d+q-2>0$, and let $1<p<\infty$ satisfy $p>p_c$.  Set $p_*:=p/(p-1)$ and
\begin{equation}
\beta:=\frac{(2-q)p_*}{d}\in(0,1).
\end{equation}
Let $\mu\in\mathcal P_1(\R^d)\cap L^p(\R^d)$, and let $\rho^{(1)},\rho^{(2)}$ be two weak solutions of \eqref{eq:wgfmmd} on $[0,T]$ in the solution class constructed above, with respective initial data $\rho_0^{(1)},\rho_0^{(2)}\in \mathcal P_1(\R^d)\cap L^p(\R^d)$.  Then, for every $t\in[0,T]$,
\begin{equation}\label{eq:W1-stability-integral}
W_1\bigl(\rho_t^{(1)},\rho_t^{(2)}\bigr)
\le
\exp\!\left(
C_{d,q,p}\int_0^t
\bigl(
\|\rho_s^{(1)}\|_{L^p}^{\beta}
+
\|\rho_s^{(2)}\|_{L^p}^{\beta}
+
\|\mu\|_{L^p}^{\beta}
\bigr)
\ud s
\right)
W_1\bigl(\rho_0^{(1)},\rho_0^{(2)}\bigr).
\end{equation}
In particular, two solutions in this class with the same initial datum coincide on $[0,T]$.
\end{proposition}

\begin{proof}
For $i=1,2$, set
\begin{equation}
v_t^{(i)}:=-\nabla K*(\rho_t^{(i)}-\mu).
\end{equation}
By \cite[Propositions~8.1.7--8.1.8]{ambrosio2008gradient}, every solution in the stated class is represented by the unique flow of its bounded, spatially Lipschitz velocity.  Let $X_t^{(i)}$ denote these flows, so that $\rho_t^{(i)}=(X_t^{(i)})_\#\rho_0^{(i)}$.  Let $\gamma_0\in\Gamma(\rho_0^{(1)},\rho_0^{(2)})$ be an optimal coupling for $W_1$, define
\begin{equation}
\gamma_t:=(X_t^{(1)},X_t^{(2)})_{\#}\gamma_0\in\Gamma(\rho_t^{(1)},\rho_t^{(2)}),
\end{equation}
and set
\begin{equation}
D(t):=\iint_{\R^d\times\R^d}|x-y|\ud\gamma_t(x,y)
=
\iint_{\R^d\times\R^d}|X_t^{(1)}(x)-X_t^{(2)}(y)|\ud\gamma_0(x,y).
\end{equation}
The map $D$ is absolutely continuous, $D(0)=W_1(\rho_0^{(1)},\rho_0^{(2)})$, and
\begin{equation}\label{eq:W1-le-D}
W_1\bigl(\rho_t^{(1)},\rho_t^{(2)}\bigr)\le D(t).
\end{equation}
For $\mathcal L^1$-a.e. $t\in[0,T]$,
\begin{align}
D'(t)
&\le
\iint |v_t^{(1)}(x)-v_t^{(2)}(y)|\ud\gamma_t(x,y)\notag\\
&\le
\iint |v_t^{(1)}(x)-v_t^{(1)}(y)|\ud\gamma_t(x,y)
+
\iint |v_t^{(1)}(y)-v_t^{(2)}(y)|\ud\gamma_t(x,y)\notag\\
&=:I_1(t)+I_2(t).\label{eq:Dprime-split}
\end{align}

For $I_1$, using $|\nabla^{\otimes 2}K(z)|\le C_{d,q}|z|^{q-2}$ and \eqref{eq:singintbd-mollified} with $a=2-q$ and $\vep=0$,
\begin{align}
\|\nabla v_t^{(1)}\|_{L^\infty}
&\le C_{d,q}\sup_x\int |x-y|^{q-2}\ud|\rho_t^{(1)}-\mu|(y) \notag\\
&\le C_{d,q,p}\bigl(\|\rho_t^{(1)}\|_{L^p}^{\beta}+\|\mu\|_{L^p}^{\beta}\bigr).
\end{align}
Therefore
\begin{equation}\label{eq:I1-bound}
I_1(t)
\le C_{d,q,p}\bigl(\|\rho_t^{(1)}\|_{L^p}^{\beta}+\|\mu\|_{L^p}^{\beta}\bigr)D(t).
\end{equation}

For $I_2$, the target terms cancel.  Since $\gamma_t\in\Gamma(\rho_t^{(1)},\rho_t^{(2)})$,
\begin{equation}
v_t^{(1)}(y)-v_t^{(2)}(y)
= -\iint
\bigl(\nabla K(y-z)-\nabla K(y-\widetilde z)\bigr)
\ud\gamma_t(z,\widetilde z).
\end{equation}
Using \Cref{lem:kernel-gradient-difference}, Fubini, and again \eqref{eq:singintbd-mollified} with $a=2-q$ and $\vep=0$, we obtain
\begin{align}
I_2(t)
&\le C_{d,q}\iiint
|z-\widetilde z|
\bigl(|y-z|^{q-2}+|y-\widetilde z|^{q-2}\bigr)
\ud\rho_t^{(2)}(y)\ud\gamma_t(z,\widetilde z) \notag\\
&\le C_{d,q}\left(\sup_{a\in\R^d}\int |y-a|^{q-2}\rho_t^{(2)}(y)\ud y\right)
\iint |z-\widetilde z|\ud\gamma_t(z,\widetilde z) \notag\\
&\le C_{d,q,p}\|\rho_t^{(2)}\|_{L^p}^{\beta}D(t).
\end{align}
Thus
\begin{equation}\label{eq:I2-bound}
I_2(t)\le C_{d,q,p}\|\rho_t^{(2)}\|_{L^p}^{\beta}D(t).
\end{equation}
Combining \eqref{eq:Dprime-split}, \eqref{eq:I1-bound}, and \eqref{eq:I2-bound}, we have for $\mathcal L^1$-a.e. $t\in[0,T]$,
\begin{equation}
D'(t)
\le C_{d,q,p}
\bigl(
\|\rho_t^{(1)}\|_{L^p}^{\beta}
+
\|\rho_t^{(2)}\|_{L^p}^{\beta}
+
\|\mu\|_{L^p}^{\beta}
\bigr)D(t).
\end{equation}
Gr\"onwall's lemma gives \eqref{eq:W1-stability-integral}.  Together with \eqref{eq:W1-le-D}, this proves the stability estimate.  If the two initial data coincide, then the right-hand side of \eqref{eq:W1-stability-integral} is zero for every $t\in[0,T]$, hence $\rho_t^{(1)}=\rho_t^{(2)}$ for all $t\in[0,T]$.
\end{proof}

Thus the finite-$p$ solution is unique; in particular, the full regularized family converges to it, completing the finite-$p$ part of \Cref{thm:wgfmmdwp}.

\begin{proof}[$L^\infty$ refinement for $d+q-2>0$ in \Cref{thm:wgfmmdwp}]
Choose $p_0>p_c$.  Since $L^1\cap L^\infty\subset L^{p_0}$, the finite-$p_0$ construction supplies the solution and its characteristic representation.  Moreover,
\begin{equation}\label{eq:endpoint-Amu-finite}
\|(-\Delta K)*\mu\|_{L^\infty}<\infty
\end{equation}
by \Cref{lem:singular-convolution-bound} with $p=\infty$ and $a=2-q$.  The density formula \eqref{eq:limit-density-formula} gives, for $\mathcal L^d$-a.e. $x$,
\begin{equation}\label{eq:endpoint-density-formula}
\rho_t(X(t,0,x))
=\rho_0(x)\exp\left(\int_0^t\Delta K*(\rho_s-\mu)(X(s,0,x))\ud s\right).
\end{equation}
Since $\Delta K\le0$ and $\rho_s\ge0$,
\begin{equation}\label{eq:endpoint-characteristic-inequality}
\Delta K*(\rho_s-\mu)\le(-\Delta K)*\mu\le \|(-\Delta K)*\mu\|_{L^\infty},
\end{equation}
so \eqref{eq:linfty-endpoint-bound} follows immediately.

For every finite $s$, the composition and approximation argument cited after \eqref{eq:limit-characteristic-density}, applied with exponent $s$, gives $\rho\in C([0,T];L^s)$.  The uniform $L^\infty$ bound and strong $L^1$ continuity imply weak-star continuity in $L^\infty$.  Finally, let $\widetilde\rho$ be any other solution in the case $p=\infty$, and retain the auxiliary exponent $p_0$.  Both densities belong to $C([0,T];L^{p_0})$.  For either solution,
\begin{align}
\|\rho_t v_t-\rho_s v_s\|_{L^{p_0}}
&\le\|v_t\|_{L^\infty}\|\rho_t-\rho_s\|_{L^{p_0}}
+\|\rho_s(v_t-v_s)\|_{L^{p_0}}.\label{eq:endpoint-flux-continuity}
\end{align}
The first term tends to zero by strong $L^{p_0}$ continuity.  Joint continuity of $v$ gives pointwise convergence in the second term, while the uniform velocity bound and $\rho_s\in L^{p_0}$ give dominated convergence.  Hence $\rho v\in C([0,T];L^{p_0})$, and the weak equation yields $\rho\in C^1([0,T];W^{-1,p_0})$; the same holds for $\widetilde\rho$.  Both solutions therefore belong to the finite-$p_0$ uniqueness class, and \Cref{prop:W1-stability} gives $\rho=\widetilde\rho$.  This completes \Cref{thm:wgfmmdwp}.
\end{proof}

We next record the energy--dissipation identity satisfied by the solution constructed above.

\begin{proposition}[Energy--dissipation identity]\label{prop:energy-dissipation}
Under the hypotheses of \Cref{thm:wgfmmdwp}, let $\rho_t$ be the corresponding solution, and set
\begin{equation}\label{eq:continuum-dissipation}
\mathscr D_q(\rho_t\mid\mu)
:=\int_{\R^d}|\nabla K*(\rho_t-\mu)|^2\ud\rho_t.
\end{equation}
Then $t\mapsto\mmd_q^2(\rho_t,\mu)$ is continuously differentiable and, for every $0\le s\le t$,
\begin{equation}\label{eq:continuum-energy-dissipation}
\mmd_q^2(\rho_t,\mu)
+\int_s^t\mathscr D_q(\rho_\tau\mid\mu)\ud\tau
=\mmd_q^2(\rho_s,\mu).
\end{equation}
\end{proposition}

\begin{proof}
Recall that $k_q$ denotes the centered kernel defined in \eqref{eq:centered-positive-kernel-intro}.
The H\"older continuity of $K$ for $0<q\le1$ and the $(q-1)$-H\"older continuity of $\nabla K$ for $1<q<2$ give
\begin{equation}
|k_q(x,y)|
\le C_q
\begin{cases}
|x|^q+|y|^q,&0<q\le1,\\
|x|^{q-1}|y|+|y|^{q-1}|x|,&1<q<2.
\end{cases}
\end{equation}
Both bounds are integrable against $\ud|\rho_t-\mu|(x)\ud|\rho_t-\mu|(y)$ under the finite-first-moment hypothesis.  Define the associated normalized potential by
\begin{equation}
u_t(x):=2\int_{\R^d}k_q(x,y)\ud(\rho_t-\mu)(y)
=\int_{\R^d}\bigl(K(x-y)-K(y)\bigr)\ud(\rho_t-\mu)(y),
\end{equation}
where the second identity uses $(\rho_t-\mu)(\R^d)=0$.  The local convolution estimates identify its weak gradient as
\begin{equation}
\nabla u_t=\nabla K*(\rho_t-\mu)=-v_t.
\end{equation}
On every bounded time interval $\nabla u=-v$ is bounded and jointly continuous, so $u_t$ has a globally Lipschitz representative.  For $0\le a\le b\le T$, symmetry and polarization in the centered representation \eqref{eq:defmmd} give
\begin{equation}\label{eq:continuum-energy-polarization}
\mmd_q^2(\rho_b,\mu)-\mmd_q^2(\rho_a,\mu)=\int_{\R^d}\frac{u_b+u_a}{2}\ud(\rho_b-\rho_a)=\int_a^b\int_{\R^d}\frac{\nabla u_b+\nabla u_a}{2}\cdot v_\tau\ud\rho_\tau\ud\tau.
\end{equation}
Indeed, the weak continuity equation extends from compactly supported smooth tests to globally Lipschitz tests with at most linear growth by cutoff and mollification, using the bounded velocity and the finite-horizon first-moment bound.  The polarization and difference-quotient argument is standard for Wasserstein gradient-flow energy identities (cf.\ \cite[Chapter~11]{ambrosio2008gradient} and \cite[Proposition~2.5]{chizat2026quantitativeconvergencewassersteingradient}); compare also the approximation-based energy inequality in \cite[(2.22)]{ChodronDeCourcelRosenzweigCoulombDiscrepancies}.

Dividing \eqref{eq:continuum-energy-polarization} by $b-a$ and letting $b\to a$, the joint continuity and uniform boundedness of $\nabla u=-v$, together with the $W_1$-continuity of $\rho_t$, give
\begin{equation}
\frac{\ud}{\ud t}\mmd_q^2(\rho_t,\mu)
=\int_{\R^d}\nabla u_t\cdot v_t\ud\rho_t
=-\int_{\R^d}|\nabla K*(\rho_t-\mu)|^2\ud\rho_t.
\label{eq:continuum-energy-derivative}
\end{equation}
The right-hand side is continuous in $t$, so $t\mapsto\mmd_q^2(\rho_t,\mu)$ is continuously differentiable.  Integrating \eqref{eq:continuum-energy-derivative} proves \eqref{eq:continuum-energy-dissipation}.
\end{proof}

Before recording the higher-regularity obstruction below, we note that the same computation as in \Cref{sec:meanfield} gives an MMD stability estimate for two continuum solutions with the same target.  For different targets, an additional forcing term involving the target mismatch appears; we do not formulate that refinement here.

\begin{remark}[Coulomb higher-regularity obstruction]\label{rem:coulomb-higher-regularity-obstruction}\label{prop:1d-coulomb-sobolev-blowup}
The admissible integrability range in \Cref{thm:wgfmmdwp} should not be interpreted as giving uniform-in-time control of stronger spatial norms. At the Coulomb endpoint, smooth compactly supported radial source and target densities can generate solutions whose H\"older seminorms grow exponentially. The same construction yields unbounded growth of every supercritical Sobolev norm. This Coulomb-specific phenomenon, including the one-dimensional mechanism previously recorded here, is proved in \cite[Theorem~1.2 and Proposition~2.5]{ChodronDeCourcelRosenzweigCoulombDiscrepancies}.
\end{remark}

\section{\texorpdfstring{Particle dynamics and mean-field estimate}{Particle dynamics and mean-field estimate}}\label{sec:meanfield}
We prove \Cref{thm:particle-global-noncollision,prop:particle-saddle-configurations,thm:wgfmmdpoc} in this section.  The noncollision and asymptotic-separation arguments exploit the repulsive sign through a centered second-moment estimate for particles approaching a common collision point.  The saddle constructions combine explicit collinear equilibria with symmetry-constrained polygonal configurations.  The passage from discrete to continuum criticality in \Cref{prop:particle-critical-closure} rests on pair symmetrization and regime-dependent control of the resulting two-point integral: bounded continuity at $q=1$, moment control for $1<q<2$, and a discrete Riesz bound for $0<q<1$.  Energy dissipation and the MMD topology yield the uniform-in-time consistency result in \Cref{cor:well-prepared-particle-uniform-time}.  The proof of \Cref{cor:particle-late-time-critical-limit} obtains limits along sequences $N_k\to\infty$ and $t_k\to\infty$ by selecting exact critical configurations from the fixed-$N$ $\omega$-limit sets, establishing tightness from uniform energy and moment bounds, and then applying the closure result; \Cref{lem:particle-critical-virial} supplies the required uniform energy bound for compactly supported targets.  When $0<q<1$, \Cref{lem:sublinear-frostman-riesz,prop:particle-hessian-nonconcentration} give two sufficient criteria for the required Riesz bound.  The mean-field estimate follows the modulated-energy approach reviewed in \cite{SerfatyLN,Rosenzweig2026Commutators}: the square of the MMD distance between $\rho_t^N$ and $\rho_t$ is a coercive quantity whose growth is controlled by a transport commutator and Gr\"onwall's inequality.

For a configuration $\XN=(x_i)_{i\in[N]}$, set $\rho^N:=\frac1N\sum_{i=1}^N\delta_{x_i}$, the time-independent analogue of the empirical measure $\rho_t^N$ in \eqref{eq:intro-empirical-error}.  When $\rho^N$ and $\rho$ have finite $q$-moment, the representation \eqref{def:mmdq} gives
\begin{equation}\label{eq:mf-modulated-energy-static}
\mmd_q^2(\rho^N,\rho)
=\frac12\int_{(\R^d)^2}K(x-y)\ud(\rho^N-\rho)(x)\ud(\rho^N-\rho)(y),
\qquad K(x)=-|x|^q.
\end{equation}
We call this the modulated energy.  Compared with the Coulomb (and more generally singular Riesz) modulated energy, the energy itself does not require an excision of the diagonal: $K$ is continuous and $K(0)=0$, so the empirical self-interaction terms in \eqref{eq:mf-modulated-energy-static} vanish.  This observation does \emph{not} remove the diagonal issue from the particle ODE or from the differentiated identity, because $\nabla K(0)$ may be undefined or infinite.  The velocity convention below is therefore diagonal-free.

Accordingly, we use the following solution concept.

\begin{definition}[Collision-free particle solution]\label{def:collision-free-particles}
Fix $T>0$.  A collision-free particle solution of \eqref{eq:wgfmmdparticle} on $[0,T]$ is an absolutely continuous map
\begin{equation}
[0,T]\ni t\mapsto \XN^t=(x_i^t)_{i\in[N]}\in(\R^d)^N
\end{equation}
such that $x_i^t\neq x_j^t$ for every distinct $i,j\in[N]$ and every $t\in[0,T]$, and such that for $\mathcal L^1$-a.e. $t\in[0,T]$,
\begin{equation}\label{eq:collision-free-particle-ode}
\dot x_i^t
=-\frac1N\sum_{j\in[N]\setminus\{i\}}\nabla K(x_i^t-x_j^t)+\nabla K*\mu(x_i^t),
\qquad i\in[N].
\end{equation}
\end{definition}

We first establish unique local well-posedness away from collisions in cases~\textup{(i)}--\textup{(ii)} of \Cref{thm:particle-global-noncollision}.  The existence-only case~\textup{(iii)} is constructed separately by Peano's theorem at the beginning of the proof of that theorem.

\begin{proposition}[Particle ODE up to first collision]\label{prop:particle-ode-local}
Let $0<q<2$.  Assume either that $d+q-2>0$ and $1<p<\infty$ satisfies $p>p_c$, or that $d=q=1$ and $p=p_c=\infty$.  Let $\mu\in\mathcal P_1(\R^d)\cap L^p(\R^d)$ be a probability density.  For $N\ge2$, set
\begin{equation}
\mathcal D_N:=\{\XN=(x_i)_{i\in[N]}\in(\R^d)^N:\ x_i\neq x_j\ \text{for every distinct }i,j\in[N]\}.
\end{equation}
For every initial configuration $\XN^0\in\mathcal D_N$, there is a unique maximal solution
\begin{equation}
\XN\in C^1([0,T_*);\mathcal D_N)
\end{equation}
of \eqref{eq:collision-free-particle-ode} with $\XN(0)=\XN^0$.  If $T_*<\infty$, then
\begin{equation}
\liminf_{t\uparrow T_*}\min_{\substack{i,j\in[N]\\i\neq j}}|x_i^t-x_j^t|=0.
\end{equation}
Consequently, for every $T<T_*$, the restriction of $\XN$ to $[0,T]$ is a collision-free particle solution in the sense of \Cref{def:collision-free-particles}.
\end{proposition}

\begin{proof}
Define on $\mathcal D_N$
\begin{equation}\label{eq:particle-vector-field-local}
b_i(\XN):=-\frac1N\sum_{j\in[N]\setminus\{i\}}\nabla K(x_i-x_j)+\nabla K*\mu(x_i),
\qquad i\in[N].
\end{equation}
The pair-interaction part is smooth away from the collision set.  If $d=q=1$, then $(K' *\mu)'=-2\mu$ distributionally, so $K'*\mu$ is globally Lipschitz.  Otherwise, \Cref{lem:singular-convolution-bound}, applied with $a=2-q$, gives
\begin{equation}\label{eq:target-field-global-lipschitz}
\|\nabla^{\otimes 2}K*\mu\|_{L^\infty}
\le C_{d,q,p}\|\mu\|_{L^1}^{1-\frac{(2-q)p_*}{d}}\|\mu\|_{L^p}^{\frac{(2-q)p_*}{d}}<\infty.
\end{equation}
Thus, in either case, $b$ is locally Lipschitz on $\mathcal D_N$.  Picard--Lindel\"of therefore gives a unique maximal solution and the usual continuation alternative.

It remains only to exclude finite-time escape to spatial infinity while the configuration stays away from the diagonal.  If $\min_{\substack{i,j\in[N]\\i\ne j}}|x_i^t-x_j^t|\ge\eta>0$, then the pair forces have at most linear growth.  The same is true of the target field: for $q\le1$ this follows by splitting the convolution into $|x-y|<1$ and its complement, and for $q>1$ from $|x-y|^{q-1}\le C(1+|x|+|y|)$ and $\mu\in\mathcal P_1$.  Hence, with $R(t):=\max_{i\in[N]}|x_i^t|$,
\begin{equation}\label{eq:particle-spatial-growth}
|\dot x_i^t|\le C_\eta(1+R(t)),
\qquad i\in[N].
\end{equation}
Gr\"onwall bounds $R$ on every finite interval.  Thus a finite maximal endpoint can occur only by approaching the collision set, which is the asserted alternative.
\end{proof}

We now exploit repulsion to rule out the collision alternative.

\begin{proof}[Proof of \Cref{thm:particle-global-noncollision}]
\emph{Global existence and noncollision.}

We first address the local-existence issue in case~\textup{(iii)}.  Let $p_*$ be the H\"older conjugate of $p$.  Since $p>1/q$ is equivalent to $(1-q)p_*<1$, \Cref{lem:singular-convolution-bound}, applied with $d=1$ and $a=1-q$, gives $K'*\mu\in L^\infty(\R)$.  Moreover, H\"older's inequality gives, for $0<r\le1$,
\begin{equation}\label{eq:one-dimensional-target-field-uniform-integrability}
\sup_{x\in\R}\int_{|x-y|<r}|x-y|^{q-1}\mu(y)\ud y
\le C_{p,q}\|\mu\|_{L^p}r^{q-1/p}.
\end{equation}
The right-hand side tends to zero as $r\downarrow0$.  Splitting the difference $(K'*\mu)(x+h)-(K'*\mu)(x)$ into neighborhoods of its two singularities and their complement, then using \eqref{eq:one-dimensional-target-field-uniform-integrability} and dominated convergence, yields $K'*\mu\in C_b(\R)$.
Thus $K*\mu\in C^1(\R)$ with $(K*\mu)'=K'*\mu$.

The pair-interaction field is smooth on $\mathcal D_N$, so the full vector field is continuous there.  Peano's theorem \cite[Theorem~2.19]{Teschl2012ODE} gives a local $C^1$ solution from every $X_N^0\in\mathcal D_N$, and standard continuation gives a maximal-lifespan solution.  Fix an arbitrary such solution in case~\textup{(iii)}.

In cases~\textup{(i)} and \textup{(ii)}, if $p=\infty$ and $d+q-2>0$, choose any finite $p_0>p_c$.  Since a probability density in $L^\infty$ belongs to $L^{p_0}$, it suffices to prove the result with $p=p_0$.  By \Cref{prop:particle-ode-local}, there is a unique maximal collision-free solution.  Thus, in all three cases, we may write the chosen maximal-lifespan solution as

\begin{equation}
\XN\in C^1([0,T_*);\mathcal D_N)
\end{equation}
with $\XN(0)=\XN^0$.  In cases~\textup{(i)} and \textup{(ii)} this solution is unique, whereas in case~\textup{(iii)} it is arbitrary.  It suffices to rule out $T_*<\infty$.

We first record a finite-energy consequence that prevents a finite-time endpoint from occurring at spatial infinity.  Abbreviate
\begin{equation}\label{eq:particle-energy-definition}
\mathcal E_N(\XN)
:=\frac12\int_{(\R^d)^2}K(x-y)\ud(\rho^N-\mu)(x)\ud(\rho^N-\mu)(y)
=\mmd_q^2(\rho^N,\mu)\ge0,
\end{equation}
The reader may check that \eqref{eq:collision-free-particle-ode} is equivalent to
\begin{equation}
\dot x_i=-N\nabla_{x_i}\mathcal E_N(\XN),
\qquad i\in[N].
\end{equation}
Consequently,
\begin{equation}
\frac{\ud}{\ud t}\mathcal E_N(\XN^t)
=-\frac1N\sum_{i=1}^N|\dot x_i^t|^2.
\end{equation}
For every $T<T_*$,
\begin{equation}
\frac1N\int_0^T\sum_{i=1}^N|\dot x_i^t|^2\ud t
\le \mathcal E_N(\XN^0).
\end{equation}
If $T_*<\infty$, it follows that each $x_i^t$ has a finite limit as $t\uparrow T_*$.

Assume now, for contradiction, that $T_*<\infty$.  In cases~\textup{(i)} and \textup{(ii)}, \Cref{prop:particle-ode-local} gives the collision alternative.  In case~\textup{(iii)}, let $\XN^*:=\lim_{t\uparrow T_*}\XN^t$, which exists by the preceding energy estimate.  If $\XN^*\in\mathcal D_N$, Peano's theorem gives a local solution starting from $\XN^*$, and concatenating it with $\XN$ yields a $C^1$ extension beyond $T_*$ because the vector field is continuous.  This contradicts maximality.  Hence, in every case,
\begin{equation}
	\lim_{t\uparrow T_*}\min_{\substack{i,j\in[N]\\i\neq j}}|x_i^t-x_j^t|=0.
\end{equation}
Since each particle has a finite limit as $t\uparrow T_*$, there is a point $x_*\in\mathbb R^d$ that is the common limiting position of at least two particles.  Set $C:=\{i\in[N]:\lim_{t\uparrow T_*}x_i^t=x_*\}$ and $m:=|C|\ge2$.  For $t$ sufficiently close to $T_*$, the particles indexed by $C$ are separated from the remaining particles:
\begin{equation}
|x_i^t-x_k^t|\ge \eta,
\qquad i\in C,\quad k\in[N]\setminus C,
\end{equation}
for some $\eta>0$, with the convention that this condition is void if $C$ is the whole index set.

To quantify the collapse of the particles indexed by $C$, define their centered quadratic spread by
\begin{equation}
\bar x_C^t:=\frac1m\sum_{i\in C}x_i^t,
\qquad
S_C(t):=\sum_{i\in C}|x_i^t-\bar x_C^t|^2.
\end{equation}
Then $S_C(t)>0$ for $t<T_*$ and $S_C(t)\to0$ as $t\uparrow T_*$.
For $i\in C$, decompose
\begin{equation}
\dot x_i^t=I_i(t)+E_i(t),
\end{equation}
where
\begin{equation}
I_i(t):=-\frac1N\sum_{\substack{j\in C\\ j\neq i}}\nabla K(x_i^t-x_j^t)
\end{equation}
and
\begin{equation}
E_i(t):=-\frac1N\sum_{k\notin C}\nabla K(x_i^t-x_k^t)+\nabla K*\mu(x_i^t).
\end{equation}
The internal forces have zero total sum.  Since $\sum_{i\in C}(x_i^t-\bar x_C^t)=0$,
\begin{equation}
\frac{\ud}{\ud t}S_C(t)
=2\sum_{i\in C}(x_i^t-\bar x_C^t)\cdot(I_i(t)+E_i(t)).
\end{equation}
Using $-\nabla K(z)=qz|z|^{q-2}$ for $z\neq0$, pair symmetrization gives the positive internal virial identity
\begin{equation}
2\sum_{i\in C}(x_i^t-\bar x_C^t)\cdot I_i(t)
=\frac{2q}{N}\sum_{\substack{i,j\in C\\ i<j}}|x_i^t-x_j^t|^q.
\end{equation}

	It remains to bound the external contribution.  Since the particles in $C$ remain separated from those outside $C$, the outside-particle field
\begin{equation}
x\mapsto -\frac1N\sum_{k\notin C}\nabla K(x-x_k^t)
\end{equation}
is Lipschitz on a fixed neighborhood of $x_*$, uniformly for all $t$ sufficiently close to $T_*$.  After subtracting its value at $\bar x_C^t$ and using $\sum_{i\in C}(x_i^t-\bar x_C^t)=0$, its contribution is bounded by $L S_C(t)$.

In cases~\textup{(i)} and \textup{(ii)}, the proof of \Cref{prop:particle-ode-local} shows that $\nabla K*\mu$ is globally Lipschitz, so its contribution has the same bound.  In case~\textup{(iii)}, we instead use $K'*\mu\in C_b(\R)$ and Cauchy--Schwarz to obtain
\begin{equation}
\left|\sum_{i\in C}(x_i^t-\bar x_C^t)(K'*\mu)(x_i^t)\right|
\le \sqrt m\,\|K'*\mu\|_{L^\infty}S_C(t)^{1/2}.
\end{equation}
Consequently, in all three cases there are constants $L,B<\infty$, with $B=0$ in cases~\textup{(i)} and \textup{(ii)}, such that
\begin{equation}\label{eq:particle-external-spread-bound}
\left|\sum_{i\in C}(x_i^t-\bar x_C^t)\cdot E_i(t)\right|
\le L S_C(t)+B S_C(t)^{1/2}
\end{equation}
for all $t$ sufficiently close to $T_*$.  Therefore

\begin{equation}
\frac{\ud}{\ud t}S_C(t)
\ge
\frac{2q}{N}\sum_{\substack{i,j\in C\\ i<j}}|x_i^t-x_j^t|^q
-2B S_C(t)^{1/2}-2L S_C(t).
\end{equation}
The identity
\begin{equation}
\sum_{\substack{i,j\in C\\ i<j}}|x_i^t-x_j^t|^2
=mS_C(t)
\end{equation}
identifies $S_C$ with the normalized pairwise quadratic spread (cf.~\cite[proof of Theorem~2.1]{carrillo2017sharp}). Since $0<q<2$, the monotonicity of $\ell^p$ norms gives
\begin{equation}
\sum_{\substack{i,j\in C\\ i<j}}|x_i^t-x_j^t|^q
\ge
\left(\sum_{\substack{i,j\in C\\ i<j}}|x_i^t-x_j^t|^2\right)^{q/2}
=(mS_C(t))^{q/2}.
\end{equation}
Thus, for all $t$ sufficiently close to $T_*$,
\begin{equation}
\frac{\ud}{\ud t}S_C(t)
\ge \frac{2q}{N}m^{q/2}S_C(t)^{q/2}-2B S_C(t)^{1/2}-2L S_C(t).
\end{equation}
In cases~\textup{(i)} and \textup{(ii)}, one has $B=0$ and $q/2<1$; in case~\textup{(iii)}, one has $q/2<1/2$.  Thus in every case the positive term dominates as $S_C(t)\downarrow0$.  Hence there exists $\delta>0$ such that
\begin{equation}
0<S_C(t)\le\delta
\quad\Longrightarrow\quad
\frac{\ud}{\ud t}S_C(t)>0
\end{equation}
for all $t$ sufficiently close to $T_*$.  But $S_C(t)\to0$ as $t\uparrow T_*$, so $S_C(t)\le\delta$ eventually, and the preceding strict monotonicity contradicts convergence to zero while $S_C(t)>0$ for $t<T_*$.  This contradiction rules out $T_*<\infty$.  Thus the unique solution is global in cases~\textup{(i)} and \textup{(ii)}, while every maximal-lifespan solution is global in case~\textup{(iii)}; since at least one such solution exists, the global-existence and noncollision assertions follow.

\emph{Long-time behavior.}
The energy calculation above gives
\begin{equation}\label{eq:particle-infinite-time-dissipation}
\mathcal E_N(X_N^t)\downarrow\mathcal E_N^\infty\ge0,
\qquad
\int_0^\infty\sum_{i=1}^N|\dot x_i^t|^2\ud t
=
N\left(\mathcal E_N(X_N^0)-\mathcal E_N^\infty\right)<\infty.
\end{equation}
Choose $0<r<q/2$ and set
\begin{equation}
\rho_t^N:=\frac1N\sum_{i=1}^N\delta_{x_i^t}.
\end{equation}
Since $\mathcal E_N(X_N^t)=\mmd_q^2(\rho_t^N,\mu)\le\mathcal E_N(X_N^0)$, \Cref{prop:mmdtomoment} yields
\begin{equation}
\sup_{t\ge0}\mathcal M_r(\rho_t^N)
\le
\mathcal M_r(\mu)
+C_{d,q,r}\mathcal E_N(X_N^0)^{r/q}<\infty.
\end{equation}
Consequently,
\begin{equation}
\sup_{t\ge0}\max_{i\in[N]}|x_i^t|^r
\le
N\sup_{t\ge0}\mathcal M_r(\rho_t^N)<\infty,
\end{equation}
which proves the first assertion in \eqref{eq:particle-uniform-boundedness-separation}.

We next rule out asymptotic collisions.  Suppose, to the contrary, that there are $t_n\to\infty$ for which the minimum interparticle distance tends to zero.  By the preceding boundedness, after passing to a subsequence,
\begin{equation}
X_N^{t_n}\longrightarrow X_N^*\in(\R^d)^N,
\end{equation}
where at least two components of $X_N^*$ coincide.  Let $C\subset[N]$ be the set of all indices whose components converge to one such common point, let $m:=|C|\ge2$, and define
\begin{equation}
\bar x_C^t:=\frac1m\sum_{i\in C}x_i^t,
\qquad
S_C(t):=\sum_{i\in C}|x_i^t-\bar x_C^t|^2.
\end{equation}
Thus $S_C(t_n)\to0$.  Since no component of $X_N^*$ indexed by $C$ coincides with a component indexed by $[N]\setminus C$, there exist $\varepsilon,\eta>0$ such that, with
\begin{equation}
\mathcal U:=B_\varepsilon(X_N^*)\subset(\R^d)^N,
\end{equation}
every $X_N=(x_i)_{i\in[N]}\in\mathcal U$ satisfies $|x_i-x_k|\ge\eta$ for $i\in C$ and $k\in[N]\setminus C$, with this condition understood as void if $C=[N]$.

The preceding cluster-spread computation therefore gives constants $a>0$, $L<\infty$, and $B<\infty$, with $B=0$ in cases~\textup{(i)}--\textup{(ii)}, such that
\begin{equation}\label{eq:particle-asymptotic-spread}
X_N^t\in\mathcal U
\quad\Longrightarrow\quad
\frac{\ud}{\ud t}S_C(t)
\ge
aS_C(t)^{q/2}-2B S_C(t)^{1/2}-2L S_C(t).
\end{equation}
In cases~\textup{(i)}--\textup{(ii)}, $q/2<1$ and $B=0$; in case~\textup{(iii)}, $q/2<1/2$.  Hence there exist $c,\delta>0$ such that
\begin{equation}\label{eq:particle-asymptotic-spread-positive}
X_N^t\in\mathcal U,\qquad 0<S_C(t)\le\delta
\quad\Longrightarrow\quad
\frac{\ud}{\ud t}S_C(t)\ge cS_C(t)^{q/2}.
\end{equation}

On the other hand, \eqref{eq:particle-infinite-time-dissipation} and Cauchy--Schwarz imply, for every fixed $T>0$,
\begin{equation}\label{eq:particle-tail-displacement}
\sup_{0\le h\le T}|X_N^{t_n+h}-X_N^{t_n}|
\le
\sqrt T
\left(
\int_{t_n}^\infty\sum_{i=1}^N|\dot x_i^t|^2\ud t
\right)^{1/2}
\longrightarrow0.
\end{equation}
Thus, for $n$ sufficiently large, $X_N^t\in\mathcal U$ for every $t\in[t_n,t_n+T]$.  Integrating \eqref{eq:particle-asymptotic-spread-positive} until either time $t_n+T$ or the first time at which $S_C=\delta$ shows that, for some $h_n\in[0,T]$,
\begin{equation}
S_C(t_n+h_n)
\ge
\min\left\{
\delta,\,
\left(c(1-q/2)T\right)^{1/(1-q/2)}
\right\}>0.
\end{equation}
This contradicts \eqref{eq:particle-tail-displacement}, since $X_N^{t_n}\to X_N^*$ and the components indexed by $C$ coincide in $X_N^*$.  Since collision-free continuity gives a positive minimum separation on every bounded time interval, the second assertion in \eqref{eq:particle-uniform-boundedness-separation} follows.

The orbit closure is therefore a compact subset of $\mathcal D_N$.  Let $b$ denote the particle vector field in \eqref{eq:particle-vector-field-local}.  Since $b$ is continuous on $\mathcal D_N$, it is bounded and uniformly continuous on the orbit closure.  Hence $X_N:[0,\infty)\to(\R^d)^N$ is globally Lipschitz, and the map $t\mapsto\dot X_N^t=b(X_N^t)$ is uniformly continuous on $[0,\infty)$.  The nonnegative function
\begin{equation}
t\longmapsto|\dot X_N^t|^2=\sum_{i=1}^N|\dot x_i^t|^2
\end{equation}
is therefore uniformly continuous and, by \eqref{eq:particle-infinite-time-dissipation}, integrable on $[0,\infty)$.  It follows that $\dot X_N^t\to0$ as $t\to\infty$.

Now let $t_n\to\infty$.  Compactness gives, after passage to a subsequence, $X_N^{t_{n_k}}\to X_N^\infty\in\mathcal D_N$.  By continuity of $b$ and the preceding velocity decay,
\begin{equation}
b(X_N^\infty)
=
\lim_{k\to\infty}b(X_N^{t_{n_k}})
=
\lim_{k\to\infty}\dot X_N^{t_{n_k}}
=0.
\end{equation}
Since $b=-N\nabla\mathcal E_N$ on $\mathcal D_N$, every subsequential limit is a collision-free critical configuration.

Finally, for each $T\ge0$, the tail orbit $\{X_N^t:t\ge T\}$ is the continuous image of the interval $[T,\infty)$ and is therefore path-connected; hence its closure is connected.  These closures,
\begin{equation}
\overline{\{X_N^t:t\ge T\}},
\qquad T\ge0,
\end{equation}
are nonempty, compact, and nested.  By \eqref{eq:particle-omega-limit}, their intersection is $\omega(X_N)$.  The nested-intersection theorem for compact connected sets shows that $\omega(X_N)$ is nonempty, compact, and connected.  Every one of its elements is critical by the preceding argument, and continuity of $\mathcal E_N$ places it on the limiting energy level $\mathcal E_N^\infty$.  This proves \eqref{eq:particle-omega-critical-level}.  The distance conclusion \eqref{eq:particle-distance-critical-set} follows by contradiction from compactness and \eqref{eq:particle-omega-critical-level}.

\end{proof}

We next give sufficient conditions under which a narrow limit of particle empirical measures with vanishing mean-square particle vector field is a Lagrangian critical point of the continuum/mean-field energy.
\begin{proposition}[Particle-to-continuum criticality]\label{prop:particle-critical-closure}\label{prop:sublinear-particle-critical-closure}
Assume the hypotheses on $d,q,p,\mu$ of \Cref{thm:particle-global-noncollision}.  For each $N\ge2$, let
\begin{equation}
X_N=(x_i^N)_{i\in[N]}\in\mathcal D_N,
\qquad
\rho^N:=\frac1N\sum_{i=1}^N\delta_{x_i^N},
\end{equation}
and let $b_i(X_N)$ be the  particle vector field in \eqref{eq:particle-vector-field-local}.  Suppose that, for some $\rho\in\mathcal P(\R^d)$,
\begin{equation}\label{eq:sublinear-particle-closure-narrow-residual}
\rho^N\rightharpoonup\rho
\quad\text{narrowly},
\qquad
\frac1N\sum_{i=1}^N|b_i(X_N)|^2\longrightarrow0.
\end{equation}
When $0<q<1$, suppose in addition that\footnote{No additional hypothesis is imposed when $q=1$.}
\begin{equation}\label{eq:sublinear-particle-riesz-bound}
\mathfrak R_q:=\sup_{N\ge2}\frac1{N^2}
\sum_{\substack{i,j\in[N]\\i\ne j}}
\frac1{|x_i^N-x_j^N|^{1-q}}<\infty.
\end{equation}
When $1<q<2$, suppose instead that, for some $r>q-1$,
\begin{equation}\label{eq:particle-closure-superlinear-moment}
\mathfrak M_r:=\sup_{N\ge2}\int_{\R^d}|x|^r\ud\rho^N(x)<\infty.
\end{equation}
Then $\rho$ is a Lagrangian critical point of the MMD energy with target $\mu$ in the sense of \Cref{critical:def:Lcrit}.
\end{proposition}

\begin{proof}
Set $F_\mu:=\nabla K*\mu$.  When $0<q<1$, the hypotheses $p>p_c$ for $d\ge2$ and $p>1/q$ for $d=1$ allow us to apply \Cref{lem:singular-convolution-bound} with $a=1-q$, giving
\begin{equation}\label{eq:sublinear-target-absolute-force-bound}
\sup_{x\in\R^d}\int_{\R^d}|\nabla K(x-y)|\ud\mu(y)<\infty.
\end{equation}
When $d\ge2$, \eqref{eq:target-field-global-lipschitz} gives $\nabla F_\mu=\nabla^{\otimes2}K*\mu\in L^\infty(\R^d)$, so together with \eqref{eq:sublinear-target-absolute-force-bound} we obtain $F_\mu\in W^{1,\infty}(\R^d;\R^d)$.  In dimension one, when $0<q<1$, set $\alpha:=q-1/p>0$.  The singular-convolution bound gives $F_\mu\in L^\infty(\R)$, while H\"older's inequality and the homogeneity of $K'$ give
\begin{equation}\label{eq:one-dimensional-target-field-holder}
\|F_\mu(\cdot+h)-F_\mu\|_{L^\infty}
\le C_{q,p}\|\mu\|_{L^p}|h|^\alpha.
\end{equation}
Thus $F_\mu\in C_b^\alpha(\R)$.

Suppose first that $0<q<1$.  For $L>0$, define the bounded continuous function
\begin{equation}
h_L(z):=
\begin{cases}
\min\{L,|z|^{-(1-q)}\},&z\ne0,\\
L,&z=0.
\end{cases}
\end{equation}
Then
\begin{equation}
\iint_{(\R^d)^2}h_L(x-y)\ud\rho^N(x)\ud\rho^N(y)
\le \mathfrak R_q+\frac LN.
\end{equation}
Letting $N\to\infty$ first and appealing to narrow convergence of product measures, then letting $L\to\infty$ and appealing to monotone convergence, we obtain
\begin{equation}\label{eq:sublinear-limit-riesz-energy}
\iint_{(\R^d)^2}\frac{\ud\rho(x)\ud\rho(y)}{|x-y|^{1-q}}
\le \mathfrak R_q.
\end{equation}
In particular, $\rho$ has no atoms.

Fix $\varphi\in C_c^1(\R^d;\R^d)$ and set
\begin{equation}
H_\varphi(x,y):=
\begin{cases}
(\varphi(x)-\varphi(y))\cdot\nabla K(x-y),&x\ne y,\\
0,&x=y.
\end{cases}
\end{equation}
For $0<q\le1$, the function $H_\varphi$ is bounded and continuous on $(\R^d)^2$: near the diagonal it is $O(|x-y|^q)$, while away from the diagonal this follows from the compact support of $\varphi$ and the boundedness of $|\nabla K(z)|=q|z|^{q-1}$ for $|z|\ge1$.  Pair symmetrization in \eqref{eq:particle-vector-field-local} yields
\begin{equation}\label{eq:sublinear-discrete-symmetrized-criticality}
\frac1N\sum_{i=1}^N\varphi(x_i^N)\cdot b_i(X_N)
=-\frac12\iint_{(\R^d)^2}H_\varphi(x,y)\ud\rho^N(x)\ud\rho^N(y)
+\int_{\R^d}\varphi\cdot F_\mu\ud\rho^N.
\end{equation}

If $q=1$, the field $F_\mu$ is bounded and uniformly continuous, since translation continuity in $L^1$ gives
\begin{equation}
\|F_\mu(\cdot+h)-F_\mu\|_{L^\infty}
\le\|\nabla K\|_{L^\infty}\|\mu(\cdot+h)-\mu\|_{L^1}
\longrightarrow0.
\end{equation}
If $1<q<2$, then $\nabla K$ is continuous and has growth of order $q-1$, so $F_\mu$ is continuous and
\begin{equation}
|H_\varphi(x,y)|
\le C_\varphi\bigl(1+|x|^{q-1}+|y|^{q-1}\bigr).
\end{equation}
The bound $\mathfrak M_r<\infty$, with $r>q-1$, makes the right-hand side uniformly integrable with respect to $\rho^N\otimes\rho^N$.  Truncation outside large balls and narrow convergence of the product measures therefore give
\begin{equation}\label{eq:particle-superlinear-symmetrized-limit}
\iint_{(\R^d)^2}H_\varphi\ud(\rho^N\otimes\rho^N)
\longrightarrow
\iint_{(\R^d)^2}H_\varphi\ud(\rho\otimes\rho).
\end{equation}
For $0<q\le1$, the same convergence follows directly from bounded continuity of $H_\varphi$.  In every case, $F_\mu$ is continuous and locally bounded, so the compact support of $\varphi$ gives
\begin{equation}
\int_{\R^d}\varphi\cdot F_\mu\ud\rho^N
\longrightarrow
\int_{\R^d}\varphi\cdot F_\mu\ud\rho.
\end{equation}
The left-hand side of \eqref{eq:sublinear-discrete-symmetrized-criticality} tends to zero by Cauchy--Schwarz and \eqref{eq:sublinear-particle-closure-narrow-residual}.  Passing to the limit, we obtain
\begin{equation}\label{eq:sublinear-limit-symmetrized-criticality}
\frac12\iint_{(\R^d)^2}H_\varphi(x,y)\ud\rho(x)\ud\rho(y)
-\int_{\R^d}\varphi\cdot F_\mu\ud\rho=0.
\end{equation}

When $0<q<1$, \eqref{eq:sublinear-limit-riesz-energy} gives
\begin{equation}
\iint_{(\R^d)^2}|\nabla K(x-y)|\ud\rho(x)\ud\rho(y)
=q\iint_{(\R^d)^2}\frac{\ud\rho(x)\ud\rho(y)}{|x-y|^{1-q}}<\infty.
\end{equation}
The target contribution is absolutely integrable by \eqref{eq:sublinear-target-absolute-force-bound}.  When $q=1$, both contributions are absolutely integrable because $|\nabla K|\le1$.  Finally, when $1<q<2$, lower semicontinuity of $r$-moments under narrow convergence and \eqref{eq:particle-closure-superlinear-moment} give $\rho\in\mathcal P_r(\R^d)$, while $\mu\in\mathcal P_q(\R^d)$; hence
\begin{equation}
\iint_{(\R^d)^2}|\nabla K(x-y)|\ud\rho(x)\ud(\rho+\mu)(y)<\infty.
\end{equation}
Thus, in every case, the field
\begin{equation}
g_\rho(x):=\int_{\R^d}\nabla K(x-y)\ud(\rho-\mu)(y)
\end{equation}
is absolutely convergent for $\rho$-a.e. $x$ and belongs to $L^1(\rho)$.  Desymmetrizing the first term in \eqref{eq:sublinear-limit-symmetrized-criticality} shows that
\begin{equation}
\int_{\R^d}\varphi\cdot g_\rho\ud\rho=0
\qquad\text{for every }\varphi\in C_c^1(\R^d;\R^d).
\end{equation}
Hence the finite vector measure $g_\rho\rho$ vanishes, so $g_\rho=0$ $\rho$-a.e.  This is precisely the Lagrangian criticality condition in \Cref{critical:def:Lcrit}.
\end{proof}

The next two criteria concern the additional discrete Riesz bound required only when $0<q<1$. Between them, we record a virial estimate that supplies tightness for compactly supported targets throughout the full range $0<q<2$.

\begin{lemma}[An averaged Frostman criterion]\label{lem:sublinear-frostman-riesz}
Let $0<q<1$, and let $X_N=(x_i^N)_{i\in[N]}\in\mathcal D_N$.  Suppose that there are $C<\infty$ and $\alpha>1-q$ such that, for every $N\ge2$ and $0<r\le1$,
\begin{equation}\label{eq:sublinear-averaged-frostman}
\frac1{N^2}
\#\left\{(i,j)\in[N]^2:i\ne j,\ |x_i^N-x_j^N|\le r\right\}
\le Cr^\alpha.
\end{equation}
Then \eqref{eq:sublinear-particle-riesz-bound} holds.
\end{lemma}

\begin{proof}
The layer-cake formula and the bound $|x_i^N-x_j^N|^{-(1-q)}\le1$ when $|x_i^N-x_j^N|\ge1$ give
\begin{align}
\frac1{N^2}\sum_{\substack{i,j\in[N]\\i\ne j}}
\frac1{|x_i^N-x_j^N|^{1-q}}
&\le
1+(1-q)\int_0^1 r^{-(1-q)-1}
\frac1{N^2}\#\left\{(i,j)\in[N]^2:i\ne j,\ |x_i^N-x_j^N|\le r\right\}\ud r\notag\\
&\le1+\frac{(1-q)C}{\alpha-(1-q)}.
\end{align}
Since $\alpha>1-q$, the last integral is finite, which completes the proof.
\end{proof}

\begin{lemma}[Uniform energy bound for particle critical configurations]\label{lem:particle-critical-virial}
Assume the hypotheses on $d,q,p,\mu$ of \Cref{thm:particle-global-noncollision}, and suppose that $\mu$ is compactly supported.  Then
\begin{equation}\label{eq:compact-target-critical-energy-bound}
\sup_{\substack{N\ge2\\X_N\in\operatorname{Crit}_{\mathcal D_N}\mathcal E_N}}
\mathcal E_N(X_N)<\infty.
\end{equation}
Consequently, the empirical measures associated with any sequence of particle critical configurations are tight.
\end{lemma}

\begin{proof}
Choose $a\in\R^d$ and $R>0$ such that $\supp\mu\subset B_R(a)$, and set
\begin{equation}
C_\mu:=\iint_{(\R^d)^2}|x-y|^q\ud\mu(x)\ud\mu(y),
\qquad
\Psi_{\mu,a}(x):=\int_{\R^d}|y-a||x-y|^{q-1}\ud\mu(y).
\end{equation}
Fix $N\ge2$ and $X_N\in\operatorname{Crit}_{\mathcal D_N}\mathcal E_N$.
Since $b_i(X_N)=0$, multiplying the particle equations by $x_i^N-a$, summing in $i$, and symmetrizing the particle-particle term give
\begin{equation}
\frac12\iint_{(\R^d)^2}|x-y|^q\ud\rho^N(x)\ud\rho^N(y)
=\iint_{\R^d\times\R^d}
(x-a)\cdot(x-y)|x-y|^{q-2}\ud\rho^N(x)\ud\mu(y).
\end{equation}
Writing $x-a=(x-y)+(y-a)$ in the right-hand side and using
\begin{equation}
\mathcal E_N(X_N)
=\iint_{\R^d\times\R^d}|x-y|^q\ud\rho^N(x)\ud\mu(y)
-\frac12\iint_{(\R^d)^2}|x-y|^q\ud\rho^N(x)\ud\rho^N(y)
-\frac12C_\mu
\end{equation}
we find that
\begin{align}
0\le\mathcal E_N(X_N)+\frac12C_\mu
&=-\iint_{\R^d\times\R^d}
(y-a)\cdot(x-y)|x-y|^{q-2}\ud\rho^N(x)\ud\mu(y)
\le\int_{\R^d}\Psi_{\mu,a}\ud\rho^N.
\label{eq:particle-critical-virial-bound}
\end{align}

If $0<q\le1$, then
\begin{equation}
\Psi_{\mu,a}(x)
\le R\int_{\R^d}\frac{\ud\mu(y)}{|x-y|^{1-q}}.
\end{equation}
The right-hand side is uniformly bounded by \Cref{lem:singular-convolution-bound}: in dimensions $d\ge2$, the stronger inequality $(2-q)p_*<d$ follows from $p>p_c$, while in dimension one the needed inequality $(1-q)p_*<1$ is equivalent to $p>1/q$.  These observations prove \eqref{eq:compact-target-critical-energy-bound} for $0<q\le1$.

Now let $1<q<2$.  Since $\supp\mu\subset B_R(a)$,
\begin{equation}
\Psi_{\mu,a}(x)\le C_{\mu,a,R,q}\bigl(1+|x|^{q-1}\bigr).
\end{equation}
Using \eqref{eq:particle-critical-virial-bound}, \Cref{prop:mmdtomoment} with exponent $q-1<q/2$, and $\mathcal E_N(X_N)=\mmd_q^2(\rho^N,\mu)$ gives
\begin{equation}
\mathcal E_N(X_N)
\le C_{\mu,a,R,q}\left(1+\mathcal E_N(X_N)^{(q-1)/q}\right).
\end{equation}
Because $(q-1)/q<1$, this proves \eqref{eq:compact-target-critical-energy-bound} also for $1<q<2$.  Finally, for any $0<r<q/2$, \Cref{prop:mmdtomoment} and \eqref{eq:particle-energy-definition} give a uniform $r$-moment bound, hence tightness.
\end{proof}

\begin{proposition}[Discrete Riesz control from a Laplacian bound]\label{prop:particle-hessian-nonconcentration}
Let $d\ge2$, $0<q<1$, $p_c<p\le\infty$, and $\mu\in\mathcal P_1(\R^d)\cap L^p(\R^d)$.  Set
\begin{equation}
U_{\mu,2-q}(x):=\int_{\R^d}|x-y|^{q-2}\ud\mu(y).
\end{equation}
Then $U_{\mu,2-q}\in L^\infty(\R^d)$.  Writing $\Delta_{X_N}:=\sum_{i=1}^N\Delta_{x_i}$ for the configuration-space Laplacian, for every $X_N=(x_i^N)_{i\in[N]}\in\mathcal D_N$ one has
\begin{equation}\label{eq:particle-hessian-trace}
\Delta_{X_N}\mathcal E_N(X_N)
=q(d+q-2)\left(
\frac1N\sum_{i=1}^NU_{\mu,2-q}(x_i^N)
-\frac1{N^2}\sum_{\substack{i,j\in[N]\\i\ne j}}
\frac1{|x_i^N-x_j^N|^{2-q}}
\right).
\end{equation}
Consequently, if a sequence $(X_N)$ satisfies
\begin{equation}\label{eq:particle-hessian-uniform-lower-bound}
\inf_{N\ge2}\Delta_{X_N}\mathcal E_N(X_N)>-\infty,
\end{equation}
then it satisfies \eqref{eq:sublinear-particle-riesz-bound}.  In particular, this conclusion holds for any sequence of local minimizers of $\mathcal E_N$ in $\mathcal D_N$.  If, in addition, $\mu$ is compactly supported, their empirical measures are tight and every narrow accumulation point is a Lagrangian critical point.
\end{proposition}

\begin{proof}
The bound on $U_{\mu,2-q}$ follows from \Cref{lem:singular-convolution-bound} and $(2-q)p_*<d$.  The same near--far decomposition, together with translation continuity in $L^{p_*}$ for the near-field kernel, gives $D^2K*\mu\in C_b(\R^d)$.  Thus $\mathcal E_N\in C^2(\mathcal D_N)$.  Differentiating \eqref{eq:particle-energy-definition} on $\mathcal D_N$ and using $\Delta K(z)=-q(d+q-2)|z|^{q-2}$ gives \eqref{eq:particle-hessian-trace}.  The pair-interaction term is negative, while the target term is positive.  Since $q(d+q-2)>0$, it follows that \eqref{eq:particle-hessian-uniform-lower-bound} and the uniform bound on $U_{\mu,2-q}$ control the discrete Riesz sum with inverse-power exponent $2-q$.  The elementary inequality
\begin{equation}
\frac1{r^{1-q}}\le1+\frac1{r^{2-q}},
\qquad r>0,
\end{equation}
then gives \eqref{eq:sublinear-particle-riesz-bound}.  A local minimizer is critical and has nonnegative configuration-space Laplacian.  Under compact support of the target, \Cref{lem:particle-critical-virial} supplies tightness, and \Cref{prop:particle-critical-closure} identifies every accumulation point.
\end{proof}

The preceding closure and compactness results yield the following conditional passage from fixed-$N$ particle $\omega$-limits to continuum/mean-field Lagrangian critical points.

\begin{corollary}[Mean-field limits of particle $\omega$-limits]\label{cor:particle-late-time-critical-limit}\label{cor:sublinear-particle-late-time-critical-limit}
Assume the hypotheses on $d,q,p,\mu$ of \Cref{prop:particle-critical-closure}.  For every $N\ge2$, let $X_N^0\in\mathcal D_N$, let $X_N^t$ be any corresponding global collision-free solution, and choose
\begin{equation}
Y_N=(y_i^N)_{i\in[N]}\in\omega(X_N),
\qquad
\bar\rho^N:=\frac1N\sum_{i=1}^N\delta_{y_i^N}.
\end{equation}
Suppose that
\begin{align}
\sup_{N\ge2}\mathcal E_N(Y_N)&<\infty,
\label{eq:sublinear-particle-omega-uniform-bounds}\\
\sup_{N\ge2}\frac1{N^2}
\sum_{\substack{i,j\in[N]\\i\ne j}}
\frac1{|y_i^N-y_j^N|^{1-q}}
&<\infty
\qquad (0<q<1).
\label{eq:sublinear-particle-omega-riesz-bound}
\end{align}
Then there are a subsequence $N_k\to\infty$, times $t_k\to\infty$, and a Lagrangian critical point $\bar\rho$ of the MMD energy with target $\mu$ such that
\begin{equation}\label{eq:sublinear-particle-late-time-convergence}
\rho_{t_k}^{N_k}\rightharpoonup\bar\rho
\qquad\text{narrowly}.
\end{equation}
The energy bound \eqref{eq:sublinear-particle-omega-uniform-bounds} is automatic if $\sup_{N\ge2}\mathcal E_N(X_N^0)<\infty$ or if $\mu$ is compactly supported.
\end{corollary}

\begin{proof}
By \Cref{thm:particle-asymptotic-criticality}, each $Y_N$ is a collision-free critical configuration.  The energy bound \eqref{eq:sublinear-particle-omega-uniform-bounds}, \Cref{prop:mmdtomoment}, and \eqref{eq:particle-energy-definition} imply that $(\bar\rho^N)$ is tight.  After passing to a subsequence, $\bar\rho^{N_k}\rightharpoonup\bar\rho$ narrowly.

If $1<q<2$, choose $r$ such that $q-1<r<q/2$, which is possible because $q<2$.  Another application of \Cref{prop:mmdtomoment} gives
\begin{equation}
\sup_{N\ge2}\int_{\R^d}|x|^r\ud\bar\rho^N(x)
\le\mathcal M_r(\mu)
+C_{d,q,r}\left(\sup_{N\ge2}\mathcal E_N(Y_N)\right)^{r/q}
<\infty.
\end{equation}
Thus \eqref{eq:particle-closure-superlinear-moment} holds.  When $0<q<1$, \eqref{eq:sublinear-particle-omega-riesz-bound} is exactly \eqref{eq:sublinear-particle-riesz-bound}, while no additional condition is needed when $q=1$.  Since $Y_{N_k}$ is critical, $b_i(Y_{N_k})=0$ for every $i\in[N_k]$, so the corresponding condition in \eqref{eq:sublinear-particle-closure-narrow-residual} holds automatically.  Hence \Cref{prop:particle-critical-closure} shows that $\bar\rho$ is Lagrangian critical.

Since $Y_{N_k}\in\omega(X_{N_k})$, we may choose $t_k\ge k$ so that
\begin{equation}
|X_{N_k}^{t_k}-Y_{N_k}|\le\frac1k.
\end{equation}
The associated empirical measures therefore have the same narrow limit, which proves \eqref{eq:sublinear-particle-late-time-convergence}.  If the initial energies are uniformly bounded, energy dissipation gives \eqref{eq:sublinear-particle-omega-uniform-bounds}; if $\mu$ is compactly supported, the same bound follows from \Cref{lem:particle-critical-virial}.  Thus, for compactly supported targets, no additional hypothesis is needed when $1\le q<2$, whereas for $0<q<1$ only the discrete Riesz bound \eqref{eq:sublinear-particle-omega-riesz-bound} remains.
\end{proof}

\begin{remark}[Scope of the particle criticality results]\label{rem:sublinear-particle-dimension-one}
For $1\le q<2$, \Cref{prop:particle-critical-closure,cor:particle-late-time-critical-limit} require no discrete nonconcentration hypothesis.  Their conclusion is nevertheless only Lagrangian criticality, not identification of $\bar\rho$ with $\mu$: the limit may be singular, whereas the rigidity results below generally assume absolute continuity.  The collision-free saddle configurations in \Cref{prop:particle-saddle-configurations} likewise preclude an unconditional conclusion that every deterministic particle trajectory converges to a global empirical minimizer.

The $0<q<1$ cases of \Cref{prop:particle-critical-closure,lem:particle-critical-virial,cor:particle-late-time-critical-limit}, together with \Cref{lem:sublinear-frostman-riesz,prop:particle-hessian-nonconcentration}, include $d=1$ under the hypothesis $p>1/q$.  This is consistent with the nested-interval examples in \Cref{critical:thm:nested-intervals}: their conclusion is itself Lagrangian criticality, and their target densities belong only to $L^p$ for $p<2/(1+q)<1/q$.  Those examples show that unrestricted one-dimensional rigidity fails, but they do not settle target identification under the stronger regular-target hypothesis used here.

The restriction $d\ge2$ in \Cref{prop:particle-hessian-nonconcentration} is instead intrinsic to that particular result.  In dimension one, $d+q-2=q-1<0$, so the sign in \eqref{eq:particle-hessian-trace} reverses; moreover, the target potential with kernel $|x-y|^{q-2}$ is not controlled under $p>1/q$.  Thus a lower bound on the Laplacian does not provide the needed discrete Riesz estimate in that regime.
\end{remark}

\begin{corollary}[Uniform-in-time consistency for well-prepared particle data]\label{cor:well-prepared-particle-uniform-time}
Assume the hypotheses on $d,q,p,\mu$ of \Cref{thm:particle-global-noncollision}.  For each $N\ge2$, let $X_N^t$ be any global collision-free solution with $X_N^0\in\mathcal D_N$.
If
\begin{equation}\label{eq:well-prepared-particle-data}
\mathcal E_N(X_N^0)\longrightarrow0
\qquad\text{as }N\to\infty,
\end{equation}
then
\begin{equation}\label{eq:well-prepared-uniform-mmd}
\sup_{t\ge0}\mmd_q^2(\rho_t^N,\mu)\longrightarrow0.
\end{equation}
Consequently, for every sequence $t_N\ge0$,
\begin{equation}\label{eq:well-prepared-arbitrary-time-narrow}
\rho_{t_N}^N\rightharpoonup\mu
\qquad\text{narrowly as }N\to\infty.
\end{equation}
\end{corollary}

\begin{proof}
By \eqref{eq:particle-energy-definition} and \eqref{eq:particle-infinite-time-dissipation},
\begin{equation}
0\le \mmd_q^2(\rho_t^N,\mu)
=\mathcal E_N(X_N^t)
\le\mathcal E_N(X_N^0),
\qquad t\ge0.
\end{equation}
This proves \eqref{eq:well-prepared-uniform-mmd}.  For any sequence $t_N\ge0$, it follows that $\mmd_q(\rho_{t_N}^N,\mu)\to0$.  Since $\rho_{t_N}^N,\mu\in\mathcal P_{q/2}(\R^d)$, \Cref{prop:mmd-topology} gives
\begin{equation}
\mathsf T_s(\rho_{t_N}^N,\mu)\longrightarrow0
\qquad\text{for every }0<s<q/2,
\end{equation}
and hence \eqref{eq:well-prepared-arbitrary-time-narrow}.  In fact, the same argument by contradiction over the choice of $t_N$ shows that
\begin{equation}
\sup_{t\ge0}\mathsf T_s(\rho_t^N,\mu)\longrightarrow0
\qquad\text{for every }0<s<q/2.
\end{equation}
\end{proof}

\begin{remark}[Convergence to one critical configuration]\label{rem:particle-point-convergence}
\Cref{thm:particle-asymptotic-criticality} does not by itself imply that $X_N^t$ converges as $t\to\infty$: compactness gives subsequential limits, but these limits need not a priori coincide.  If
\begin{equation}
\operatorname{Crit}_{\mathcal D_N}\mathcal E_N
\cap\{\mathcal E_N=\mathcal E_N^\infty\}
\cap\overline{\{X_N^t:t\ge0\}}
\end{equation}
is totally disconnected (in particular, if it is finite), then the connectedness of $\omega(X_N)$ forces it to be a singleton, and $X_N^t$ converges to one collision-free critical configuration.  The same conclusion follows if $\mathcal E_N$ satisfies a Kurdyka--\L ojasiewicz inequality in a neighborhood of $\omega(X_N)$: such an inequality yields finite trajectory length.\footnote{Schematically, set $e(X):=\mathcal E_N(X)-\mathcal E_N^\infty$.  The condition asks for $\eta>0$ and an increasing $C^1$ desingularizing function $\varphi$, with $\varphi(0)=0$, such that $\left|\nabla(\varphi\circ e)(X)\right|=\varphi'(e(X))|\nabla\mathcal E_N(X)|\ge1$ for $X$ near $\omega(X_N)$ with $0<e(X)<\eta$.  Along $\dot X_N=-N\nabla\mathcal E_N(X_N)$, this implies $\int_t^\infty|\dot X_N^s|\ud s\le\varphi(e(X_N^t))$. See \cite[Theorems~1 and~2]{Kurdyka1998Gradients}.}
\end{remark}

With global collision-free particle dynamics and its fixed-$N$ asymptotics established, we next exhibit collision-free saddle equilibria before turning to the mean-field estimate.

\begin{proof}[Proof of \Cref{prop:particle-saddle-configurations}]
Write
\begin{equation}\label{eq:particle-saddle-target-potential}
W(x):=\int_{B_1}|x-y|\ud\mu(y).
\end{equation}
Up to a constant depending only on $\mu$, the particle energy is
\begin{equation}\label{eq:particle-saddle-energy-expansion}
\mathcal E_N(X_N)
=
\frac1N\sum_{i=1}^N W(x_i)
-\frac1{N^2}\sum_{1\le i<j\le N}|x_i-x_j|.
\end{equation}
The function $W$ is radial and belongs to $C^2(\R^d)$ when $d\ge2$.

We first prove part~\textup{(i)}.  Define
\begin{equation}\label{eq:particle-line-force-profile}
\phi(s):=\partial_s W(se_1)
=
\int_{B_1}\frac{s-y_1}{|se_1-y|}\ud\mu(y).
\end{equation}
This function is odd, and differentiation under the integral gives
\begin{equation}\label{eq:particle-line-force-monotonicity}
\phi'(s)
=
\int_{B_1}
\frac{|y-(y\cdot e_1)e_1|^2}{|se_1-y|^3}\ud\mu(y)>0.
\end{equation}
Moreover, $\phi(s)\to\pm1$ as $s\to\pm\infty$.  Thus $\phi:\R\to(-1,1)$ is an increasing homeomorphism.  Set
\begin{equation}\label{eq:particle-line-saddle-points}
\xi_i:=\phi^{-1}\left(\frac{2i-N-1}{N}\right),
\qquad
x_i:=\xi_i e_1,
\qquad i\in[N].
\end{equation}
The $\xi_i$ are strictly increasing, and hence the resulting configuration $X_N^{\mathrm{line}}$ belongs to $\mathcal D_N$.  Since
\begin{equation}
\sum_{j\ne i}\frac{x_i-x_j}{|x_i-x_j|}
=\sum_{j\ne i}\sgn(\xi_i-\xi_j)e_1
=(2i-N-1)e_1
=N\phi(\xi_i)e_1,
\end{equation}
differentiating \eqref{eq:particle-saddle-energy-expansion} shows that
\begin{equation}
\nabla\mathcal E_N(X_N^{\mathrm{line}})=0.
\end{equation}

Let $e_2$ be any unit vector orthogonal to $e_1$, and set
\begin{equation}
h(s):=
\begin{cases}
\phi(s)/s,&s\ne0,\\
\phi'(0),&s=0.
\end{cases}
\end{equation}
For a transverse variation $Z=(z_i e_2)_{i\in[N]}$, one has
\begin{equation}\label{eq:particle-line-transverse-hessian}
D^2\mathcal E_N(X_N^{\mathrm{line}})[Z,Z]
=
\frac1N\sum_{i=1}^N h(\xi_i)z_i^2
-\frac1{N^2}\sum_{1\le i<j\le N}
\frac{(z_i-z_j)^2}{|\xi_i-\xi_j|}.
\end{equation}
Since $\phi$ is odd, so is $\phi^{-1}$.  If $N=2M$, set $a:=\phi^{-1}(1/N)>0$.  Then $\xi_M=-a$, $\xi_{M+1}=a$, and $\phi(a)=1/N$.  Choose
\begin{equation}
z_i
:=
\begin{cases}
1,&i=M,\\
-1,&i=M+1,\\
0,&\text{otherwise}.
\end{cases}
\end{equation}
The target contribution in \eqref{eq:particle-line-transverse-hessian} cancels exactly with the interaction of the central pair.  The interactions of that pair with every remaining particle are strictly negative.  Hence
\begin{equation}
D^2\mathcal E_N(X_N^{\mathrm{line}})[Z,Z]<0.
\end{equation}
If $N=2M+1\ge5$, set $a:=\phi^{-1}(2/N)>0$.  By the oddness of $\phi^{-1}$, the central triple is $(-a,0,a)$, and $\phi(a)=2/N$.  Taking transverse coefficients $(1,0,-1)$ on this triple and zero elsewhere again gives exact cancellation between the target term and the interactions internal to the triple, while its interactions with the remaining particles are strictly negative.  This proves the existence of a negative direction for every $N\ge4$.

On the other hand, for an axial variation $H=(\eta_i e_1)_{i\in[N]}$, the ordering of the particles is unchanged for small variations, so the pair-distance terms are affine.  Therefore
\begin{equation}\label{eq:particle-line-longitudinal-hessian}
D^2\mathcal E_N(X_N^{\mathrm{line}})[H,H]
=
\frac1N\sum_{i=1}^N\phi'(\xi_i)\eta_i^2>0
\end{equation}
whenever $H\ne0$.  Thus the Hessian is indefinite.

The points in \eqref{eq:particle-line-saddle-points} are the images under $\phi^{-1}$ of the midpoint grid on $(-1,1)$.  Consequently,
\begin{equation}\label{eq:particle-line-saddle-limit}
\frac1N\sum_{i=1}^N\delta_{\xi_i e_1}
\rightharpoonup
\left(s\mapsto\phi^{-1}(s)e_1\right)_\#
\left(\frac12\indic_{(-1,1)}(s)\ud s\right).
\end{equation}
The limiting measure is nonatomic and supported on the line $\R e_1$.

We next prove part~\textup{(ii)}.  Let $\Pi:=\operatorname{span}\{e_1,e_2\}$, and let $R_\theta$ denote rotation through angle $\theta$ in $\Pi$, extended by the identity on $\Pi^\perp$.  For $m\ge3$, consider the phase-aligned regular-polygon class
\begin{equation}\label{eq:aligned-polygonal-class}
\mathscr A_m
:=
\left\{
x_{\ell,j}
=
R_{2\pi j/m}(r_\ell e_1+z_\ell):
\ 1\le\ell\le m,\quad
j\in\mathbb Z/m\mathbb Z,\quad
r_\ell\ge0,\quad z_\ell\in\Pi^\perp
\right\}.
\end{equation}
Thus each configuration consists of $m$ regular $m$-gons with a common angular phase.

The restriction of $\mathcal E_{m^2}$ to $\mathscr A_m$ attains its minimum.  Indeed, \Cref{prop:mmdtomoment}, with any $0<r<1/2$, shows that a bounded-energy sequence has uniformly bounded empirical $r$-moment.  Since $m^2$ is fixed and all particles have the same weight, every individual particle, and hence every parameter $(r_\ell,z_\ell)$, is bounded.  Existence then follows from continuity.

Every minimizing configuration in $\mathscr A_m$ is collision-free.  A collision in this class occurs only if either $r_\ell=0$ for some $\ell$, or two seeds $(r_\ell,z_\ell)$ coincide.  Write
\begin{equation}
u_j:=R_{2\pi j/m}e_1,
\qquad j=0,\ldots,m-1,
\qquad
\sum_{j=0}^{m-1}u_j=0.
\end{equation}

Suppose first that $r_\ell=0$.  For $\varepsilon>0$, define $X_\varepsilon$ by
\begin{equation}
x_{\ell',k}^\varepsilon
:=
\begin{cases}
z_\ell+\varepsilon u_k,&\ell'=\ell,\\
x_{\ell',k},&\ell'\ne\ell.
\end{cases}
\end{equation}
The change in the target term is
\begin{equation}
\frac1{m^2}\sum_{j=0}^{m-1}
\bigl(W(z_\ell+\varepsilon u_j)-W(z_\ell)\bigr)
=
\frac{\varepsilon}{m^2}\nabla W(z_\ell)\cdot
\sum_{j=0}^{m-1}u_j
+O(\varepsilon^2)
=
O(\varepsilon^2).
\end{equation}
The interactions among the particles in the perturbed orbit contribute
\begin{equation}
-\frac1{m^4}
\sum_{0\le j<k\le m-1}
\left|x_{\ell,j}^\varepsilon-x_{\ell,k}^\varepsilon\right|
=
-\frac{\varepsilon}{m^4}
\sum_{0\le j<k\le m-1}|u_j-u_k|
=:-c_m\varepsilon,
\end{equation}
where $c_m>0$.  For any unchanged particle $y\ne z_\ell$,
\begin{equation}
-\frac1{m^4}\sum_{j=0}^{m-1}
\bigl(|z_\ell+\varepsilon u_j-y|-|z_\ell-y|\bigr)
=
-\frac{\varepsilon}{m^4}
\frac{z_\ell-y}{|z_\ell-y|}
\cdot\sum_{j=0}^{m-1}u_j
+O(\varepsilon^2)
=
O(\varepsilon^2).
\end{equation}
If instead $y=z_\ell$, the corresponding change is
\begin{equation}
-\frac1{m^4}\sum_{j=0}^{m-1}
|\varepsilon u_j|
=
-\frac{\varepsilon}{m^3}\le0.
\end{equation}
Consequently,
\begin{equation}
\mathcal E_{m^2}(X_\varepsilon)
-\mathcal E_{m^2}(X_0)
\le
-c_m\varepsilon+O(\varepsilon^2)<0
\end{equation}
for all sufficiently small $\varepsilon>0$, contradicting minimality.

Suppose next that two nonzero seeds coincide, say
\begin{equation}
(r_\ell,z_\ell)=(r_{\ell'},z_{\ell'})=(r,z),
\qquad r>0.
\end{equation}
For $0<\varepsilon<r$, define $X_\varepsilon$ by
\begin{equation}
x_{\ell'',j}^\varepsilon
:=
\begin{cases}
z+(r+\varepsilon)u_j,&\ell''=\ell,\\
z+(r-\varepsilon)u_j,&\ell''=\ell',\\
x_{\ell'',j},&\ell''\notin\{\ell,\ell'\}.
\end{cases}
\end{equation}
Writing $X_\varepsilon=(x_a^\varepsilon)_{a\in[m^2]}$, decompose the configuration-dependent part of the energy as
\begin{equation}
\begin{aligned}
\mathcal E_{\mathrm{sm}}(\varepsilon)
&:=
\frac1{m^2}\sum_{a=1}^{m^2}W(x_a^\varepsilon)
-\frac1{m^4}
\sum_{\substack{1\le a<b\le m^2\\x_a^0\ne x_b^0}}
|x_a^\varepsilon-x_b^\varepsilon|,
\\
\mathcal E_{\mathrm{coll}}(\varepsilon)
&:=
-\frac1{m^4}
\sum_{\substack{1\le a<b\le m^2\\x_a^0=x_b^0}}
|x_a^\varepsilon-x_b^\varepsilon|.
\end{aligned}
\end{equation}
The first function is $C^2$ near $\varepsilon=0$.  Since $\varepsilon\mapsto-\varepsilon$ exchanges the $\ell$-th and $\ell'$-th polygonal orbits,
\begin{equation}
\mathcal E_{\mathrm{sm}}(\varepsilon)
=
\mathcal E_{\mathrm{sm}}(-\varepsilon),
\qquad
\mathcal E_{\mathrm{sm}}(\varepsilon)
-\mathcal E_{\mathrm{sm}}(0)
=
O(\varepsilon^2).
\end{equation}
Moreover, $\mathcal E_{\mathrm{coll}}(0)=0$, and the interactions between corresponding vertices give
\begin{equation}
\mathcal E_{\mathrm{coll}}(\varepsilon)
\le
-\frac1{m^4}\sum_{j=0}^{m-1}
\left|x_{\ell,j}^\varepsilon
      -x_{\ell',j}^\varepsilon\right|
=
-\frac1{m^4}\sum_{j=0}^{m-1}2\varepsilon
=
-\frac{2\varepsilon}{m^3}.
\end{equation}
Consequently,
\begin{equation}
\mathcal E_{m^2}(X_\varepsilon)
-\mathcal E_{m^2}(X_0)
=
\mathcal E_{\mathrm{sm}}(\varepsilon)
-\mathcal E_{\mathrm{sm}}(0)
+\mathcal E_{\mathrm{coll}}(\varepsilon)
\le
-\frac{2\varepsilon}{m^3}
+O(\varepsilon^2)
<0
\end{equation}
for all sufficiently small $\varepsilon>0$, again contradicting minimality.

Fix a collision-free minimizer $X\in\mathscr A_m$ and write
\begin{equation}
G_{\ell,j}:=\nabla_{x_{\ell,j}}\mathcal E_{m^2}(X).
\end{equation}
Cyclic symmetry gives
\begin{equation}
G_{\ell,j}=R_{2\pi j/m}G_{\ell,0}.
\end{equation}
Let $S$ denote reflection across
$\operatorname{span}\{e_1\}\oplus\Pi^\perp$.  Since
$Su_j=u_{-j}$ and $Sz_\ell=z_\ell$, one has
\begin{equation}
Sx_{\ell,j}
=
S(r_\ell u_j+z_\ell)
=
r_\ell u_{-j}+z_\ell
=
x_{\ell,-j}.
\end{equation}
The target measure $\mu$ is invariant under $S$, and the particle
energy is invariant under particle relabeling.  Since $X$ is
collision-free, equivariance of the gradient therefore gives
$SG_{\ell,j}=G_{\ell,-j}$.
Taking $j=0$ yields
\begin{equation}
SG_{\ell,0}=G_{\ell,0} \Longrightarrow
G_{\ell,0}
\in\operatorname{span}\{e_1\}\oplus\Pi^\perp.
\end{equation}

Stationarity with respect to $r_\ell>0$, together with cyclic symmetry,
gives
\begin{equation}
0
=
\partial_{r_\ell}\mathcal E_{m^2}(X)
=
\sum_{j=0}^{m-1}G_{\ell,j}\cdot u_j
=
mG_{\ell,0}\cdot e_1.
\end{equation}
Likewise, for every $z\in\Pi^\perp$, stationarity with respect to
$z_\ell$ gives
\begin{equation}
0
=
D_{z_\ell}\mathcal E_{m^2}(X)[z]
=
\sum_{j=0}^{m-1}G_{\ell,j}\cdot z
=
mG_{\ell,0}\cdot z.
\end{equation}
Thus
$G_{\ell,0}\in
(\operatorname{span}\{e_1\}\oplus\Pi^\perp)
\cap
(\operatorname{span}\{e_1\}\oplus\Pi^\perp)^\perp
=\{0\}$,
and cyclic symmetry yields
$G_{\ell,j}=0$ for every $\ell,j$.  The constrained minimizer is
therefore a critical point of the unrestricted particle energy.

It remains to prove that its Hessian is indefinite.  Rotate one polygon, say the $\ell$-th, through an angle $\theta$ while keeping all other polygons fixed.  The target term and the interactions within that polygon are unchanged.  Writing
$x_{\ell,j}(\theta):=z_\ell+r_\ell R_{\theta+2\pi j/m}e_1$
for the rotated vertices, the sum of their cross-distances from the
$k$-th polygon, for $k\ne\ell$, is
\begin{equation}
D_{\ell k}(\theta)
:=
\sum_{j,j'=0}^{m-1}
\left|x_{\ell,j}(\theta)-x_{k,j'}\right|
=
\sum_{j,j'=0}^{m-1}
f_{\ell k}\left(
\theta+\frac{2\pi(j-j')}{m}
\right)
=
m\sum_{a=0}^{m-1}
f_{\ell k}\left(\theta+\frac{2\pi a}{m}\right).
\end{equation}
Here
$f_{\ell k}(\alpha):=
\sqrt{A_{\ell k}-2c_{\ell k}\cos\alpha}$,
where
\begin{equation}
A_{\ell k}
=
r_\ell^2+r_k^2+|z_\ell-z_k|^2,
\qquad
c_{\ell k}=r_\ell r_k.
\end{equation}
The last equality in the preceding display follows because, for every
$a\in\mathbb Z/m\mathbb Z$, there are exactly $m$ pairs $(j,j')$ such
that $j-j'=a$.
Collision-freeness implies $A_{\ell k}>2c_{\ell k}>0$.  In the cosine expansion
\begin{equation}
f_{\ell k}(\theta)
=
\widehat f_0
+2\sum_{n\ge1}\widehat f_n\cos(n\theta),
\end{equation}
we claim that every coefficient $\widehat f_n$ with $n\ge1$ is strictly negative.  Indeed, for brevity, set $A:=A_{\ell k}$ and $c:=c_{\ell k}$.  The convergent binomial expansion
\begin{equation}
\sqrt{A-2c\cos\theta}
=
A^{1/2}\sum_{s=0}^\infty
\binom{1/2}{s}\left(-\frac{2c}{A}\right)^s\cos^s\theta
\end{equation}
has a strictly negative coefficient in front of $\cos^s\theta$ for every $s\ge1$.  Set
\begin{equation}
\beta_s
:=
\binom{1/2}{s}\left(-\frac{2c}{A}\right)^s.
\end{equation}
Thus $\beta_s<0$ for every $s\ge1$.  For fixed $n\ge1$, the coefficient of $\cos(n\theta)$ in $\cos^s\theta$ vanishes unless $s=n+2r$ for some $r\ge0$; when $s=n+2r$, that coefficient is $2^{1-n-2r}\binom{n+2r}{r}>0$.  Comparing the coefficients of $\cos(n\theta)$ in the two expansions therefore gives
\begin{equation}
2\widehat f_n
=
A^{1/2}
\sum_{r=0}^\infty
\beta_{n+2r}
2^{1-n-2r}
\binom{n+2r}{r}
<0.
\end{equation}

Next, substituting the Fourier expansion of $f_{\ell k}$ into the definition of $D_{\ell k}$ and using the root-of-unity sum identity $\sum_{a=0}^{m-1}e^{2\pi i n a/m} = m \indic_{m|n}$  gives
\begin{equation}
D_{\ell k}(\theta)
=
m^2\widehat f_0
+
2m^2\sum_{p\ge1}
\widehat f_{pm}\cos(pm\theta).
\end{equation}
Since $A>2c$, the function $f_{\ell k}$ is real analytic, so the Fourier series may be differentiated termwise to give
\begin{equation}
D_{\ell k}''(0)
=
-2m^2\sum_{p\ge1}(pm)^2\widehat f_{pm}
>0,
\end{equation}
where the positivity follows from the fact that $\widehat f_{pm}<0$ for every $p\ge1$.
Because the cross-distance sum occurs with a negative sign in \eqref{eq:particle-saddle-energy-expansion}, rotating one polygon gives
\begin{equation}
\left.\frac{\ud^2}{\ud\theta^2}
\mathcal E_{m^2}(X(\theta))\right|_{\theta=0}
=
-\frac1{m^4}\sum_{k\ne\ell}D_{\ell k}''(0)<0.
\end{equation}

For a positive direction, dilate the whole configuration about the origin.  The pair-distance part of the energy is affine under dilation, whereas
\begin{equation}\label{eq:uniform-ball-dilation-hessian}
D^2W(x)[x,x]
=
\int_{B_1}
\frac{|x|^2|y|^2-(x\cdot y)^2}{|x-y|^3}\ud\mu(y)>0
\qquad (x\ne0).
\end{equation}
Every particle has nonzero planar radius, and hence
\begin{equation}
D^2\mathcal E_{m^2}(X)[X,X]
=
\frac1{m^2}\sum_{\ell,j}
D^2W(x_{\ell,j})[x_{\ell,j},x_{\ell,j}]>0.
\end{equation}
Thus the critical configuration is a Hessian-indefinite saddle.

Finally, let $\lambda$ be the pushforward of $\mu$ under
\begin{equation}
x\longmapsto\left(|P_\Pi x|,P_{\Pi^\perp}x\right).
\end{equation}
Choose $m$ seeds in $\supp\lambda$ whose empirical measures converge narrowly to $\lambda$, and place an aligned regular $m$-gon above each seed.  The resulting competitor measures $\widetilde\rho_m$, all supported in $\overline B_1$, converge narrowly to $\mu$: the angular sums are Riemann sums on $\mathbb S^1$, while the seed measures converge to $\lambda$.  Since $|x-y|$ is bounded and continuous on $\overline B_1\times\overline B_1$,
\begin{equation}
\mmd_1^2(\widetilde\rho_m,\mu)\longrightarrow0.
\end{equation}
If $\rho_m^{\mathrm{bulk}}$ is the empirical measure of the minimizing configuration constructed above, its restricted minimality gives
\begin{equation}
0\le
\mmd_1^2(\rho_m^{\mathrm{bulk}},\mu)
\le
\mmd_1^2(\widetilde\rho_m,\mu)
\longrightarrow0.
\end{equation}
The topological implication in \Cref{prop:mmd-topology} now yields $\rho_m^{\mathrm{bulk}}\rightharpoonup\mu$, proving \eqref{eq:bulk-saddle-empirical-convergence}.

Finally, the same fixed-$N$ moment coercivity used above shows that the set of global minimizers of $\mathcal E_N$ is nonempty and compact.  Since each saddle constructed above has a negative Hessian direction, it is not a global minimizer and therefore has positive distance from this compact set.  Its constant particle trajectory consequently does not approach the set of global minimizers.
\end{proof}

\begin{remark}[The one-dimensional energy-distance case]
When $d=q=1$, the target term in \eqref{eq:particle-saddle-energy-expansion} is convex and the pair-interaction term is affine on every ordered chamber.  Hence $\mathcal E_N$ is convex on each ordered chamber, so every collision-free critical configuration minimizes on its chamber.  By permutation symmetry and continuity on the chamber closures, it is a global minimizer.  Thus the restriction $d\ge2$ in \Cref{prop:particle-saddle-configurations} is essential.
\end{remark}

We now turn to the mean-field estimate.  The first step is an exact evolution identity for the modulated MMD energy.

\begin{lemma}[Modulated-energy identity]\label{lem:mf-modulated-energy-identity}
Let $0<q<2$.  Assume either that $d+q-2>0$ and $1<p<\infty$ satisfies $p>p_c$, or that $d=q=1$ and $p=p_c=\infty$.  Let $\rho_t$ be the corresponding solution of \eqref{eq:wgfmmd} on $[0,T]$, and assume that $\sup_{0\le t\le T}\mathcal M_q(\rho_t)<\infty$ and $v\in L^\infty([0,T];W^{1,\infty})$.  Let $\XN^t$ be a collision-free particle solution on $[0,T]$ in the sense of \Cref{def:collision-free-particles}.  Define
\begin{equation}
\rho_t^N:=\frac1N\sum_{i=1}^N\delta_{x_i^t},
\qquad
\nu_t:=\rho_t^N-\rho_t.
\end{equation}
Set
\begin{equation}\label{eq:Et-particle-value}
E_t:=\nabla K*(\rho_t^N-\rho_t),
\end{equation}
where the convolution is understood in the principal-value sense.
Then $t\mapsto\mmd_q^2(\rho_t^N,\rho_t)$ is absolutely continuous on $[0,T]$, and for $\mathcal L^1$-a.e. $t\in[0,T]$,
\begin{equation}\label{eq:mf-modulated-energy-identity}
\frac{\ud}{\ud t}\mmd_q^2(\rho_t^N,\rho_t)
=-\int_{\R^d}|E_t|^2\ud\rho_t^N
+\mathcal C_t^N,
\end{equation}
where
\begin{equation}\label{eq:mf-commutator-CN}
\mathcal C_t^N
:=\frac12\iint_{(\R^d)^2\setminus\Delta}
\bigl(v_t(x)-v_t(y)\bigr)\cdot\nabla K(x-y)\ud\nu_t(x)\ud\nu_t(y),
\qquad
\Delta:=\{(x,x):x\in\R^d\}.
\end{equation}
Since $\rho_t$ is absolutely continuous, the removal of $\Delta$ only removes the empirical self-interaction terms.
\end{lemma}

\begin{proof}
This is the standard modulated-energy computation; compare \cite[\S 6.2.3]{SerfatyLN}.  We record the details needed to justify it under the present hypotheses.  Write $F_N(t):=\mmd_q^2(\rho_t^N,\rho_t)$.  By the definitions of the particle and continuum velocities,
\begin{equation}\label{eq:particle-velocity-v-minus-E}
\dot x_i^t=v_t(x_i^t)-E_t(x_i^t).
\end{equation}
Collision-freeness and compactness of $[0,T]$ give a positive minimum interparticle distance, so the off-diagonal particle terms are absolutely continuous.  Moreover, $|\nabla^{\otimes 2}K(z)|\lesssim|z|^{q-2}$ and $(2-q)p_*<d$ give the required local integrability against $\rho_t$, while the finite $q$-moment controls the far field.  Testing the weak equation with smooth truncations and mollifications of $K(x_i^t-\cdot)$ and passing to the limits therefore justifies the following differentiation for $\mathcal L^1$-a.e. $t$.  Expanding the three terms of \eqref{eq:mf-modulated-energy-static}, using the oddness of $\nabla K$, and then substituting \eqref{eq:particle-velocity-v-minus-E} give
\begin{align}
F_N'(t)
&=\frac1N\sum_{i\in[N]}\dot x_i^t\cdot E_t(x_i^t)
   -\int_{\R^d} E_t(y)\cdot v_t(y)\rho_t(y)\ud y\notag\\
&=-\int_{\R^d}|E_t|^2\ud\rho_t^N+\int_{\R^d} E_t\cdot v_t\ud\nu_t.\label{eq:identity-before-symmetrization}
\end{align}
By the diagonal-free definition at the atoms, the last term is the off-diagonal double integral of $v_t(x)\cdot\nabla K(x-y)$ against $\nu_t\otimes\nu_t$.  Exchanging $x$ and $y$ and using the oddness of $\nabla K$ identifies it with \eqref{eq:mf-commutator-CN}, proving \eqref{eq:mf-modulated-energy-identity}.
\end{proof}

We next prove the commutator estimate which controls the last term in \eqref{eq:mf-modulated-energy-identity}.

At a schematic level, the Coulomb/super-Coulomb regime $d+q\le2$ is handled directly by the unlocalized commutator estimate of the first author and Serfaty.  In the remaining regime, repeated integration by parts transfers derivatives from the zero-mean inputs to the kernel until the terminal differentiated kernel has \mbox{Riesz exponent $s_m\in[d-2,d)$}. The estimate of Nguyen, the first author, and Serfaty controls this terminal term, while the lower-order terms are handled by the homogeneous Kato--Ponce and Hardy--Littlewood--Sobolev inequalities.

The only endpoint obstruction in the iterative physical-space integration-by-parts proof occurs when $q=1$ and $d$ is odd.  We avoid the gap by choosing the number of integrations by parts so that the terminal kernel is in the Coulomb/super-Coulomb range and then invoking the Calder\'on commutator argument of Nguyen, the first author, and Serfaty \cite[Proposition~3.1]{NRS2021} for the terminal term.

\begin{proposition}[Transport commutator and empirical extension]\label{prop:comm}
Let $d\ge1$ and $0<q<2$.  Set
\begin{equation}
\alpha:=\frac{d+q}{2}.
\end{equation}
There exists $C=C(d,q)>0$ such that the following holds.  If $v:\R^d\to\R^d$ satisfies
\begin{equation}
\|\nabla v\|_{L^\infty}
+\|\dm^\alpha v\|_{L^{\frac{2d}{d+q-2}}}\indic_{d+q>2}<\infty,
\end{equation}
then for all real-valued $f,g\in\Sc(\R^d)$ with zero mean,
\begin{multline}\label{eq:comm-smooth}
\left|\iint_{(\R^d)^2}(v(x)-v(y))\cdot\nabla K(x-y)f(x)g(y)\ud x\ud y\right| \\
\le C\left(\|\nabla v\|_{L^\infty}+\|\dm^\alpha v\|_{L^{\frac{2d}{d+q-2}}}\indic_{d+q>2}\right)
\|f\|_{\dot H^{-\alpha}}\|g\|_{\dot H^{-\alpha}}.
\end{multline}
Moreover, the left-hand side extends uniquely to a bounded bilinear form $B_v$ on the zero-mean homogeneous space $\dot H^{-\alpha}(\R^d)$, with the same bound.

Let $\rho\in\mathcal P_q(\R^d)$ and let $X_N=(x_i)_{i\in[N]}$ be a pairwise distinct configuration.  With $\rho^N$ denoting the associated empirical measure, set $\nu:=\rho^N-\rho$.  Then $\nu\in\dot H^{-\alpha}(\R^d)$ and
\begin{equation}\label{eq:comm-measure-extension}
B_v(\nu,\nu)
=\iint_{(\R^d)^2\setminus\Delta}(v(x)-v(y))\cdot\nabla K(x-y)\ud\nu(x)\ud\nu(y),
\end{equation}
where the diagonal value is understood to be zero.\footnote{In the intended application below, $\rho$ is absolutely continuous, so the only removed diagonal mass is the empirical self-interaction.}  Consequently,
\begin{equation}\label{eq:comm-measure-bound}
\left|\iint_{(\R^d)^2\setminus\Delta}(v(x)-v(y))\cdot\nabla K(x-y)\ud\nu(x)\ud\nu(y)\right|
\le C\left(\|\nabla v\|_{L^\infty}+\|\dm^\alpha v\|_{L^{\frac{2d}{d+q-2}}}\indic_{d+q>2}\right)
\|\nu\|_{\dot H^{-\alpha}}^2.
\end{equation}
\end{proposition}

\begin{proof}
We first prove \eqref{eq:comm-smooth} for smooth $v$; the extension to general $v$ satisfying the displayed norm assumptions follows at the end by mollification.  By density in $\dot H^{-\alpha}$, we may assume that $\widehat f$ and $\widehat g$ vanish in a neighborhood of the origin.  If $d+q\le2$, then necessarily $d=1$ and $0<q\le1$.  Setting $s=-q\in[d-2,d)$ and accounting for the normalization $K(x)=q s^{-1}|x|^{-s}$, the estimate and bounded bilinear extension follow from the unlocalized first-order commutator estimate \cite[Theorem~4.1]{RosenzweigSerfaty2026SharpCommutators}.  We may therefore assume below that $d+q>2$.

For $j\ge1$, write
\begin{equation}
f=\partial_{i_1}\cdots\partial_{i_j} f_j^{i_1\cdots i_j},\qquad
g=\partial_{\ell_1}\cdots\partial_{\ell_j} g_j^{\ell_1\cdots \ell_j},
\end{equation}
where
\begin{equation}
f_j^{i_1\cdots i_j}:=(-1)^j\partial_{i_1}\cdots\partial_{i_j}(-\Delta)^{-j}f,\qquad
g_j^{\ell_1\cdots\ell_j}:=(-1)^j\partial_{\ell_1}\cdots\partial_{\ell_j}(-\Delta)^{-j}g.
\end{equation}
The Fourier-support assumption ensures that these tensor fields are Schwartz, and
\begin{equation}\label{eq:tensor-sobolev-shift}
\|f_j\|_{\dot H^{j-\alpha}}\le C_{d,j}\|f\|_{\dot H^{-\alpha}},\qquad
\|g_j\|_{\dot H^{j-\alpha}}\le C_{d,j}\|g\|_{\dot H^{-\alpha}}.
\end{equation}
Choose the integer $m=m(d,q)$ by
\begin{equation}
 m=\begin{cases}
 d/2, & d\ \text{even},\\[2mm]
 (d-1)/2, & d\ \text{odd and }0<q\le1,\\[2mm]
 (d+1)/2, & d\ \text{odd and }1<q<2.
 \end{cases}
\end{equation}
Then
\begin{equation}\label{eq:sm-choice}
s_m:=2m-q\in[d-2,d),\qquad m<\alpha.
\end{equation}
\begin{samepage}
For $1\le j\le m$, define, with summation over repeated indices,
\begin{align}
T_j^g
&:=\int_{\R^d}\partial_{i_j}v\cdot\nabla\bigl(\partial_{i_1}\cdots\partial_{i_{j-1}}K*g\bigr)
\, f_j^{i_1\cdots i_j}\ud x,\notag\\
T_j^f
&:=\int_{\R^d}\partial_{\ell_j}v\cdot\nabla\bigl(\partial_{\ell_1}\cdots\partial_{\ell_{j-1}}K*f\bigr)
\, g_j^{\ell_1\cdots \ell_j}\ud x.
\label{eq:comm-lower-order-terms}
\end{align}
\end{samepage}
For $1\le j\le m$, let $I_j=(i_1,\ldots,i_j)$ and $L_j=(\ell_1,\ldots,\ell_j)$, set
$K_{I_j,L_j}:=\partial_{i_1}\cdots\partial_{i_j}\partial_{\ell_1}\cdots\partial_{\ell_j}K$, and define
\begin{equation}\label{eq:comm-ibp-remainder-definition}
R_j:=\iint_{(\R^d)^2}(v(x)-v(y))\cdot\nabla K_{I_j,L_j}(x-y)
 f_j^{I_j}(x)g_j^{L_j}(y)\ud x\ud y.
\end{equation}
We adopt the conventions $I_0=L_0=\emptyset$, $f_0:=f$, $g_0:=g$, and $R_0:=B_v(f,g)$.  A paired integration by parts in $x$ and $y$ gives, for every $1\le j\le m$,
\begin{equation}\label{eq:comm-ibp-recurrence}
R_{j-1}=-\bigl(T_j^f+T_j^g\bigr)-R_j.
\end{equation}
Indeed, the $x$ integration by parts, using $f_{j-1}^{I_{j-1}}=\partial_{i_j}f_j^{I_j}$, produces $-T_j^g$; the subsequent $y$ integration by parts, using $g_{j-1}^{L_{j-1}}=\partial_{\ell_j}g_j^{L_j}$, produces $-T_j^f-R_j$.  Iterating \eqref{eq:comm-ibp-recurrence} yields
\begin{equation}\label{eq:comm-ibp-decomposition}
B_v(f,g)=\sum_{j=1}^m(-1)^j\bigl(T_j^f+T_j^g\bigr)+(-1)^mR_m.
\end{equation}
For $j=m$, abbreviate $I:=I_m$, $L:=L_m$, and $K_{I,L}:=K_{I_m,L_m}$.
Rigorously speaking, identity \eqref{eq:comm-ibp-decomposition} is first obtained with a truncation $|x-y|>\varepsilon$; the boundary terms vanish as $\varepsilon\downarrow0$ because the differentiated kernels have the homogeneities recorded below and the tensor fields are Schwartz.  This is the same integration-by-parts step as in the proof of \cite[Proposition~3.1]{NRS2021}.

We estimate the remainder $R_m$ first.  The kernel $K_{I,L}$ is even and homogeneous of degree $q-2m=-s_m$, and its derivatives satisfy
\begin{equation}
|\nabla^{\otimes k}K_{I,L}(x)|\le C_{d,q,k}|x|^{-s_m-k},\qquad x\neq0.
\end{equation}
Its Fourier transform is a smooth angular multiplier times $|\xi|^{s_m-d}$.  Hence $K_{I,L}$ satisfies the hypotheses of \cite[Proposition~3.1]{NRS2021} with Riesz exponent $s_m$.  If $s_m=d-2$, which occurs only when $d$ is odd and $q=1$, the additional Calder\'on--Zygmund hypothesis in that proposition is satisfied by the standard Calder\'on kernel associated with constant-coefficient derivatives of the Coulomb kernel.  Applying \cite[Proposition~3.1]{NRS2021} and using $s_m\ge d-2$, so that no sub-Coulomb auxiliary Sobolev norm of $v$ appears, gives
\begin{align}
|R_m|
&\le C_{d,q}\|\nabla v\|_{L^\infty}
\|f_m\|_{\dot H^{\frac{s_m-d}{2}}}
\|g_m\|_{\dot H^{\frac{s_m-d}{2}}} \notag\\
&\le C_{d,q}\|\nabla v\|_{L^\infty}
\|f\|_{\dot H^{-\alpha}}
\|g\|_{\dot H^{-\alpha}},\label{eq:terminal-comm-bound}
\end{align}
where the last line follows from $(s_m-d)/2=m-\alpha$ and \eqref{eq:tensor-sobolev-shift}.

It remains to estimate $T_j^f,T_j^g$.  We treat $T_j^g$; the other term is identical.  Since $j\le m<\alpha$, duality and Cauchy--Schwarz give
\begin{align}
|T_j^g|
&\le
\|\dm^{j-\alpha}(\partial_{i_j}v\, f_j^{i_1\cdots i_j})\|_{L^2}
\|\dm^{\alpha-j}\nabla(\partial_{i_1}\cdots\partial_{i_{j-1}}K*g)\|_{L^2} \notag\\
&\le C_{d,q}\|\dm^{j-\alpha}(\partial_{i_j}v\, f_j^{i_1\cdots i_j})\|_{L^2}
\|g\|_{\dot H^{-\alpha}}.\label{eq:Tjg-first-bound}
\end{align}
For the first factor, by duality,
\begin{align}
\|\dm^{j-\alpha}(\partial_{i_j}v\, f_j)\|_{L^2}
\le \sup_{\|h\|_{L^2}\le1}
\left|\left\langle
\dm^{\alpha-j}(\partial_{i_j}v\,\dm^{j-\alpha}h),
\dm^{j-\alpha}f_j
\right\rangle\right|.
\end{align}
The homogeneous Kato--Ponce inequality \cite[Corollary~5.2]{li_kato-ponce_2019} with
\begin{equation}
(p_2,q_2)=(\infty,2),\qquad
(p_1,q_1)=\left(\frac{2d}{d+q-2j},\frac{2d}{2j-q}\right)
\end{equation}
combined with the Hardy--Littlewood--Sobolev inequality \cite[Theorem~4.3]{LiebLoss} and boundedness of Riesz transforms \cite[Chapter~III, \S1]{Stein1970singular} yields
\begin{equation}\label{eq:product-factor-bound}
\|\dm^{j-\alpha}(\partial_{i_j}v\, f_j)\|_{L^2}
\le C_{d,q}\left(\|\nabla v\|_{L^\infty}+\|\dm^\alpha v\|_{L^{\frac{2d}{d+q-2}}}\right)
\|f\|_{\dot H^{-\alpha}}.
\end{equation}
Here is the exponent bookkeeping.  The choice \eqref{eq:sm-choice} gives
\begin{equation}
0<\alpha-j<\frac d2,
\qquad
\frac1{q_1}=\frac12-\frac{\alpha-j}{d}=\frac{2j-q}{2d},
\end{equation}
so the Hardy--Littlewood--Sobolev inequality gives
\begin{equation}
\|\dm^{j-\alpha}h\|_{L^{q_1}}\le C\|h\|_{L^2}.
\end{equation}
Moreover,
\begin{equation}
\dm^{\alpha-j}(\partial_{i_j}v)=\mathcal R_{i_j}\dm^{\alpha-j+1}v,
\qquad
\frac1{p_1}=\frac{d+q-2}{2d}-\frac{j-1}{d},
\end{equation}
where $\mathcal R_{i_j}$ denotes the $i_j$-th Riesz transform.  Its $L^{p_1}$-boundedness and Sobolev embedding therefore give
\begin{equation}
\|\dm^{\alpha-j}(\partial_{i_j}v)\|_{L^{p_1}}
\le C\|\dm^\alpha v\|_{L^{\frac{2d}{d+q-2}}}.
\end{equation}
The second Kato--Ponce product term is bounded by
\begin{equation}
\|\partial_{i_j}v\|_{L^\infty}
\|\dm^{\alpha-j}\dm^{j-\alpha}h\|_{L^2}
\le \|\nabla v\|_{L^\infty}\|h\|_{L^2}.
\end{equation}
Finally, \eqref{eq:tensor-sobolev-shift} controls $\|\dm^{j-\alpha}f_j\|_{L^2}$ by $\|f\|_{\dot H^{-\alpha}}$.  This proves \eqref{eq:product-factor-bound}.  Combining \eqref{eq:Tjg-first-bound} and \eqref{eq:product-factor-bound}, and summing over the finitely many tensor indices and over $j$, gives
\begin{equation}
\sum_{j=1}^m(|T_j^f|+|T_j^g|)
\le C_{d,q}\left(\|\nabla v\|_{L^\infty}+\|\dm^\alpha v\|_{L^{\frac{2d}{d+q-2}}}\right)
\|f\|_{\dot H^{-\alpha}}\|g\|_{\dot H^{-\alpha}}.
\end{equation}
Together with \eqref{eq:terminal-comm-bound}, this proves \eqref{eq:comm-smooth} for smooth $v$.

For a general vector field satisfying the displayed assumptions, let $v_\varepsilon=\eta_\varepsilon*v$.  Then $v_\varepsilon\to v$ locally uniformly,
\begin{equation}
\|\nabla v_\varepsilon\|_{L^\infty}\le\|\nabla v\|_{L^\infty},\qquad
\|\dm^\alpha v_\varepsilon\|_{L^{\frac{2d}{d+q-2}}}\indic_{d+q>2}
\le\|\dm^\alpha v\|_{L^{\frac{2d}{d+q-2}}}\indic_{d+q>2}.
\end{equation}
Passing to the limit in the smooth estimate gives \eqref{eq:comm-smooth} for $v$ and the bounded extension $B_v$.

We finally identify the extension on empirical signed measures.  Since $\nu(\R^d)=0$ and $\nu$ has finite $q$-moment, the bounds
\begin{equation}\label{eq:empirical-Hminus-membership}
|\widehat\nu(\xi)|\le C_\beta|\xi|^\beta\int(1+|x|^\beta)\ud|\nu|(x)
\quad\left(\frac q2<\beta\le\min\{1,q\}\right),
\qquad
|\widehat\nu(\xi)|\le|\nu|(\R^d),
\end{equation}
near zero and at infinity show that $\nu\in\dot H^{-\alpha}$.  Let $\theta_R\in C_c^\infty(\R^d)$ be equal to one on $B_R$, let $\zeta$ be a fixed compactly supported probability density, and set
\begin{equation}\label{eq:empirical-approximation}
\nu_R:=\theta_R\nu-\nu(\theta_R)\zeta\ud x,
\qquad
\nu_{R,\varepsilon}:=\eta_\varepsilon*\nu_R.
\end{equation}
Then $\nu_R$ is a compactly supported zero-mass signed measure, while $\nu_{R,\varepsilon}\in C_c^\infty(\R^d)$ has zero mass.  The feature-map representation from \Cref{lem:genSob} gives
\begin{align}
\|\nu-\nu_R\|_{\dot H^{-\alpha}}
&\le C_{d,q}\int_{|x|>R}|x|^{q/2}\ud|\nu|(x)
+|\nu(\theta_R)|\left\|\int_{\R^d}\Phi_x\zeta(x)\ud x\right\|_H,
\label{eq:empirical-cutoff-tail-estimate}
\end{align}
where $|\nu(\theta_R)|=|\nu(1-\theta_R)|\to0$ because $\nu(\R^d)=0$.  The right-hand side tends to zero by the finite $q$-moment.  Standard mollification gives the first convergence below, and the same moment tail gives the stated weighted weak convergence:
\begin{equation}\label{eq:empirical-approximation-convergence}
\nu_{R,\varepsilon}\to\nu_R\quad(\varepsilon\downarrow0),
\qquad
\nu_R\to\nu\quad(R\uparrow\infty)
\end{equation}
in $\dot H^{-\alpha}$ and against continuous functions of growth at most $1+|x|^q$.

The kernel $b_v(x,y):=(v(x)-v(y))\cdot\nabla K(x-y)$, with $b_v(x,x):=0$, is continuous and satisfies
\begin{equation}\label{eq:commutator-kernel-growth}
|b_v(x,y)|\le C\|\nabla v\|_{L^\infty}|x-y|^q
\le C\|\nabla v\|_{L^\infty}(1+|x|^q+|y|^q).
\end{equation}
Thus the boundedness of $B_v$ and the two convergences in \eqref{eq:empirical-approximation-convergence} identify $B_v(\nu,\nu)$ with the weighted weak limit of the kernel integrals.  The empirical self-interactions vanish in the mollification limit because $|b_v(x,y)|\lesssim|x-y|^q$ near the diagonal.  This proves \eqref{eq:comm-measure-extension}; \eqref{eq:comm-measure-bound} follows from the norm bound for $B_v$.
\end{proof}

\begin{proof}[Proof of \Cref{thm:wgfmmdpoc}]
In the case $d+q>2$, if $p=\infty$, choose a finite $p_0>p_c$.  The $p=\infty$ continuum solution coincides, by uniqueness, with the finite-$p_0$ solution, and the target satisfies the corresponding finite-$p_0$ hypotheses.  We may therefore argue below with $p=p_0$.  By \Cref{thm:particle-global-noncollision}, the particle trajectory is collision-free on $[0,T]$ in the sense of \Cref{def:collision-free-particles}.  Let $F_N(t):=\mmd_q^2(\rho_t^N,\rho_t)$.  The signed measure $\nu_t$ has zero total mass and finite $q$-moment, because $\rho_t$ has finite $q$-moment and $\rho_t^N$ is a finite empirical measure.  For $\mathcal L^1$-a.e. $t\in[0,T]$, \Cref{prop:comm} bounds the commutator in \eqref{eq:mf-commutator-CN} by
\begin{equation}
C\Big(\|\nabla v_t\|_{L^\infty}+\|\dm^{\frac{d+q}{2}}v_t\|_{L^{\frac{2d}{d+q-2}}}\indic_{d+q>2}\Big)
\|\nu_t\|_{\dot H^{-\frac{d+q}{2}}}^2.
\end{equation}
\Cref{lem:genSob} identifies this homogeneous Sobolev norm with $F_N(t)$ up to a normalization constant.  Thus, after absorbing that constant and the factor $1/2$ in \eqref{eq:mf-commutator-CN} into $C$, \Cref{lem:mf-modulated-energy-identity} gives
\begin{equation}
F_N'(t)
=-\int_{\R^d}\left|\nabla K*(\rho_t^N-\rho_t)\right|^2\ud\rho_t^N+\mathcal C_t^N
\le C\Big(\|\nabla v_t\|_{L^\infty}+\|\dm^{\frac{d+q}{2}}v_t\|_{L^{\frac{2d}{d+q-2}}}\indic_{d+q>2}\Big)F_N(t).
\end{equation}
The coefficient on the right-hand side belongs to $L^1(0,T)$ by \Cref{lem:mf-velocity-factor}.  Gr\"onwall's inequality therefore yields \eqref{eq:main}.
\end{proof}

The preceding Gr\"onwall argument used one remaining analytic input: the continuum velocity must lie in the commutator class of \Cref{prop:comm}.  We now verify this from the solution bounds by interpolation and fractional integration.

\begin{lemma}[Velocity bound for the mean-field estimate]\label{lem:mf-velocity-factor}
Assume the hypotheses of \Cref{thm:wgfmmdwp} with moment exponent $r\ge q$.  Let $v_t=-\nabla K*(\rho_t-\mu)$ be the corresponding velocity.  Set
\begin{equation}
\sigma:=\frac{d+q}{2}.
\end{equation}
Then, for every finite $T>0$,
\begin{equation}
\|\nabla v\|_{L^\infty(0,T;L^\infty)}
+
\|\dm^\sigma v\|_{L^\infty(0,T;L^{\frac{2d}{d+q-2}})}\indic_{d+q>2}<\infty.
\end{equation}
Here and above, the term multiplied by $\indic_{d+q>2}$ is omitted when $d+q\le2$.
Consequently, the exponential factor in \eqref{eq:main} is finite on every interval $[0,T]$.
\end{lemma}

\begin{proof}
If $d=q=1$, then $\partial_x v_t=2(\rho_t-\mu)$ distributionally, so the asserted bound follows from the $L^\infty$ bounds on the source and target; the second term is omitted.  We may therefore assume $d+q>2$.  The bound for $\nabla v$ is part of the solution class in \Cref{thm:wgfmmdwp}.  Let $\alpha:=\sigma-1=(d+q-2)/2$.  Since the subcritical hypothesis gives $p>p_c$, interpolation between $L^1$ and $L^p$ yields
\begin{equation}
\sup_{0\le t\le T}\|\rho_t-\mu\|_{L^{p_c}}<\infty.
\end{equation}
On zero-mass functions, the Fourier multiplier identity for $K(x)=-|x|^q$ reads $\dm^\sigma v_t=\mathcal R\dm^{-\alpha}(\rho_t-\mu)$, up to a dimensional constant, where $\mathcal R$ is a vector-valued Riesz transform.  Since
\begin{equation}
\left(\frac{2d}{d+q-2}\right)^{-1}=p_c^{-1}-\frac{\alpha}{d},
\end{equation}
the same Hardy--Littlewood--Sobolev inequality and the $L^{\frac{2d}{d+q-2}}$ boundedness of Riesz transforms give
\begin{equation}
\|\dm^\sigma v_t\|_{L^{\frac{2d}{d+q-2}}}
\lesssim_{d,q}\|\dm^{-\alpha}(\rho_t-\mu)\|_{L^{\frac{2d}{d+q-2}}}
\lesssim_{d,q}\|\rho_t-\mu\|_{L^{p_c}}.
\end{equation}
Taking the supremum in time proves the lemma.
\end{proof}

\section{Stationary solutions and critical points}\label{sec:stationary}
Critical points arise naturally from the question of stationary states for the Wasserstein gradient flow \eqref{eq:wgfmmd}.  A stationary solution is, in the weakest sense, a time-independent distributional solution of the continuity equation.  To relate this distributional notion of stationarity to variational criticality, we use intrinsic first-variation notions close to the terminology of Boufad\`ene and Vialard \cite[Definitions~3.1--3.2]{boufadene2023global}.  Two distinctions are important.  First, a Lagrangian condition says that the velocity field vanishes on the mass actually carried by the measure.  Second, a Wasserstein criticality condition says that all sufficiently small proximal/JKO steps are blocked.  The second notion is stronger in general and is not needed for the constructions below.

The stationary picture is rigid across almost the entire energy-kernel family.  We first treat the energy-distance kernel in odd dimensions using locality of the associated polyharmonic operator, and then obtain the near-complete natural-moment classification by a Riesz-potential argument.  We treat the remaining regime $d=3$, $0<q<1$ separately and contrast its conditional rigidity with explicit one-dimensional non-minimizing critical points.  Finally, we combine the energy--dissipation identity with compactness to obtain asymptotic criticality without additional long-time bounds for $1\le q<2$, and under uniform bounds for $0<q<1$; the corresponding rigidity results then yield conditional convergence to the target.

\begin{definition}[Lagrangian critical point]\label{critical:def:Lcrit}

		Let $\rho,\mu\in\mathcal P(\R^d)$, and let $E_\mu$ be the MMD energy with kernel $K$.  We interpret the first variation $K*(\rho-\mu)$ modulo additive constants.  In particular, for $K(z)=-|z|^q$ and a fixed base point $x_0\in\R^d$, we use the normalized representative
		\begin{equation}\label{critical:eq:normalized-potential}
		u_{\rho,x_0}(x):=\int_{\R^d}\bigl(K(x-y)-K(x_0-y)\bigr)\ud(\rho-\mu)(y).
		\end{equation}
		This integral is locally absolutely convergent for $0<q\le1$ and, for $1<q<2$, whenever $\rho,\mu\in\mathcal P_{q-1}(\R^d)$; these conditions hold in every application below.  If the ordinary convolution is defined, then \eqref{critical:eq:normalized-potential} differs from it only by a finite constant.  We write $u_\rho$ for any such locally integrable representative. 
		
		We use the diagonal convention for the Borel measurable force kernel $\nabla K$ stated following \eqref{eq:wgfmmdparticle}.
		At every point where it is absolutely convergent, set
		\begin{equation}\label{critical:eq:convolution-field}
		g_\rho(x):=\int_{\R^d}\nabla K(x-y)\ud(\rho-\mu)(y),
		\end{equation}
		and extend $g_\rho$ by zero elsewhere.  We require the integral in \eqref{critical:eq:convolution-field} to converge absolutely $\rho$-a.e.\ and $g_\rho\in L^1_{\mathrm{loc}}(\rho)$.  Under the preceding hypotheses, $g_\rho$ is a Borel representative of the weak gradient of $u_\rho$.  We use the following weak-gradient, almost-everywhere formulation (cf.\ \cite[Definition~3.2]{boufadene2023global}). 
		
		We say that $\rho$ is a \emph{Lagrangian critical point} of $E_\mu$ if
		\begin{equation}\label{critical:eq:lagcrit}
		g_\rho=0
		\qquad \rho\text{-a.e.}
		\end{equation}

\end{definition}

If $\rho\,(\nabla K*(\rho-\mu))$ is locally finite, the Lagrangian condition immediately implies stationarity.  Indeed, for every $\varphi\in C_c^\infty(\R^d)$,
\begin{equation}\label{critical:eq:lagrangian-stationarity}
\left\langle\div\bigl(\rho\,(\nabla K*(\rho-\mu))\bigr),\varphi\right\rangle
=-\int_{\R^d}\nabla\varphi\cdot\bigl(\nabla K*(\rho-\mu)\bigr)\ud\rho=0.
\end{equation}
Hence $\rho_t\equiv\rho$ is a stationary distributional solution of \eqref{eq:wgfmmd}.

\begin{remark}[Wasserstein versus Lagrangian criticality]
Boufad\`ene--Vialard call $\rho$ a \emph{Wasserstein critical point} of $E_\mu(\rho):=\mmd^2(\rho,\mu)$ if there is $\tau_0>0$ such that, for every $0<\tau\le\tau_0$,
\begin{equation}\label{critical:eq:Wcrit}
\rho\in\underset{\sigma\in\mathcal P(\R^d)}{\arg\min}\left(E_\mu(\sigma)+\frac1{2\tau}W_2^2(\rho,\sigma)\right);
\end{equation}
see \cite[Definition~3.1]{boufadene2023global}.  Thus all sufficiently small JKO steps started from $\rho$ are blocked.  The Wasserstein notion is stronger in general; throughout this section, we use only the weaker Lagrangian notion of \Cref{critical:def:Lcrit}.
\end{remark}

\subsection{\texorpdfstring{Odd-dimensional local rigidity for energy distance}{Odd-dimensional local rigidity for energy distance}}

We first specialize to the energy-distance kernel $K(z)=-|z|$.
Boufad\`ene--Vialard prove an interior rigidity statement for this kernel in odd dimensions: a Lagrangian critical point agrees with the target on the interior of its support; see \cite[Theorem~3.4]{boufadene2023global}.  If the critical point is absolutely continuous, the same local mechanism gives a global statement without any regularity assumption on the target.  The odd-dimensional restriction belongs to this local proof: for odd $d$, the function $|x|$ is, up to normalization, a fundamental solution for the polyharmonic operator $(-\Delta)^m$ with $m=(d+1)/2$, i.e., 
\begin{equation}\label{critical:eq:polyharmonic-fund}
(-\Delta)^m\big(C_d|x|\big)=\delta_0
\qquad\text{in }\mathcal D'(\R^d)
\end{equation}
for some $C_d\ne0$.

\begin{proposition}[Odd-dimensional energy-distance rigidity]\label{critical:prop:energy-Lp}
Let $d$ be odd and let $m=(d+1)/2$.  Assume $\mu,\rho\in\mathcal P_1(\R^d)$ and $\rho\ll\mathcal{L}^d$.  If $\rho$ is a Lagrangian critical point of $E_\mu$, then $\rho=\mu$.
\end{proposition}

\begin{proof}
Let $u:=K*(\rho-\mu)$.
The finite-first-moment assumption makes $u$ a locally integrable distribution.  By \eqref{critical:eq:polyharmonic-fund},
\begin{equation}
(-\Delta)^m u = c_d(\rho-\mu)
\qquad\text{in }\mathcal D'(\R^d)
\end{equation}
for some nonzero constant $c_d$.

Let $r:\R^d\to[0,\infty)$ be a Borel measurable function such that $\rho=r\,\mathcal{L}^d$, and let
\begin{equation}
	\mu=g\,\mathcal{L}^d+\mu_{\mathrm{s}}
\end{equation}
be the Lebesgue decomposition of $\mu$, where $g\in L^1(\R^d)$ is nonnegative and $\mu_{\mathrm{s}}\perp\mathcal{L}^d$ is nonnegative.  Set
\begin{equation}
	\nu:=\rho-\mu=(r-g)\,\mathcal{L}^d-\mu_{\mathrm{s}}
\end{equation}
and set $A:=\{x:r(x)>0\}$.  The Lagrangian criticality condition gives $\nabla u=0$ $\rho$-a.e.  Because $r>0$ on $A$, this implies $\nabla u=0$ $\mathcal L^d$-a.e.\ on $A$.

Suppose first that $d=1$ and put $v:=u'$.  Since $(|\cdot|)'=\sgn$ distributionally and $\nu$ is a finite signed measure, we have $v=-\sgn*\nu\in L^\infty(\R)$.  Moreover,
\begin{equation}
	Dv=u''=-c_1\nu,
\end{equation}
so $v\in BV_{\mathrm{loc}}(\R)$.  Criticality gives $v=0$ $\mathcal L^1$-a.e.\ on $A$.  By the standard level-set locality property for $BV$ functions \cite[Theorems~6.3--6.4 and the following remark]{EvansGariepy2015}, the absolutely continuous part of $Dv$ vanishes on $A$.  Since
\begin{equation}
	(Dv)^{\mathrm{a}}=-c_1(r-g)\,\mathcal{L}^1,
\end{equation}
we obtain $r=g$ $\mathcal L^1$-a.e.\ on $A$.

Now suppose that $d\ge3$.  For every multi-index $\alpha$ with $1\le|\alpha|\le d$, the distributional derivative $D^\alpha|x|$ is represented by a homogeneous, locally integrable function of degree $1-|\alpha|$, which is bounded at infinity.  Convolution against the finite signed measure $\nu$ therefore gives
\begin{equation}
	u\in W^{d,1}_{\mathrm{loc}}(\R^d).
\end{equation}
Applying the same level-set property successively to the weak derivatives of $u$ shows that every derivative of $u$ of order between $1$ and $d-1$ vanishes $\mathcal L^d$-a.e.\ on $A$.  In particular, with $w:=(-\Delta)^{m-1}u$, we have $w=0$ $\mathcal L^d$-a.e.\ on $A$, while
\begin{equation}
	-\Delta w=c_d\nu
	\qquad\text{in }\mathcal D'(\R^d).
\end{equation}
	Let $w^*$ denote the precise (Lebesgue) representative of $w$.\footnote{We use the standard precise representative $w^*(x):=\lim_{r\downarrow0}\fint_{B_r(x)}w(y)\,\ud y$ whenever the limit exists, and set $w^*(x)=0$ otherwise; see \cite[p.~46]{EvansGariepy2015}.}  For $\mathcal L^d$-a.e. $x\in A$, the point $x$ is both a density point of $A$ and a Lebesgue point of $w$, and $w^*(x)=0$.  Hence $A\subset\{w^*=0\}$ up to a Lebesgue-null set.  Thus $\Delta w$ is a signed Radon measure, with
\begin{equation}
	(\Delta w)^{\mathrm{a}}=-c_d(r-g)\,\mathcal{L}^d.
\end{equation}
By \cite[Theorem~1.1]{AmbrosioPonceRodiac2020}, the absolutely continuous part of $\Delta w$ vanishes $\mathcal L^d$-a.e.\ on the level set $\{w^*=0\}$, and hence $\mathcal L^d$-a.e.\ on $A$.  It follows that $r=g$ $\mathcal L^d$-a.e.\ on $A$.

Since $r=0$ $\mathcal L^d$-a.e.\ on $A^c$ and $r=g$ $\mathcal L^d$-a.e.\ on $A$,
\begin{equation}
	1=\rho(A)=\int_A r\ud x=\int_A g\ud x.
\end{equation}
On the other hand,
\begin{equation}
	1=\mu(\R^d)=\int_A g\ud x+\int_{A^c}g\ud x+\mu_{\mathrm{s}}(\R^d).
\end{equation}
All terms on the right-hand side are nonnegative, and the first equals one.  Hence $g=0$ $\mathcal L^d$-a.e.\ on $A^c$ and $\mu_{\mathrm{s}}=0$.  Therefore $\mu=r\,\mathcal{L}^d=\rho$.
\end{proof}

\subsection{\texorpdfstring{Natural-moment rigidity for energy kernels}{Natural-moment rigidity for energy kernels}}\label{ssec:natural-moment-rigidity}

The preceding odd-dimensional energy-distance rigidity result is the local polyharmonic
case of a broader Riesz-potential mechanism.  The following result establishes
natural-moment rigidity except when $d\in\{1,3\}$ and $0<q<1$.
The residual three-dimensional regime is treated separately in
\Cref{critical:thm:residual-three-rigidity}.

\begin{theorem}[Rigidity under the $q$-moment condition]\label{critical:thm:natural-moment-rigidity}
Let $d\ge1$ and $0<q<2$.  Assume $\mu,\rho\in\mathcal P_q(\R^d)$,
$\rho\ll\mathcal{L}^d$, and that $\rho$ is a Lagrangian critical point of
$E_\mu$.  Then $\rho=\mu$ in each of the following
regimes:
\begin{enumerate}[label=\textup{(\roman*)}]
\item $d$ is even;
\item $d$ is odd and $1\le q<2$;
\item $d\ge5$ is odd and $0<q<1$.
\end{enumerate}
\end{theorem}

Write $\mathcal M(\R^d)$ for the space of finite signed Borel measures
on $\R^d$ and $\mathcal M_+(\R^d)$ for its positive cone.

Throughout the Riesz-potential argument, set $\nu:=\rho-\mu$ and
$u:=K*\nu$.
The local Sobolev regularity of $u$ follows from the moment assumptions
alone; it does not use absolute continuity of either measure.  The exact
source-dependent condition used below is
\begin{equation}\label{critical:eq:rigidity-positive-part-criticality}
 \nu^+\ll\mathcal{L}^d,
 \qquad
 \nabla u=0\quad\nu^+\text{-a.e.}
\end{equation}
The assumptions of \Cref{critical:thm:natural-moment-rigidity} imply this
condition because $\nu^+=(\rho-\mu)^+\le\rho$.
Thus the moment assumptions provide the regularity of $u$, while absolute
continuity of $\rho$ and Lagrangian criticality provide precisely the two
conditions in \eqref{critical:eq:rigidity-positive-part-criticality}.

When $s\notin\mathbb N$, set
\begin{equation}\label{critical:eq:rigidity-residual-parameters}
 s:=\frac{d+q}{2},\qquad
 m:=\floor{s},\qquad
 \alpha:=2(s-m)=d+q-2m\in(0,2).
\end{equation}
The equality $s\in\mathbb N$ occurs exactly when $d$ is odd and $q=1$;
that case has already been proved in
\Cref{critical:prop:energy-Lp}.  We now treat the rigidity
regimes by a common Riesz-potential reduction. We first derive from Lagrangian criticality an equality of Riesz
potentials on the positive part of the discrepancy.

For $0<\alpha<d$ and $\lambda\in\mathcal M(\R^d)$, denote the Riesz
potential of order $\alpha$ of $\lambda$ by
\begin{equation}\label{critical:eq:rigidity-riesz-potential}
 \mathcal I_\alpha\lambda(x)
 :=c_{d,\alpha}\int_{\R^d}|x-y|^{\alpha-d}\,\ud\lambda(y),
\end{equation}
where $c_{d,\alpha}>0$ is normalized so that
$\widehat{\mathcal I_\alpha f}(\xi)=|2\pi\xi|^{-\alpha}\widehat f(\xi)$ for
Schwartz functions. 
The Fourier multiplier, composition, and inversion properties used below
are standard; see \cite[Chapter~V, \S1]{Stein1970singular}.

We begin with the local regularity and measurable-set locality used in the reduction.

\begin{lemma}[Local Sobolev regularity]\label{critical:lem:power-regularity}
Let $\eta\in\mathcal M(\R^d)$ be such that
\begin{equation}
 \int_{\R^d}|y|^q\,\ud|\eta|(y)<\infty.
\end{equation}
If $\ell$ is a nonnegative integer satisfying $\ell<d+q$, then
$K*\eta\in W^{\ell,1}_{\mathrm{loc}}(\R^d)$.  For every multi-index $\beta$ with
$1\le|\beta|\le \ell$,
\begin{equation}\label{critical:eq:rigidity-derivative-convolution}
 D^\beta(K*\eta)=D^\beta K*\eta
 \qquad\text{in }\mathcal D'(\R^d).
\end{equation}
\end{lemma}

\begin{proof}
For $1\le|\beta|<d+q$, the classical derivative $D^\beta K$ on
$\R^d\setminus\{0\}$ is homogeneous of degree $q-|\beta|$ and hence locally
integrable.  Repeated integration by parts on
$\R^d\setminus\overline B_\varepsilon$ produces boundary terms of order
$\varepsilon^{d+q-|\beta|}$, which vanish as $\varepsilon\downarrow0$.
Thus the distributional derivative $D^\beta K$ equals this locally
integrable classical derivative.

For every $B\Subset\R^d$ and $1\le|\beta|\le\ell$, splitting according to
whether $y$ lies in a fixed enlargement of $B$ gives
\begin{equation}
 \int_B \bigl|D^\beta(|x-y|^q)\bigr|\,\ud x
 \le C_B(1+|y|^q).
\end{equation}
The moment assumption and Fubini's theorem now yield
\eqref{critical:eq:rigidity-derivative-convolution}.  The case $\ell=0$
follows from $|x-y|^q\le C_q(|x|^q+|y|^q)$.
\end{proof}

The first-order locality property follows, for example, from
\cite[Theorem~6.19]{LiebLoss}: if
$f\in W^{1,1}_{\mathrm{loc}}$ is constant $\mathcal L^d$-a.e.\ on a Lebesgue-measurable set,
then $\nabla f=0$ $\mathcal L^d$-a.e.\ on that set. Iterating this result over the
weak derivatives of $f$ gives the following higher-order form.

\begin{lemma}[Sobolev locality on a measurable set]\label{critical:lem:sobolev-locality}
Let $E\subset\R^d$ be Lebesgue measurable, let $j\ge1$ be an integer, and let
$f\in W^{j,1}_{\mathrm{loc}}(\R^d)$.  If $f=0$ $\mathcal L^d$-a.e.\ on $E$,
then
\begin{equation}
 D^\beta f=0
 \qquad\mathcal L^d\text{-a.e.\ on }E
\end{equation}
for every multi-index $\beta$ with $1\le|\beta|\le j$.
\end{lemma}

We next record the elementary domination properties of the Jordan parts.

\begin{lemma}[Domination of the Jordan parts]\label{critical:lem:jordan-domination}
Let $\rho,\mu\in\mathcal M_+(\R^d)$.  Then
\begin{equation}\label{critical:eq:rigidity-jordan-domination}
 (\rho-\mu)^+\le\rho,
 \qquad
 (\rho-\mu)^-\le\mu.
\end{equation}
If $\rho(\R^d)=\mu(\R^d)$, then $(\rho-\mu)^+$ and $(\rho-\mu)^-$ have equal mass.
Consequently, $\rho\ll\mathcal{L}^d$ implies
$(\rho-\mu)^+\ll\mathcal{L}^d$, and every condition holding
$\rho$-a.e.\ also holds $(\rho-\mu)^+$-a.e.
\end{lemma}

\begin{proof}
For every Borel set $E$,
\begin{equation}
 (\rho-\mu)^+(E)=\sup_{F\subset E}(\rho-\mu)(F)
 \le\sup_{F\subset E}\rho(F)
 =\rho(E).
\end{equation}
The inequality for $(\rho-\mu)^-$ follows by interchanging $\rho$ and $\mu$.
The remaining assertions are immediate.
\end{proof}

These facts give the Riesz-potential reduction at the remaining
fractional order $\alpha$.

\begin{proposition}[Integer reduction to finite Riesz contact]\label{critical:prop:residual-contact}
Let $d\ge1$, let $0<q<2$, and let
$\rho,\mu\in\mathcal M_+(\R^d)$ be such that
\begin{equation}
 \int_{\R^d}|y|^q\,\ud(\rho+\mu)(y)<\infty.
\end{equation}
Set
\begin{equation}
 \nu:=\rho-\mu,
 \qquad
 u:=K*\nu.
\end{equation}
Let $m\ge1$ be an integer and set
\begin{equation}\label{critical:eq:rigidity-alpha-general}
 \alpha:=d+q-2m.
\end{equation}
Assume $0<\alpha<d$.  Then
\begin{equation}\label{critical:eq:rigidity-integer-reduction}
 (-\Delta)^m u
 =b_{d,q,m}\,\mathcal I_\alpha\nu
 \qquad\text{in }\mathcal D'(\R^d),
\end{equation}
where $b_{d,q,m}>0$.  If
\eqref{critical:eq:rigidity-positive-part-criticality} holds, write
$\nu^+=h\,\mathcal{L}^d$ and set
\begin{equation}\label{critical:eq:rigidity-positive-density-set}
 B:=\{x\in\R^d:h(x)>0\}.
\end{equation}
Then
\begin{equation}\label{critical:eq:rigidity-contact-on-B}
 \mathcal I_\alpha\nu=0
 \qquad\text{for $\mathcal L^d$-a.e. }x\in B,
\end{equation}
and consequently
\begin{equation}\label{critical:eq:rigidity-contact-sigma}
 \mathcal I_\alpha\nu^+=\mathcal I_\alpha\nu^-<\infty
 \qquad\nu^+\text{-a.e.}
\end{equation}
We introduce the term \emph{finite Riesz contact} for the relation
\eqref{critical:eq:rigidity-contact-sigma} between $\nu^+$ and $\nu^-$.
\end{proposition}

\begin{proof}
For $x\ne0$ and any real $a$,
\begin{equation}
 \Delta|x|^a=a(a+d-2)|x|^{a-2}.
\end{equation}
Since $2m=d+q-\alpha<d+q$, \Cref{critical:lem:power-regularity} shows that all
intermediate distributional derivatives are the classical homogeneous
ones and that no origin-supported term occurs.  Hence
\begin{align}
 (-\Delta)^mK(x)=(-\Delta)^m(-|x|^q)
 &=(-1)^{m+1}
   \prod_{j=0}^{m-1}(q-2j)(q+d-2-2j)|x|^{q-2m} \notag\\
 &=a_{d,q,m}|x|^{\alpha-d}.
 \label{critical:eq:rigidity-reduction-product}
\end{align}
Every factor $q+d-2-2j$ is positive, its smallest value being
$q+d-2m=\alpha$.  Among the factors $q-2j$, the factor with $j=0$ is
positive and the remaining $m-1$ factors are negative.  The sign
$(-1)^{m-1}$ of their product is canceled by $(-1)^{m+1}$, so
$a_{d,q,m}>0$.  Since
$|\cdot|^{\alpha-d}*\nu=c_{d,\alpha}^{-1}\mathcal I_\alpha\nu$,
convolution with $\nu$ proves
\eqref{critical:eq:rigidity-integer-reduction} with
$b_{d,q,m}:=a_{d,q,m}/c_{d,\alpha}>0$.

The regularity used here is supplied entirely by the moment hypothesis:
\Cref{critical:lem:power-regularity} gives
$u\in W^{2m,1}_{\mathrm{loc}}(\R^d)$.  Because
$\nu^+=h\,\mathcal{L}^d$, the second condition in
\eqref{critical:eq:rigidity-positive-part-criticality} implies that
$\nabla u=0$ $\mathcal L^d$-a.e.\ on $B$.  Each component of
$\nabla u$ belongs to $W^{2m-1,1}_{\mathrm{loc}}$.  Applying
\Cref{critical:lem:sobolev-locality} to these components shows that every
weak derivative of $u$ of order between $1$ and $2m$ vanishes $\mathcal L^d$-a.e.\ on
$B$.  In particular, $(-\Delta)^m u=0$ there, and
\eqref{critical:eq:rigidity-integer-reduction} proves
\eqref{critical:eq:rigidity-contact-on-B}.

The two positive Riesz potentials in
\eqref{critical:eq:rigidity-contact-sigma} are locally integrable and
hence finite $\mathcal L^d$-a.e.  Since
$\nu^+\ll\mathcal{L}^d$ and is concentrated on $B$, the locally integrable
identity \eqref{critical:eq:rigidity-contact-on-B} yields the finite
contact \eqref{critical:eq:rigidity-contact-sigma}.
\end{proof}

We next show that the resulting Riesz contact is rigid. The key potential-theoretic ingredient is the complete maximum principle for
Riesz kernels.  If $n\ge2$, $0<\beta\le2$, $\beta<n$, $\lambda$ is a
positive measure of finite $\beta$-Riesz energy, and
$\gamma\in\mathcal M_+(\R^n)$, then
\begin{equation}
 \mathcal I_\beta\lambda\le\mathcal I_\beta\gamma
 \quad\lambda\text{-a.e.}
 \quad\Longrightarrow\quad
 \mathcal I_\beta\lambda\le\mathcal I_\beta\gamma
 \quad\text{everywhere}.
\end{equation}
This is the zero-additive-constant case of the complete maximum principle
in \cite[Theorem~2.2]{Zorii2023Deny}.  The one-dimensional case needed below is
obtained by embedding $\R$ as a line in $\R^2$: an order-$\alpha$ kernel
on $\R$, $0<\alpha<1$, is the restriction of the order-$(\alpha+1)$ kernel
on $\R^2$.

A finite Riesz contact need not itself have finite energy, so we first
approximate its positive measure from below by finite-energy measures to
which the complete maximum principle applies.

\begin{lemma}[Exhaustion by finite-energy measures]\label{critical:lem:riesz-exhaustion}
Let $0<\alpha<d$ and let $\sigma\in\mathcal M_+(\R^d)$ satisfy
\begin{equation}\label{critical:eq:rigidity-self-potential-finite}
 \mathcal I_\alpha\sigma<\infty
 \qquad\sigma\text{-a.e.}
\end{equation}
Then there are $\sigma_j\in\mathcal M_+(\R^d)$ with
$\sigma_j\uparrow\sigma$ such that each $\sigma_j$ has finite
$\alpha$-Riesz energy.
\end{lemma}

\begin{proof}
The potential $\mathcal I_\alpha\sigma$ is lower semicontinuous. For each $j\ge 1$, let $\overline B_j$ denote the closed ball of radius $j$ centered at the origin, and define
\begin{equation}
 D_j:=\overline B_j\cap\{\mathcal I_\alpha\sigma\le j\},
 \qquad
 \sigma_j:=\sigma|_{D_j}.
\end{equation}
The sets $D_j$ are compact and increasing, and
\eqref{critical:eq:rigidity-self-potential-finite} implies
$\sigma_j\uparrow\sigma$.  Moreover,
$\mathcal I_\alpha\sigma_j\le\mathcal I_\alpha\sigma$, so
\begin{equation}
 \int_{\R^d}\mathcal I_\alpha\sigma_j\,\ud\sigma_j
 \le\int_{D_j}\mathcal I_\alpha\sigma\,\ud\sigma
 \le j\sigma(D_j)<\infty.
\end{equation}
Thus $\sigma_j$ has finite $\alpha$-Riesz energy.
\end{proof}

The following cancellation statement is the part of the argument that uses
the moment assumption quantitatively.

\begin{proposition}[One-sign cancellation below order two]\label{critical:prop:riesz-one-sign}
Let $d\ge1$ and $0<\alpha<\min\{2,d\}$.  Let
$\eta\in\mathcal M(\R^d)$ satisfy
\begin{equation}\label{critical:eq:rigidity-cancellation-assumptions}
 \eta(\R^d)=0,
 \qquad
 \int_{\R^d}|y|^\alpha\,\ud|\eta|(y)<\infty.
\end{equation}
If $\mathcal I_\alpha\eta\ge0$ $\mathcal L^d$-a.e., or if
$\mathcal I_\alpha\eta\le0$ $\mathcal L^d$-a.e., then $\eta=0$.
\end{proposition}

\begin{proof}
Suppose that $\mathcal I_\alpha\eta\le0$ $\mathcal L^d$-a.e.  The other sign is reduced to
this one by replacing $\eta$ with $-\eta$.  For $R>0$, define
\begin{equation}
 F_R(y):=c_{d,\alpha}\int_{B_R}|x-y|^{\alpha-d}\,\ud x,
 \qquad
 J_R:=\int_{B_R}\mathcal I_\alpha\eta(x)\,\ud x.
\end{equation}
The function $F_R$ is radial and nonincreasing in $|y|$: this follows, for
example, from the layer-cake representation of the radial decreasing
kernel and the fact that the intersection of two balls has maximal volume
when their centers coincide.  In particular,
\begin{equation}\label{critical:eq:rigidity-FR-bound}
 0\le F_R(y)\le F_R(0)
 =c_{d,\alpha}\frac{|\mathbb S^{d-1}|}{\alpha}R^\alpha.
\end{equation}
Fubini's theorem is therefore applicable, and the zero-mass condition gives
\begin{equation}\label{critical:eq:rigidity-JR-zero-mass}
 J_R=\int_{\R^d}\bigl(F_R(y)-F_R(0)\bigr)\,\ud\eta(y).
\end{equation}

We claim that, uniformly in $R>0$,
\begin{equation}\label{critical:eq:rigidity-FR-holder}
 |F_R(y)-F_R(0)|\le C_{d,\alpha}|y|^\alpha,
\end{equation}
and that the left-hand side tends to zero as $R\to\infty$ for every fixed
$y$.  If $|y|>R/2$, \eqref{critical:eq:rigidity-FR-bound} gives
\begin{equation}
 |F_R(y)-F_R(0)|\le2F_R(0)
 \le C R^\alpha\le C|y|^\alpha.
\end{equation}
		If $|y|\le R/2$, let $k(z):=c_{d,\alpha}|z|^{\alpha-d}$.  For $d\ge2$, differentiating
	inside the integral and applying the divergence theorem
	give
	\begin{equation}
	 \partial_{y_i y_j}F_R(y)
	 =\int_{\partial B_R}\partial_{x_j}k(x-y)n_i(x)\,\ud\mathcal H^{d-1}(x).
	\end{equation}
	Since $|x-y|\ge R/2$ for $x\in\partial B_R$,
\begin{equation}
 \norm{D^2F_R(y)}\le C R^{d-1}R^{\alpha-d-1}
 =C R^{\alpha-2}.
\end{equation}
	For $d=1$, necessarily $0<\alpha<1$, and direct differentiation at the
	two endpoints gives, for $|y|<R$,
	\begin{equation}
		 F_R'(y)=c_{1,\alpha}\bigl((R+y)^{\alpha-1}-(R-y)^{\alpha-1}\bigr),
	\end{equation}
	which yields the same bound $|F_R''(y)|\le C R^{\alpha-2}$ when
	$|y|\le R/2$.
	By symmetry, $\nabla F_R(0)=0$, so Taylor's theorem gives
\begin{equation}\label{critical:eq:rigidity-FR-sharp}
 |F_R(y)-F_R(0)|
 \le C|y|^2R^{\alpha-2}
 \le C|y|^\alpha,
\end{equation}
where the final inequality uses $R\ge2|y|$ and $\alpha<2$.  The sharper
	bound in \eqref{critical:eq:rigidity-FR-sharp} also tends to zero as $R\to\infty$ for fixed
	$y$.

The moment assumption and dominated convergence in
\eqref{critical:eq:rigidity-JR-zero-mass} now imply
\begin{equation}\label{critical:eq:rigidity-JR-limit}
 \lim_{R\to\infty}J_R=0.
\end{equation}
Since $\mathcal I_\alpha\eta\le0$ $\mathcal L^d$-a.e., the map $R\mapsto J_R$ is nonincreasing and nonpositive.
For $S>R$, one has $J_S\le J_R\le0$; letting $S\to\infty$ and using
\eqref{critical:eq:rigidity-JR-limit} gives $J_R=0$.  Thus
$\mathcal I_\alpha\eta=0$ $\mathcal L^d$-a.e.

Finally, the Riesz inversion identity gives
\begin{equation}\label{critical:eq:rigidity-riesz-inversion}
 (-\Delta)^{\alpha/2}\mathcal I_\alpha\eta=\eta
 \qquad\text{in }\mathcal S'(\R^d);
\end{equation}
see \cite[Chapter~V, \S1]{Stein1970singular}.
Since $\mathcal I_\alpha\eta=0$ as a locally integrable
distribution, \eqref{critical:eq:rigidity-riesz-inversion} yields $\eta=0$.
\end{proof}

Combining the finite-energy exhaustion and the complete maximum principle with
the preceding one-sign cancellation yields the following contact rigidity statement. No absolute-continuity hypothesis is required.

\begin{proposition}[Riesz contact rigidity]\label{critical:prop:riesz-contact}
Let $d\ge1$ and $0<\alpha<\min\{2,d\}$.  Let
$\sigma,\tau\in\mathcal M_+(\R^d)$ have equal mass and satisfy
\begin{equation}\label{critical:eq:rigidity-contact-moment}
 \int_{\R^d}|x|^\alpha\,\ud(\sigma+\tau)(x)<\infty.
\end{equation}
If the canonical Riesz potentials\footnote{Recall that, for
$\lambda\in\mathcal M_+(\R^d)$, the integral in
\eqref{critical:eq:rigidity-riesz-potential} defines
$\mathcal I_\alpha\lambda(x)\in[0,\infty]$ for every $x\in\R^d$.} satisfy
\begin{equation}\label{critical:eq:rigidity-abstract-contact}
 \mathcal I_\alpha\sigma=\mathcal I_\alpha\tau<\infty
 \qquad\sigma\text{-a.e.},
\end{equation}
then $\sigma=\tau$.  
\end{proposition}

\begin{proof}
Let $\sigma_j\uparrow\sigma$ be the exhaustion from
\Cref{critical:lem:riesz-exhaustion}.  Since $\sigma_j\le\sigma$, the
contact equality implies
\begin{equation}
 \mathcal I_\alpha\sigma_j
 \le\mathcal I_\alpha\sigma
 =\mathcal I_\alpha\tau
 \qquad\sigma_j\text{-a.e.}
\end{equation}
For $d\ge2$, the complete maximum principle
\cite[Theorem~2.2]{Zorii2023Deny} yields
\begin{equation}
 \mathcal I_\alpha\sigma_j\le\mathcal I_\alpha\tau
 \qquad\text{everywhere on }\R^d.
\end{equation}
For $d=1$, push the measures forward by $x\mapsto(x,0)$ into $\R^2$ and
apply the same theorem at order $\alpha+1\in(1,2)$; the resulting kernel
restricted to the line is a positive multiple of the order-$\alpha$ kernel
on $\R$.  Monotone convergence gives
\begin{equation}\label{critical:eq:rigidity-global-domination}
 \mathcal I_\alpha\sigma\le\mathcal I_\alpha\tau
 \qquad\text{everywhere, in the extended-valued sense}.
\end{equation}
The two potentials are locally integrable and hence finite
$\mathcal L^d$-a.e.  Therefore
$\mathcal I_\alpha(\sigma-\tau)\le0$ $\mathcal L^d$-a.e.  The signed
measure $\sigma-\tau$ has zero mass and finite $\alpha$-moment, so
\Cref{critical:prop:riesz-one-sign} gives $\sigma=\tau$.
\end{proof}

The preceding results immediately yield the single-contact mechanism.

\begin{theorem}[Single-contact rigidity]\label{critical:thm:single-contact}
Let $d\ge1$, let $0<q<2$, let $m\ge1$ be an integer, and set
\begin{equation}
 \alpha=d+q-2m\in(0,\min\{2,d\}).
\end{equation}
Let $\rho,\mu\in\mathcal M_+(\R^d)$ have equal mass, and set
\begin{equation}
 \nu:=\rho-\mu,
 \qquad
 u:=K*\nu.
\end{equation}
Assume that $\nu^+\ll\mathcal{L}^d$ and that
\begin{equation}\label{critical:eq:rigidity-single-moments}
 \int_{\R^d}|x|^{\max\{q,\alpha\}}
 \ud(\rho+\mu)(x)<\infty.
\end{equation}
If $\nabla u=0$ $\nu^+$-a.e., then $\rho=\mu$.
\end{theorem}

\begin{proof}
By \Cref{critical:prop:residual-contact}, the Jordan parts $\nu^+$ and
$\nu^-$ have finite contact
$\mathcal I_\alpha\nu^+=\mathcal I_\alpha\nu^-$
$\nu^+$-a.e.  They have equal mass by
\Cref{critical:lem:jordan-domination}, and the moment assumption gives the
finite $\alpha$-moment required by \Cref{critical:prop:riesz-contact}.
Hence $\nu^+=\nu^-$, and therefore $\rho=\mu$.
\end{proof}

\begin{corollary}[Regimes for single-contact rigidity]\label{critical:cor:single-contact}
In the setting of \Cref{critical:thm:natural-moment-rigidity}, single-contact rigidity
applies in each of the following cases:
\begin{enumerate}[label=\textup{(\roman*)}]
 \item if $d=2k$ is even, take $m=k$ and $\alpha=q$;
 \item if $d=2k+1$ is odd and $1<q<2$, take $m=k+1$ and
       $\alpha=q-1$;
 \item if $d=2k+1\ge3$ is odd and $0<q<1$, take $m=k$ and
       $\alpha=q+1$, provided in this last case that
       $\rho,\mu\in\mathcal P_{q+1}(\R^d)$.
\end{enumerate}
\end{corollary}

\begin{proof}
By \Cref{critical:lem:jordan-domination}, the assumptions of
\Cref{critical:thm:natural-moment-rigidity} imply
\eqref{critical:eq:rigidity-positive-part-criticality}.  The displayed values of
$\alpha$ follow from $\alpha=d+q-2m$.  In case (i), the required residual moment is exactly the
natural $q$-moment.  In case (ii), a finite $q$-moment implies a finite
$(q-1)$-moment.  Case (iii) is the additional-moment statement.
\end{proof}

Only case (iii) fails to close at the natural moment.  In odd dimensions at
least five, the Lagrangian condition supplies a second contact that removes
this loss.

\begin{theorem}[Two-contact rigidity]\label{critical:thm:two-contact}
Let $d=2k+1\ge5$ and $0<q<1$.  Let
$\rho,\mu\in\mathcal P_q(\R^d)$ be probability measures, and set
\begin{equation}
 \nu:=\rho-\mu,
 \qquad
 u:=K*\nu.
\end{equation}
Assume
\begin{equation}
 \nu^+\ll\mathcal{L}^d,
 \qquad
 \nabla u=0\quad\nu^+\text{-a.e.}
\end{equation}
Then $\rho=\mu$.
\end{theorem}

\begin{proof}
Set
\begin{equation}
 \alpha:=q+1\in(1,2),
 \qquad
 \nu^+=h\,\mathcal{L}^d,
 \qquad
 B:=\{h>0\}.
\end{equation}
Applying \Cref{critical:prop:residual-contact} first with $m=k$ and then
with $m=k-1$ gives
\begin{equation}\label{critical:eq:rigidity-two-contacts}
 \mathcal I_\alpha\nu=0,
 \qquad
 \mathcal I_{\alpha+2}\nu=0
 \qquad\mathcal L^d\text{-a.e.\ on }B.
\end{equation}
The second reduction lies in the Riesz range because
\begin{equation}
 \alpha+2=q+3<d.
\end{equation}
The first finite contact, the exhaustion from
\Cref{critical:lem:riesz-exhaustion}, and the complete maximum principle
give, exactly as in the proof of \Cref{critical:prop:riesz-contact},
\begin{equation}\label{critical:eq:rigidity-first-contact-domination}
 \mathcal I_\alpha\nu^+\le\mathcal I_\alpha\nu^-
 \qquad\text{everywhere}.
\end{equation}
Define
\begin{equation}\label{critical:eq:rigidity-def-V}
 V:=-\mathcal I_\alpha\nu\ge0
 \qquad\mathcal L^d\text{-a.e.}
\end{equation}
Assume for the sake of a contradiction that $\nu\ne0$.  Riesz inversion implies that $V$ does not vanish
$\mathcal L^d$-a.e.
Thus $V>0$ on a set of positive Lebesgue measure.  Moreover,
$\nu^+\ne0$: otherwise $\nu\le0$ and $\nu(\R^d)=0$ would imply
$\nu=0$.  Since $\nu^+=h\,\mathcal{L}^d$, the set $B$ has positive
Lebesgue measure.

Choose $x\in B$ such that the second equality in
\eqref{critical:eq:rigidity-two-contacts} holds and both positive
potentials $\mathcal I_{\alpha+2}\nu^+(x)$ and
$\mathcal I_{\alpha+2}\nu^-(x)$ are finite.  Such points form a
full-measure subset of $B$, because potentials of finite measures are
locally integrable.  The Riesz composition identity
\begin{equation}\label{critical:eq:rigidity-riesz-composition}
 \mathcal I_2\bigl(\mathcal I_\alpha\lambda\,\mathcal{L}^d\bigr)
 =\mathcal I_{\alpha+2}\lambda,
 \qquad \alpha+2<d,
\end{equation}
holds for every positive measure $\lambda$ by Fubini-Tonelli and the
classical convolution identity; see \cite[Chapter~V, \S1]{Stein1970singular}.
Applying it separately to $\nu^{\pm}$ at the selected point, where
both sides are finite, gives
\begin{equation}
 \mathcal I_2(V\,\mathcal{L}^d)(x)
 =\mathcal I_{\alpha+2}\nu^-(x)-\mathcal I_{\alpha+2}\nu^+(x)=0.
\end{equation}
On the other hand, the order-two Riesz kernel is strictly positive and
$V\ge0$ is positive on a set of positive measure.  Hence
$\mathcal I_2(V\,\mathcal{L}^d)(x)>0$, which is a
contradiction.  Therefore $\nu=0$.
\end{proof}

\begin{proof}[Proof of \Cref{critical:thm:natural-moment-rigidity}]
By \Cref{critical:lem:jordan-domination}, the absolute continuity of $\rho$
and the Lagrangian condition imply
\eqref{critical:eq:rigidity-positive-part-criticality}.
Let $d=2k$ be even.  Then the residual parameters in
\eqref{critical:eq:rigidity-residual-parameters} satisfy $m=k$ and
$\alpha=q$, so \Cref{critical:cor:single-contact}(i) applies under exactly
the natural $q$-moment.

Let $d=2k+1$ be odd.  If $1<q<2$, then $m=k+1$ and $\alpha=q-1$, so
\Cref{critical:cor:single-contact}(ii) applies.  If $q=1$, use
\Cref{critical:prop:energy-Lp}.  If $0<q<1$ and $d\ge5$, use
\Cref{critical:thm:two-contact}.  This proves the theorem.
\end{proof}

\begin{remark}[Role of absolute continuity]
The moment assumptions alone give all local Sobolev regularity used in the
integer reduction.  Absolute continuity of $\rho$ enters the theorem only
through the implications
\begin{equation}
 (\rho-\mu)^+\ll\mathcal{L}^d,
 \qquad
 \nabla u=0\quad(\rho-\mu)^+\text{-a.e.},
\end{equation}
which follow from $(\rho-\mu)^+\le\rho$.  Indeed, writing
$\nu^+=h\mathcal L^d$ and $B:=\{h>0\}$, the Lagrangian condition implies
$\nabla u=0$ $\mathcal L^d$-a.e.\ on $B$.  Sobolev locality can therefore
be applied on $B$, and the resulting finite Riesz contact holds $\nu^+$-a.e.\ because
$\nu^+$ is concentrated on $B$.  From that point on,
\Cref{critical:prop:riesz-contact} uses no absolute continuity.

No hypothesis of compact support, positivity on the support, bounded
density, or absolute continuity of the target enters the proof.  The relevant
positive-density set is merely measurable, and the singular part of
$\mu$ is permitted.
\end{remark}

\subsection{\texorpdfstring{The residual three-dimensional regime}{The residual three-dimensional regime}}\label{ssec:residual-three-rigidity}

Among the regimes with $d+q>2$ and $(d+q)/2\notin\mathbb N$, the only regime not covered by
\Cref{critical:thm:natural-moment-rigidity} is $d=3$ and $0<q<1$.
Finite moments of order $q+1$ close the single-contact argument directly.
At the natural $q$-th moment level, it suffices that
$\supp(\rho-\mu)^+$ be compact.  If $\rho-\mu$ is invariant under rotations
about some point, the radial flux identity together with the Lagrangian
criticality condition forces $(\rho-\mu)^+$ to be compactly supported,
reducing the radial case to the compact-support case.
We first record the equilibrium-potential facts used by the compactness
argument and then consolidate the three alternatives.

\begin{lemma}[Compact Riesz equilibrium potentials]\label{critical:lem:compact-riesz-equilibrium}
Let $d\ge1$, let $0<\beta<\min\{2,d\}$, and let $F\subset\R^{d}$ be
compact with positive $\beta$-Riesz capacity.\footnote{Equivalently, $F$
supports a nonzero finite positive measure $\gamma$ of finite
$\beta$-Riesz energy $\int_{\R^d}\mathcal I_\beta\gamma\,\ud\gamma$.
A property holds quasi-everywhere on $F$ if it fails only on a subset of
$F$ of $\beta$-Riesz capacity zero.  Since every such subset is
Lebesgue-null, quasi-everywhere is stronger than $\mathcal L^d$-almost
everywhere; see
\cite[Chapter~II, \S~1]{Landkof1972Foundations}.}  There is a finite positive
measure $\gamma_F$ supported on $F$ which can be normalized such that
\begin{equation}\label{critical:eq:compact-equilibrium-potential}
 0\le \mathcal I_\beta\gamma_F\le1\quad\text{on }\R^d,
 \qquad
 \mathcal I_\beta\gamma_F=1\quad\text{quasi-everywhere on }F,
\end{equation}
and
\begin{equation}\label{critical:eq:compact-equilibrium-strict}
 \mathcal I_\beta\gamma_F(x)<1
 \qquad\text{for every }x\in\R^d\setminus F.
\end{equation}
\end{lemma}

\begin{proof}
Set $V_F:=\mathcal I_\beta\gamma_F$.
The equilibrium-measure properties in
\eqref{critical:eq:compact-equilibrium-potential} are classical; see
\cite[Theorem~4.1]{Fuglede1960Potentials} and
\cite[Chapter~II, \S~1, especially p.~137]{Landkof1972Foundations}.  We verify the strict inequality
because it is used below.

Fix $x_0\notin F$ and choose
$0<r<\dist(x_0,F)$.  The exterior-ball Riesz balayage identity provides a
probability measure $\varepsilon_r^{x_0}$ on
$\R^d\setminus\overline B_r(x_0)$, with strictly positive density, such
that
\begin{equation}
 |x_0-y|^{\beta-d}
 =\int_{\R^d\setminus\overline B_r(x_0)}
 |z-y|^{\beta-d}\,\ud\varepsilon_r^{x_0}(z),
 \qquad y\in F.
\end{equation}
See \cite[\S~4, (4.2)]{Kwasnicki2017FractionalLaplacian}
for the explicit strictly positive density and
\cite[\S~3, (3.2)]{DragnevOriveSaffWielonsky2021}
for the exterior-ball balayage identity.
Integrating against $\gamma_F$ gives
\begin{equation}
 V_F(x_0)
 =\int_{\R^d\setminus\overline B_r(x_0)}
 V_F(z)\,\ud\varepsilon_r^{x_0}(z).
\end{equation}
If $V_F(x_0)=1$, then $0\le V_F\le1$ and strict positivity of the
balayage density would imply that $V_F=1$ $\mathcal L^d$-a.e.\ on
the exterior of the ball.  This is impossible because $\gamma_F$ is
finite and compactly supported, while $\beta<d$ implies
$V_F(x)\to0$ as $|x|\to\infty$.  This proves
\eqref{critical:eq:compact-equilibrium-strict}.
\end{proof}

We now use the preceding lemma to treat the compact-support case and
consolidate the three rigidity criteria for the residual regime.

\begin{theorem}[Three-dimensional rigidity for $0<q<1$]
\label{critical:thm:residual-three-rigidity}
\label{critical:thm:support-confinement}
Let $d=3$ and $0<q<1$.  Assume $\mu,\rho\in\mathcal P_q(\R^3)$,
$\rho\ll\mathcal{L}^3$, and that $\rho$ is a Lagrangian critical point
of $E_\mu$.  Set
\begin{equation}
 \nu:=\rho-\mu.
\end{equation}
Then $\rho=\mu$ if at least one of the following conditions holds:
\begin{enumerate}[label=\textup{(\roman*)}]
\item $\rho,\mu\in\mathcal P_{q+1}(\R^3)$;
\item $\supp\nu^+$ is compact;
\item $\nu$ is invariant under rotations about some point.
\end{enumerate}
\end{theorem}

\begin{proof}
Alternative~\textup{(i)} is
\Cref{critical:cor:single-contact}(iii).  We prove
alternative~\textup{(ii)} next.  Set
\begin{equation}
 \alpha:=q+1\in(1,2),
 \qquad
 u:=K*\nu.
\end{equation}
By \Cref{critical:lem:jordan-domination}, $\nu^+$ and $\nu^-$ have equal
mass, $\nu^+\ll\mathcal L^3$, and criticality gives
$\nabla u=0$ $\nu^+$-a.e.  Taking $m=1$ in
\Cref{critical:prop:residual-contact} therefore gives the finite contact
\begin{equation}\label{critical:eq:residual-three-contact}
 \mathcal I_\alpha\nu^+=\mathcal I_\alpha\nu^-<\infty
 \qquad\nu^+\text{-a.e.}
\end{equation}
Applying \Cref{critical:lem:riesz-exhaustion} to $\nu^+$, then applying
the complete maximum principle to the resulting finite-energy exhaustion
and passing to the monotone limit, yields
\begin{equation}\label{critical:eq:residual-three-domination}
 \mathcal I_\alpha\nu^+\le\mathcal I_\alpha\nu^-
 \qquad\text{everywhere on }\R^3
\end{equation}
in the extended-valued sense, exactly as in the proof of
\Cref{critical:prop:riesz-contact}.

Assume now that $F:=\supp\nu^+$ is compact.  If
$\nu^+(\R^3)=\nu^-(\R^3)=0$, then $\nu=0$.  Otherwise, $F$ has positive
Lebesgue measure, and therefore positive $\alpha$-Riesz capacity, since a set of zero
$\alpha$-Riesz capacity is Lebesgue-null. 
Let $\gamma_F$ be the equilibrium measure from
\Cref{critical:lem:compact-riesz-equilibrium}.  Then
$\mathcal I_\alpha\gamma_F=1$ $\nu^+$-a.e.  Writing
$m:=\nu^+(\R^3)=\nu^-(\R^3)$, Fubini-Tonelli and
\eqref{critical:eq:residual-three-domination} give
\begin{equation}
 m
 =\int_{\R^3}\mathcal I_\alpha\gamma_F\,\ud\nu^+
 =\int_{\R^3}\mathcal I_\alpha\nu^+\,\ud\gamma_F
 \le\int_{\R^3}\mathcal I_\alpha\nu^-\,\ud\gamma_F
 =\int_{\R^3}\mathcal I_\alpha\gamma_F\,\ud\nu^-
 \le m.
\end{equation}
Consequently, $\mathcal I_\alpha\gamma_F=1$ $\nu^-$-a.e.  The strict exterior inequality
\eqref{critical:eq:compact-equilibrium-strict} implies that $\nu^-$ is
supported on $F$.  Thus $\nu^+$ and $\nu^-$ are both compactly supported
and in particular have finite $\alpha$-moments.  By
\eqref{critical:eq:residual-three-domination},
$\mathcal I_\alpha\nu\le0$ $\mathcal L^d$-a.e., so
\Cref{critical:prop:riesz-one-sign} gives $\nu=0$.  This proves
alternative~\textup{(ii)}.

It remains to prove alternative~\textup{(iii)}.  Without loss of generality, we may assume that
the center of
rotation is the origin.  The Jordan parts $\nu^\pm$, the
potential $u$, and the locally integrable Riesz potential
$V:=\mathcal I_\alpha(\nu^--\nu^+)$
are radial.  Moreover, \eqref{critical:eq:residual-three-domination}
gives a representative of $V$ which is nonnegative
$\mathcal L^d$-a.e.  Write
\begin{equation}
 u(x)=U(r),
 \qquad
 V(x)=\mathsf V(r),
 \qquad
 r:=|x|.
\end{equation}
The integer reduction and the definition of $V$ give
\begin{equation}
 \Delta u=b_{3,q,1}V
 \qquad\text{in }\mathcal D'(\R^3),
\end{equation}
with $b_{3,q,1}>0$.  Hence, distributionally on $(0,\infty)$,
\begin{equation}
 \bigl(r^2U'(r)\bigr)'=b_{3,q,1}r^2\mathsf V(r).
\end{equation}
Since $r^2\mathsf V(r)\in L^1(0,R)$ for every $R>0$, the
distribution $r^2U'$ has an absolutely continuous representative with a
finite trace $C$ at $0$.  Hence, for almost every $r>0$,
\begin{equation}\label{critical:eq:residual-three-radial-flux}
 r^2U'(r)
 =C+b_{3,q,1}\int_0^r t^2\mathsf V(t)\,\ud t.
\end{equation}
We claim the constant $C=0$.  Indeed, $V\in L^1_{\mathrm{loc}}$, so the integral
on the right tends to zero as $r\downarrow0$.  On the other hand,
\begin{equation}
 |u(x)-u(z)|
 \le|\nu|(\R^3)|x-z|^q,
\end{equation}
so $U$ is continuous at the origin.  If $C\ne0$, then for all sufficiently
small $r$ the right-hand side of
\eqref{critical:eq:residual-three-radial-flux} has a fixed sign and
absolute value at least $|C|/2$.  Integrating
$|U'(r)|\ge |C|/(2r^2)$ toward the origin contradicts this continuity.

If $V=0$ $\mathcal L^d$-a.e., Riesz inversion
\eqref{critical:eq:rigidity-riesz-inversion} gives $\nu=0$.
Otherwise, since $\mathsf V\ge0$ $\mathcal L^d$-a.e., there is an $R>0$ such that
\begin{equation}
 \int_0^R t^2\mathsf V(t)\,\ud t>0.
\end{equation}
The radial flux identity then gives $U'(r)>0$ for $\mathcal L^1$-a.e. $r>R$.
Because $\nu^+$ is radial and absolutely continuous, write
\begin{equation}
 \nu^+=H(|x|)\,\mathcal L^3.
\end{equation}
The radial Sobolev identity
$\nabla u(x)=U'(|x|)x/|x|$ holds for $\mathcal L^d$-a.e. $x\ne0$.
Criticality and polar Fubini therefore give
\begin{equation}
 4\pi\int_R^\infty
 \indic_{\{U'(r)\ne0\}}H(r)r^2\,\ud r=0.
\end{equation}
Since $U'(r)>0$ for $\mathcal L^1$-a.e. $r>R$, it follows that
$H=0$ $\mathcal L^1$-a.e.\ on
$(R,\infty)$.  Thus $\supp\nu^+\subseteq\overline B_R$, and
alternative~\textup{(ii)} completes the proof.
\end{proof}

\begin{remark}[Necessary features of a counterexample]
Any nontrivial counterexample to natural-moment rigidity in the residual
three-dimensional regime must have nonradial discrepancy and unbounded
support of its positive part $(\rho-\mu)^+$.
\end{remark}

\begin{corollary}[Compact-support rigidity at noninteger order]\label{critical:cor:nonlocal-rigidity}
Let $d\ge1$ and $0<q<2$ satisfy
\begin{equation}
 d+q>2,
 \qquad
 \frac{d+q}{2}\notin\mathbb N.
\end{equation}
Assume $\mu,\rho\in\mathcal P_q(\R^d)$, $\rho\ll\mathcal{L}^d$, and that
$\rho$ is a Lagrangian critical point of $E_\mu$.  If $\supp\rho$ is compact, then $\rho=\mu$.
\end{corollary}

\begin{proof}
All parameter regimes covered by the corollary fall under
\Cref{critical:thm:natural-moment-rigidity}, except
$d=3$ and $0<q<1$.  In the latter case,
the domination $(\rho-\mu)^+\le\rho$ from
\Cref{critical:lem:jordan-domination} gives
\begin{equation}
 \supp\bigl((\rho-\mu)^+\bigr)\subseteq\supp\rho.
\end{equation}
The compact-positive-part alternative in
\Cref{critical:thm:residual-three-rigidity} therefore gives $\rho=\mu$.
\end{proof}

\begin{remark}
In particular, in the range of
\Cref{critical:cor:nonlocal-rigidity}, if the target $\mu$ has unbounded
support, then it admits no compactly supported absolutely continuous
Lagrangian critical point.
\end{remark}

\subsection{Flexibility below the energy-distance endpoint}

The preceding results leave two regimes with $0<q<1$ outside unconditional natural-moment rigidity, but only one is known to be flexible.  For $d=3$ and $0<q<1$, any unresolved natural-moment case must have nonradial discrepancy and unbounded support of its positive part; rigidity holds under finite moments of order $q+1$, when the positive part is compactly supported, or when the discrepancy is radial.  In contrast, for $d=1$ and $0<q<1$, compactly supported non-minimizing critical points exist.  For the energy-kernel interaction profile
\begin{equation}
	K(z)=-|z|^q,
	\qquad 0<q<1,
\end{equation}
there are explicit absolutely continuous Lagrangian critical points which are not minimizers.  The construction is a nested-interval version of the classical equilibrium-measure mechanism.  We record the normalization and the short variational consequence of the explicit interval equilibrium formula that are used below.

\begin{lemma}[Constant-potential density on an interval]\label{critical:lem:interval-eqm}
Let $0<q<1$ and $R>0$.  Define
\begin{equation}\label{critical:eq:interval-density}
\eta_{q,R}(x)
:= a_q R^q (R^2-x^2)^{-\frac{1+q}{2}}\indic_{(-R,R)}(x),
\qquad
 a_q:=\frac{\Gamma(1-q/2)}{\sqrt{\pi}\,\Gamma((1-q)/2)}.
\end{equation}
Then $\eta_{q,R}\ud x$ is a probability measure, $\eta_{q,R}\in L^p(\R)$ for every $1\le p<2/(1+q)$, and the potential
\begin{equation}
U_R(x):=\int_{-R}^R |x-y|^q\eta_{q,R}(y)\ud y
\end{equation}
is constant for $x\in(-R,R)$.
\end{lemma}

\begin{proof}
For $R=1$, the asserted density is the $(-q)$-equilibrium measure of $[-1,1]$; see \cite[Proposition~4.6.1]{BorodachovHardinSaff2019} and \cite[Appendix~A]{ClarkLaugesen2025}.  The equilibrium Euler--Lagrange condition and continuity of the potential give its constancy on $(-1,1)$.  Scaling gives the result for general $R$, while the endpoint behavior yields the stated $L^p$ range.
\end{proof}

\begin{proposition}[Nonminimizing critical points on nested intervals]\label{critical:thm:nested-intervals}
Let $0<q<1$ and let $0<r<R$.  Let
\begin{equation}
\mu:=\eta_{q,R}\ud x,
\qquad
\rho_r:=\eta_{q,r}\ud x,
\end{equation}
where $\eta_{q,R}$ is defined by \eqref{critical:eq:interval-density}.  Then $\rho_r$ is a Lagrangian critical point of the MMD energy with target $\mu$, but $\rho_r$ is not a minimizer.
\end{proposition}

\begin{proof}
By \Cref{critical:lem:interval-eqm}, the potentials
\begin{equation}
U_R(x)=\int |x-y|^q\ud\mu(y),
\qquad
U_r(x)=\int |x-y|^q\ud\rho_r(y)
\end{equation}
are constant on $(-R,R)$ and $(-r,r)$, respectively.  Since $(-r,r)\subset(-R,R)$,
\begin{equation}
K*(\rho_r-\mu)=-U_r+U_R
\end{equation}
is constant on $(-r,r)$, and its weak derivative is zero there.  The boundary points have zero $\rho_r$-mass, so
\begin{equation}
\nabla K*(\rho_r-\mu)=0
\qquad \rho_r\text{-a.e.}
\end{equation}
Thus $\rho_r$ is a Lagrangian critical point.

The measures $\mu$ and $\rho_r$ are obviously distinct, compactly supported probability measures.  By the Fourier--Sobolev representation in \Cref{lem:genSob}, the MMD energy is strictly positive unless the two measures agree.  Hence $E_\mu(\rho_r)>E_\mu(\mu)=0$, so $\rho_r$ is not a minimizer.
\end{proof}

\begin{remark}
The preceding proposition is deliberately stated only in the one-dimensional energy-kernel MMD range $0<q<1$.  Classical Riesz equilibrium measures yield analogous constant-potential constructions in other super-Coulomb regimes; for background, see \cite[Chapter~4]{BorodachovHardinSaff2019}.  Those kernels, however, do not coincide with the positive-order energy-kernel MMD family studied in the rest of the paper.  The Coulomb endpoint $q=1$ is also excluded.  As $q\uparrow1$, the measures $\eta_{q,R}\ud x$ converge weakly to $\frac12(\delta_{-R}+\delta_R)$.  Thus the nested-interval construction leaves the absolutely continuous class at the endpoint and does not contradict the rigidity result above.
\end{remark}

\subsection{\texorpdfstring{Asymptotic criticality and conditional target convergence}{Asymptotic criticality and conditional target convergence}}

The following is a LaSalle-type consequence of the energy--dissipation identity.  For related abstract results, see \cite{CarrilloGvalaniWu2023}. 
For $1\le q<2$, the energy itself supplies the moment bound needed for asymptotic compactness.  When $0<q<1$, the uniform moment and $L^p$ hypotheses are used below to pass to the singular critical-point equation.

\begin{theorem}[Asymptotic criticality]\label{thm:conditional-critical-convergence}
Assume the hypotheses of \Cref{thm:wgfmmdwp}, and let $\rho_t$ be the corresponding solution.  Suppose either that $1\le q<2$, or that the additional uniform bound
\begin{equation}\label{eq:conditional-uniform-bounds}
\sup_{t\ge0}\left(\mathcal M_r(\rho_t)+\|\rho_t\|_{L^p}\right)<\infty
\end{equation}
holds.
Let $\mathcal C_\mu$ be the set of Lagrangian critical points of $E_\mu$ in the sense of \Cref{critical:def:Lcrit}, and define the narrow $\omega$-limit set
\begin{equation}\label{eq:conditional-omega-limit}
\omega(\rho_0):=\left\{\bar\rho\in\mathcal P(\R^d):
\rho_{t_n}\rightharpoonup\bar\rho\text{ narrowly for some }t_n\to\infty\right\}.
\end{equation}
Then the orbit $\{\rho_t:t\ge0\}$ is narrowly relatively compact, $\omega(\rho_0)$ is nonempty and narrowly compact, and
\begin{equation}\label{eq:conditional-omega-critical}
\omega(\rho_0)\subset\mathcal C_\mu.
\end{equation}
Whenever \eqref{eq:conditional-uniform-bounds} holds,
\begin{equation}\label{eq:conditional-limit-inheritance}
\omega(\rho_0)\subset\mathcal P_r(\R^d)\cap L^p(\R^d).
\end{equation}
In particular, if $d_{\mathrm{BL}}$ denotes the bounded--Lipschitz metric for narrow convergence, then
\begin{equation}\label{eq:conditional-distance-critical-set}
\lim_{t\to\infty}d_{\mathrm{BL}}(\rho_t,\mathcal C_\mu)=0,
\qquad
d_{\mathrm{BL}}(\rho,\mathcal C_\mu):=\inf_{\nu\in\mathcal C_\mu}d_{\mathrm{BL}}(\rho,\nu).
\end{equation}
\end{theorem}

\begin{proof}
Suppose first that $1\le q<2$, and choose
\begin{equation}\label{eq:conditional-sigma-choice}
q-1<\sigma<\frac q2,
\end{equation}
which is possible because $q<2$.  Since $q/2<1\le r$, \Cref{prop:mmdtomoment} applies to $\rho_t$ and $\mu$ at every time.  Together with \Cref{prop:energy-dissipation}, it gives
\begin{equation}\label{eq:conditional-energy-moment-bound}
\sup_{t\ge0}\mathcal M_\sigma(\rho_t)
\le \mathcal M_\sigma(\mu)
+C_{d,q,\sigma}\mmd_q(\rho_0,\mu)^{2\sigma/q}<\infty.
\end{equation}
Thus the orbit is tight.  If $0<q<1$, tightness instead follows from the moment bound in \eqref{eq:conditional-uniform-bounds}.  Hence in either case the orbit is narrowly relatively compact, and its $\omega$-limit set is nonempty and narrowly compact.
By \Cref{prop:energy-dissipation},
\begin{equation}\label{eq:conditional-integrable-dissipation}
\int_0^\infty\mathscr D_q(\rho_t\mid\mu)\ud t
\le\mmd_q^2(\rho_0,\mu)<\infty.
\end{equation}

Let $\bar\rho\in\omega(\rho_0)$ and choose $t_n\to\infty$ such that $\rho_{t_n}\rightharpoonup\bar\rho$ narrowly.  There are $s_n\in[t_n,t_n+1]$ such that
\begin{equation}\label{eq:conditional-low-dissipation-times}
\mathscr D_q(\rho_{s_n}\mid\mu)
\le\int_{t_n}^{t_n+1}\mathscr D_q(\rho_t\mid\mu)\ud t
\longrightarrow0.
\end{equation}
Coupling $\rho_{t_n}$ and $\rho_{s_n}$ by the characteristic flow and
applying Cauchy--Schwarz in space and time give
\begin{equation}\label{eq:conditional-nearby-times}
d_{\mathrm{BL}}(\rho_{t_n},\rho_{s_n})
\le\int_{t_n}^{s_n}\int_{\R^d}|v_t|\ud\rho_t\ud t
\le\left(\int_{t_n}^{t_n+1}\mathscr D_q(\rho_t\mid\mu)\ud t\right)^{1/2}
\longrightarrow0.
\end{equation}
Thus $\rho_{s_n}\rightharpoonup\bar\rho$ narrowly as well.  Whenever \eqref{eq:conditional-uniform-bounds} holds, the uniform $L^p$ bound, weak compactness when $p<\infty$, weak-star compactness when $p=\infty$, and lower semicontinuity of the $r$-th moment show that
\begin{equation}
\bar\rho\in\mathcal P_r(\R^d)\cap L^p(\R^d).
\end{equation}
This proves \eqref{eq:conditional-limit-inheritance}.  When $1\le q<2$, \eqref{eq:conditional-energy-moment-bound} and lower semicontinuity likewise give $\bar\rho\in\mathcal P_\sigma(\R^d)$.

Set
\begin{equation}
F_n:=\nabla K*(\rho_{s_n}-\mu),
\qquad
F:=\nabla K*(\bar\rho-\mu).
\end{equation}
We first consider $q\ne1$.  If $0<q<1$, then \eqref{eq:conditional-uniform-bounds} holds.  H\"older's inequality gives, for every $R>0$ and $0<\delta<1$,
\begin{equation}\label{eq:conditional-near-field}
\sup_n\sup_{x\in B_R}\int_{|x-y|<\delta}|\nabla K(x-y)|\bigl(\rho_{s_n}(y)+\bar\rho(y)\bigr)\ud y\le C\delta^{q-1+d/p_*}\left(\sup_n\|\rho_{s_n}\|_{L^p}+\|\bar\rho\|_{L^p}\right),
\end{equation}
and the exponent is positive under \eqref{eq:admissible-p-range}.  After inserting a smooth cutoff that removes the $\delta$-neighborhood of the diagonal, the force kernel is bounded and continuous and tends uniformly to zero in the far field.  Narrow convergence and a finite-net argument in $x$, followed by letting $\delta\downarrow0$, therefore yield local uniform convergence of the forces.

For a finite Borel measure $\eta$ with $\eta(\R^d)=0$, one has
\begin{equation}\label{eq:conditional-zero-mass-cancellation}
\nabla K*\eta(x)
=\int_{\R^d}\bigl(\nabla K(x-y)-\nabla K(x)\bigr)\ud\eta(y).
\end{equation}
For $\eta=\rho_{s_n}-\bar\rho$, the left-hand side is $F_n-F$.  If $1<q<2$, then for every $M\ge1$, \eqref{eq:kernel-gradient-global-holder} implies
\begin{align}
&\sup_n\sup_{x\in\R^d}
\int_{|y|>M}\left|\nabla K(x-y)-\nabla K(x)\right|
\ud(\rho_{s_n}+\bar\rho)(y)\notag\\
&\hspace{7em}\le
CM^{q-1-\sigma}
\left(\sup_n\mathcal M_\sigma(\rho_{s_n})
+\mathcal M_\sigma(\bar\rho)\right)
\longrightarrow0
\qquad\text{as }M\to\infty.
\label{eq:conditional-far-field}
\end{align}
After inserting a cutoff $\chi_M\in C_c(\R^d)$ equal to one on $B_M$, narrow convergence gives pointwise convergence of the cutoff integrals, while \eqref{eq:kernel-gradient-global-holder} makes them equicontinuous in $x$.  The preceding tail estimate then permits $M\to\infty$.  Thus, for either $0<q<1$ or $1<q<2$,
\begin{equation}\label{eq:conditional-force-convergence}
F_n\longrightarrow F
\qquad\text{locally uniformly on }\R^d.
\end{equation}
Hence, for every nonnegative $\chi\in C_c(\R^d)$,
\begin{equation}
\int_{\R^d}\chi|F|^2\ud\bar\rho
=\lim_{n\to\infty}\int_{\R^d}\chi|F_n|^2\ud\rho_{s_n}
\le\|\chi\|_{L^\infty}\lim_{n\to\infty}\mathscr D_q(\rho_{s_n}\mid\mu)=0.
\end{equation}
It follows that $F=0$ $\bar\rho$-a.e.
When $0<q<1$, \Cref{lem:singular-convolution-bound} gives $F\in L^\infty(\R^d)$.  When $1<q<2$, applying \eqref{eq:conditional-zero-mass-cancellation} with $\eta=\bar\rho-\mu$ and using \eqref{eq:kernel-gradient-global-holder} gives
\begin{equation}
\|F\|_{L^\infty}
\le C\left(\mathcal M_{q-1}(\bar\rho)+\mathcal M_{q-1}(\mu)\right)<\infty.
\end{equation}
In either case, the normalized potential \eqref{critical:eq:normalized-potential} is locally integrable and has weak gradient $F\in L^1_{\mathrm{loc}}(\bar\rho)$.  Thus $\bar\rho\in\mathcal C_\mu$.

It remains to treat $q=1$.  We use throughout the odd representative of $\nabla K$ fixed after \eqref{eq:wgfmmdparticle}.  For $\varphi\in C_c^1(\R^d;\R^d)$, set
\begin{equation}
H_\varphi(x,y):=
\begin{cases}
(\varphi(x)-\varphi(y))\cdot\nabla K(x-y),&x\ne y,\\
0,&x=y.
\end{cases}
\end{equation}
Since $|\nabla K(z)|=1$ for $z\ne0$ and $\varphi$ is Lipschitz, $H_\varphi$ is bounded and continuous on $(\R^d)^2$.  Pair symmetrization gives
\begin{equation}\label{eq:conditional-symmetrized-force}
\int_{\R^d}\varphi\cdot F_n\ud\rho_{s_n}
=\frac12\iint_{(\R^d)^2}H_\varphi(x,y)
\ud\rho_{s_n}(x)\ud\rho_{s_n}(y)
-\int_{\R^d}\varphi\cdot(\nabla K*\mu)\ud\rho_{s_n}.
\end{equation}
The target field is bounded and uniformly continuous: translation continuity in $L^1$ gives
\begin{equation}
\|(\nabla K*\mu)(\cdot+h)-\nabla K*\mu\|_{L^\infty}
\le\|\nabla K\|_{L^\infty}\|\mu(\cdot+h)-\mu\|_{L^1}
\longrightarrow0.
\end{equation}
Narrow convergence of $\rho_{s_n}$ and of the product measures therefore yields
\begin{equation}
\lim_{n\to\infty}\int_{\R^d}\varphi\cdot F_n\ud\rho_{s_n}
=\int_{\R^d}\varphi\cdot g_{\bar\rho}\ud\bar\rho.
\end{equation}
On the other hand,
\begin{equation}
\left|\int_{\R^d}\varphi\cdot F_n\ud\rho_{s_n}\right|
\le\|\varphi\|_{L^\infty}
\mathscr D_1(\rho_{s_n}\mid\mu)^{1/2}
\longrightarrow0.
\end{equation}
Since $\nabla K(0)=0$ and $|\nabla K(z)|\le1$ for every $z$, the integral
defining $g_{\bar\rho}$ is absolutely convergent for every $x$, with
$|g_{\bar\rho}(x)|\le2$.  Combining the preceding limits gives
\begin{equation}
\int_{\R^d}\varphi\cdot g_{\bar\rho}\ud\bar\rho=0,
\qquad \forall \varphi\in C_c^1(\R^d;\R^d).
\end{equation}
Hence the vector measure $g_{\bar\rho}\bar\rho$ vanishes, equivalently
$g_{\bar\rho}=0$ $\bar\rho$-a.e.  Therefore all the requirements of
\Cref{critical:def:Lcrit} hold, so $\bar\rho\in\mathcal C_\mu$.
This proves \eqref{eq:conditional-omega-critical}.

The distance conclusion \eqref{eq:conditional-distance-critical-set} follows by contradiction from relative compactness: any sequence of times along which the bounded--Lipschitz distance from $\mathcal C_\mu$ stays bounded below by a positive constant would have a subsequential limit in $\omega(\rho_0)\subset\mathcal C_\mu$.
\end{proof}

Without \eqref{eq:conditional-uniform-bounds}, the $\omega$-limit points produced for $1\le q<2$ need not be absolutely continuous or satisfy the moment hypotheses required by the rigidity results.  Under that bound, combining \Cref{thm:conditional-critical-convergence} with
\Cref{critical:thm:natural-moment-rigidity,critical:thm:residual-three-rigidity}
yields convergence to the target under the following hypotheses.

\begin{corollary}[Conditional convergence for energy kernels]\label{cor:conditional-energy-convergence}

Let $d\ge1$ and let $0<q<2$ satisfy $d+q-2\ge0$.  Let
\begin{equation}
r\ge\max\{1,q\},
\qquad
p_c\le p\le\infty,
\qquad
p>p_c\quad\text{if }d+q-2>0.
\end{equation}

If $d=3$ and $0<q<1$, assume in addition either that $r\ge q+1$ or
that $\rho_0$ and $\mu$ are invariant under rotations about a common
point.
Suppose that $\mu,\rho_0\in\mathcal P_r(\R^d)\cap L^p(\R^d)$ are
probability densities, and let $\rho_t$ be the solution of
\eqref{eq:wgfmmd}.  If
\eqref{eq:conditional-uniform-bounds} holds, then
\begin{equation}
\rho_t\rightharpoonup\mu\text{ narrowly},
\qquad
\mmd_q(\rho_t,\mu)\longrightarrow0.
\end{equation}
\end{corollary}

\begin{proof}
Every element $\bar\rho\in\omega(\rho_0)$ is an $L^p$ probability
density in $\mathcal P_r(\R^d)$ and a Lagrangian critical point by
\Cref{thm:conditional-critical-convergence}.  Suppose first that either
$d\ne3$ or $q\ge1$.  Since
$r\ge\max\{1,q\}$, both $\bar\rho$ and $\mu$ belong to
$\mathcal P_q(\R^d)$.  The admissible pairs $(d,q)$ lie in one of
the following regimes of \Cref{critical:thm:natural-moment-rigidity}:
$d$ is even; $d$ is odd and $q\ge1$; or $d\ge5$ is odd and $0<q<1$.  Hence
$\bar\rho=\mu$.

Suppose now that $d=3$ and $0<q<1$.  If $r\ge q+1$, then
$\bar\rho,\mu\in\mathcal P_{q+1}(\R^3)$, so
\Cref{critical:thm:residual-three-rigidity}(i) gives $\bar\rho=\mu$.
Otherwise, $\rho_0$ and $\mu$ are invariant under rotations about a
common point.  For any such rotation $T$, rotation invariance of $K$
implies that $T_\#\rho_t$ solves \eqref{eq:wgfmmd} with target
$T_\#\mu=\mu$ and initial datum $T_\#\rho_0=\rho_0$.  Uniqueness in
\Cref{thm:wgfmmdwp} therefore gives $T_\#\rho_t=\rho_t$ for every $t\ge0$.
Pushforward by $T$ is continuous for narrow convergence, so every
$\bar\rho\in\omega(\rho_0)$ is invariant under the same rotations.
Consequently, $\bar\rho-\mu$ is radial about the common center, and
\Cref{critical:thm:residual-three-rigidity}(iii) gives $\bar\rho=\mu$.
Thus $\omega(\rho_0)=\{\mu\}$.  Since the orbit is narrowly relatively
compact, this singleton $\omega$-limit set forces
$\rho_t\rightharpoonup\mu$ narrowly.  Indeed, otherwise there would be
$\varepsilon>0$ and $t_n\to\infty$ such that
$d_{\mathrm{BL}}(\rho_{t_n},\mu)\ge\varepsilon$; relative compactness
would give a subsequence converging narrowly to an element of
$\omega(\rho_0)$, necessarily $\mu$, which is a contradiction.  Finally,
since $q/2<1\le r$, narrow convergence and the uniform $r$-th moment
bound imply $\mathsf T_{q/2}(\rho_t,\mu)\to0$, and
\Cref{prop:mmdwass} gives convergence in $\mmd_q$.
\end{proof}

\begin{remark}
The only one-dimensional range not covered by the corollary is $0<q<1$, for which $d+q-2<0$.  In that range, \Cref{critical:thm:nested-intervals} exhibits distinct absolutely continuous Lagrangian critical points.  The endpoint $d=q=1$ is covered by the corollary.
\end{remark}

\section{PL inequalities}\label{sec:PL}
We record several static obstructions to global Polyak--\L ojasiewicz inequalities for the MMD energy.  The common mechanism is that the energy may be carried either far from the target or in high-frequency directions where the Wasserstein dissipation is too weak.  Throughout the section, constants depending only on fixed kernels, the dimension, the target, and the fixed bump functions may change from line to line.

For an interaction profile $K$, set, whenever the expressions are finite,
\begin{equation}\label{eq:PL-energy-diss-def}
\mathcal E_K(\sigma):=\frac12\iint K(x-y)\ud\sigma(x)\ud\sigma(y),
\qquad
\mathcal D_K(\rho\mid\mu):=\int_{\R^d}|\nabla K*(\rho-\mu)|^2\ud\rho,
\end{equation}
where $\sigma$ is a zero-mass signed measure.

For the whole-space statements we use the Riesz family
\begin{equation}\label{eq:PL-riesz-kernel}
K_q(x)=
\begin{cases}
-\dfrac1q |x|^q, & q\in(-d,2)\setminus\{0\},\\[0.5em]
-\log |x|, & q=0.
\end{cases}
\end{equation}
For $K=K_q$, we abbreviate
\begin{equation}
\mathcal E_q(\sigma):=\mathcal E_{K_q}(\sigma),
\qquad
\mathcal D_q(\rho\mid\mu):=\mathcal D_{K_q}(\rho\mid\mu).
\end{equation}
For $0<q<2$, this normalization gives $\mathcal D_q=q^{-2}\mathscr D_q$.

For $-d<q<0$ and $\rho\in C_c^1(\R^d)$, the source field $\nabla K_q*\rho$ is understood as the ordinary convolution when $q>1-d$ and as the Cauchy principal value when $q\le1-d$.  Indeed, $\nabla K_q(z)=-z|z|^{q-2}$, so the ordinary kernel is locally integrable exactly when $q>1-d$.  In the principal-value range, oddness gives
\begin{equation}\label{eq:PL-negative-pv-definition}
\begin{aligned}
\operatorname{p.v.}(\nabla K_q*\rho)(x)
&:=\lim_{\varepsilon\downarrow0}
\int_{|z|>\varepsilon}\nabla K_q(z)\rho(x-z)\ud z\\
&=\int_{|z|<1}\nabla K_q(z)\bigl(\rho(x-z)-\rho(x)\bigr)\ud z
+\int_{|z|\ge1}\nabla K_q(z)\rho(x-z)\ud z.
\end{aligned}
\end{equation}
The first integral is absolutely convergent because its integrand is $O(|z|^q)$ and $q>-d$, while the second is absolutely convergent by compact support.  When $q>1-d$, \eqref{eq:PL-negative-pv-definition} agrees with the ordinary convolution.

When $0<q<2$ and $\rho,\mu\in\mathcal P_q(\R^d)$, $\mathcal E_q(\rho-\mu)=q^{-1}\mmd_q^2(\rho,\mu)$ with the convention \eqref{def:mmdq}.  Thus replacing $\mathcal E_q$ by $\mmd_q^2$ only changes constants in this range.

\subsection{\texorpdfstring{Preliminary observations}{Preliminary observations}}\label{ssec:PL-equiv}

The only dynamical input used below is elementary. Suppose that a sufficiently regular Wasserstein gradient flow satisfies
\begin{equation}\label{eq:PL-energy-identity-abstract}
\frac{\ud}{\ud t}\mathcal E_K(\rho_t-\mu)=-\mathcal D_K(\rho_t\mid\mu).
\end{equation}
If
\begin{equation}\label{eq:PL-abstract}
\mathcal E_K(\rho-\mu)\le C_{\mathrm{PL}}\mathcal D_K(\rho\mid\mu)
\end{equation}
holds on a flow-invariant class, then
\begin{equation}\label{eq:PL-decay-abstract}
\mathcal E_K(\rho_t-\mu)\le e^{-t/C_{\mathrm{PL}}}\mathcal E_K(\rho_0-\mu)
\end{equation}
by Gr\"onwall's lemma. Conversely, if every regular initial datum in a class satisfies
\begin{equation}\label{eq:PL-rate-implies}
\mathcal E_K(\rho_t-\mu)\le \gamma(t)\mathcal E_K(\rho_0-\mu),
\qquad \gamma(0)=1,\quad\gamma'_+(0)<0,
\end{equation}
then differentiating at $t=0$ gives
\begin{equation}\label{eq:PL-converse-derivative}
0\le\gamma'_+(0)\mathcal E_K(\rho_0-\mu)+\mathcal D_K(\rho_0\mid\mu),
\end{equation}
so \eqref{eq:PL-abstract} holds with $C_{\mathrm{PL}}=-1/\gamma'_+(0)$.

We also record once the far-field MMD estimate used in both the static PL obstruction and the dynamical no-rate construction.

\begin{lemma}[Far-field MMD estimate]\label{lem:noconv-separated-mmd}
Let $0<q<2$ and let $\mu\in\mathcal P_{q/2}(\R^d)$ be a probability measure.  There are constants $R_0\ge1$ and $0<c<C<\infty$, depending only on $d,q,\mu$, such that if $R\ge R_0$, $|x_R|=R$, and $\rho$ is a probability measure with $\supp\rho\subset B_{R/4}(x_R)$, then
\begin{equation}\label{eq:noconv-separated-mmd}
cR^q\le \mmd_q^2(\rho,\mu)\le CR^q.
\end{equation}
\end{lemma}

\begin{proof}
Since $\mmd_q$ is a Hilbert-space distance, Minkowski's inequality for the Bochner mean embedding of $\rho$ gives
\begin{equation}
\mmd_q(\rho,\delta_{x_R})
\le \int_{\R^d}\mmd_q(\delta_x,\delta_{x_R})\ud\rho(x)
\le 2^{-q}R^{q/2},
\end{equation}
where we used $\mmd_q(\delta_x,\delta_y)=|x-y|^{q/2}$ and $\supp\rho\subset B_{R/4}(x_R)$.  The triangle inequality and $\mmd_q(\delta_{x_R},\delta_0)=R^{q/2}$ then give
\begin{equation}\label{eq:noconv-separated-mmd-bounds}
(1-2^{-q})R^{q/2}-\mmd_q(\mu,\delta_0)\le \mmd_q(\rho,\mu)\le (1+2^{-q})R^{q/2}+\mmd_q(\mu,\delta_0).
\end{equation}
Since $1-2^{-q}>0$ and $\mmd_q(\mu,\delta_0)<\infty$, these bounds are comparable to $R^{q/2}$ for all sufficiently large $R$.  Squaring proves \eqref{eq:noconv-separated-mmd}.
\end{proof}

\subsection{Whole-space Riesz kernels}\label{ssec:PL-whole-space}

The next proposition gives a static proof that no global PL inequality can hold in the whole space.  We separate the positive-order kernels, which include the active MMD range of the paper, from the negative-order Riesz kernels to avoid imposing an artificial common hypothesis at infinity.

\begin{proposition}[Whole-space failure of global PL]\label{prop:nogPL}
Let $q\in(-d,2)\setminus\{0\}$.
\begin{enumerate}[label=\textup{(\roman*)}]
\item If $0<q<2$ and $\mu\in\mathcal P(\R^d)$ is compactly supported, then there exists a sequence $\rho_n\in C_c^\infty(\R^d)$ of probability densities such that
\begin{equation}\label{eq:PL-whole-positive-ratio}
\frac{\mathcal D_q(\rho_n\mid\mu)}{\mathcal E_q(\rho_n-\mu)}\longrightarrow0.
\end{equation}
Equivalently, in the notation \eqref{def:mmdq},
\begin{equation}\label{eq:PL-whole-positive-ratio-mmd}
\frac{\mathcal D_q(\rho_n\mid\mu)}{\mmd_q^2(\rho_n,\mu)}\longrightarrow0.
\end{equation}
\item
If $-d<q<0$, assume that $\mu\in\mathcal P(\R^d)$ and
\begin{equation}\label{eq:PL-negative-target-energy}
0<\iint |x-y|^q\ud\mu(x)\ud\mu(y)<\infty.
\end{equation}
Assume moreover that there exist $R_0<\infty$ and a locally integrable Borel vector field
$F_\mu:\R^d\setminus\overline{B}_{R_0}\to\R^d$ such that $K_q*\mu$ is finite on this exterior region,
\begin{equation}\label{eq:PL-negative-target-field-distribution}
F_\mu=\nabla(K_q*\mu)
\qquad\text{in }\mathcal D'(\R^d\setminus\overline{B}_{R_0}),
\end{equation}
and
\begin{equation}\label{eq:PL-negative-target-assumptions}
\lim_{|x|\to\infty}\left(|K_q*\mu(x)|+|F_\mu(x)|\right)=0.
\end{equation}
Then there exists a sequence $\rho_n\in C_c^\infty(\R^d)$ of probability densities satisfying \eqref{eq:PL-whole-positive-ratio}.
\end{enumerate}
Consequently, under either set of hypotheses there is no finite constant $C_{\mathrm{PL}}(\mu)$ such that
\begin{equation}\label{eq:gPL}
\mathcal E_q(\rho-\mu)\le C_{\mathrm{PL}}(\mu)\mathcal D_q(\rho\mid\mu)
\end{equation}
for all admissible probability densities $\rho$.
\end{proposition}

\begin{proof}
We first treat $0<q<2$.  Fix a nonnegative $\theta\in C_c^\infty(B_1(0))$ with $\int_{\R^d}\theta(x)\ud x=1$, choose $e\in\mathbb S^{d-1}$, and set
\begin{equation}
\rho_R(x):=\theta(x-Re),\qquad R\ge1.
\end{equation}
Choose $R_\mu\ge1$ with $\supp\mu\subset B_{R_\mu}(0)$. For all sufficiently large $R$, $\supp\rho_R\subset B_{R/4}(Re)$, and \Cref{lem:noconv-separated-mmd} gives
\begin{equation}\label{eq:PL-positive-energy-lower}
\mathcal E_q(\rho_R-\mu)=q^{-1}\mmd_q^2(\rho_R,\mu)\ge cR^q.
\end{equation}

We next bound the dissipation.  For $x\in\supp\rho_R$, the self-field
\begin{equation}
\nabla K_q*\rho_R(x)=\int \nabla K_q(x-y)\theta(y-Re)\ud y
\end{equation}
is a translate of the fixed smooth function $\nabla K_q*\theta$ on $\supp\theta$; it is bounded because $|\nabla K_q(z)|\le C|z|^{q-1}$ is locally integrable for $q>0$.  The target field satisfies
\begin{equation}\label{eq:PL-positive-target-field}
|\nabla K_q*\mu(x)|\le
\begin{cases}
C, & 0<q\le1,\\
CR^{q-1}, & 1<q<2,
\end{cases}
\qquad x\in\supp\rho_R.
\end{equation}
Therefore
\begin{equation}\label{eq:PL-positive-diss-upper}
\mathcal D_q(\rho_R\mid\mu)\le C\left(1+R^{2(q-1)_+}\right),
\qquad (q-1)_+:=\max\{q-1,0\}.
\end{equation}
Combining \eqref{eq:PL-positive-energy-lower} and \eqref{eq:PL-positive-diss-upper},
\begin{equation}
\frac{\mathcal D_q(\rho_R\mid\mu)}{\mathcal E_q(\rho_R-\mu)}
\le C\left(R^{-q}+R^{q-2}\right)\longrightarrow0,
\end{equation}
because $0<q<2$.

It remains to treat $-d<q<0$.  With a slight abuse of notation, throughout this part of the proof we write $\nabla(K_q\ast\mu)$ for the locally integrable Borel representative specified in part~\textup{(ii)}.  In this regime, dilation makes the source self-energy and self-dissipation small, while translation makes its interaction with the fixed target negligible.  Let $\theta$ be as above and define
\begin{equation}
\rho_{a,r}(x):=r^{-d}\theta\left(\frac{x-a}{r}\right).
\end{equation}
Choose $r_n\to\infty$ and then choose $a_n\in\R^d$ so that $|a_n|\ge4r_n$ and
\begin{equation}\label{eq:PL-negative-field-small}
\supp\rho_{a_n,r_n}\subset\R^d\setminus\overline{B}_{R_0},
\qquad
\sup_{x\in\supp\rho_{a_n,r_n}}\left(|K_q*\mu(x)|+|\nabla(K_q\ast\mu)(x)|\right)\le r_n^{-1}.
\end{equation}
This is possible by the decay assumption in \eqref{eq:PL-negative-target-assumptions}.
By scaling,
\begin{equation}\label{eq:PL-negative-self-scaling}
\iint K_q(x-y)\ud\rho_{a_n,r_n}(x)\ud\rho_{a_n,r_n}(y)=O(r_n^q).
\end{equation}
The identity \eqref{eq:PL-negative-pv-definition} shows that $\nabla K_q*\theta$ is continuous, and hence bounded, on $\supp\theta$.  Homogeneity gives, for $u\in\supp\theta$,
\begin{equation}\label{eq:PL-negative-field-scaling}
(\nabla K_q*\rho_{a,r})(a+ru)
=r^{q-1}(\nabla K_q*\theta)(u).
\end{equation}
Consequently,
\begin{equation}\label{eq:PL-negative-diss-self-scaling}
\begin{aligned}
\int |\nabla K_q*\rho_{a_n,r_n}|^2\ud\rho_{a_n,r_n}
&=r_n^{2(q-1)}\int_{\R^d}|\nabla K_q*\theta|^2\theta\ud u=O(r_n^{2(q-1)}).
\end{aligned}
\end{equation}
Since $q<0$, both quantities tend to zero.  The source--target energy term tends to zero by the potential bound in \eqref{eq:PL-negative-field-small}.  Consequently,
\begin{equation}\label{eq:PL-negative-energy-limit}
\mathcal E_q(\rho_{a_n,r_n}-\mu)
\longrightarrow \frac12\iint K_q(x-y)\ud\mu(x)\ud\mu(y)>0,
\end{equation}
For the dissipation, \eqref{eq:PL-negative-field-small} gives
\begin{align}
\mathcal D_q(\rho_{a_n,r_n}\mid\mu)
&=\int\left|\nabla K_q*\rho_{a_n,r_n}-\nabla(K_q\ast\mu)\right|^2\ud\rho_{a_n,r_n}\notag\\
&\le2\int|\nabla K_q*\rho_{a_n,r_n}|^2\ud\rho_{a_n,r_n}
+2\sup_{\supp\rho_{a_n,r_n}}|\nabla(K_q\ast\mu)|^2
\longrightarrow0.\label{eq:PL-negative-diss-limit}
\end{align}
This proves the ratio convergence and hence the failure of \eqref{eq:gPL}.
\end{proof}

\begin{remark}\label{rem:PL-log-case}
For compactly supported bounded targets, the logarithmic endpoint $q=0$ in dimension two follows from the same translated-bump argument used in \cite[proof of Proposition~1.15]{ChodronDeCourcelRosenzweigCoulombDiscrepancies}: translating a fixed smooth bump to distance $R$ gives energy of order $\log R$ and bounded dissipation.  We omit the repetition.
\end{remark}

\subsection{Periodic sub-Coulomb kernels}\label{ssec:PL-periodic-subcoulomb}

On the torus, spatial escape is unavailable, so the translated-bump mechanism is replaced by a high-frequency oscillation whose dissipation is small relative to its energy.

Throughout this subsection, let $K_q$ denote the zero-average periodic Riesz kernel on $\mathbb T^d$ whose nonzero Fourier coefficients satisfy
\begin{equation}\label{eq:periodic-riesz-fourier}
\widehat{K_q}(k)=c_{d,q}|2\pi k|^{-d-q},
\qquad k\in\mathbb Z^d\setminus\{0\},
\end{equation}
with $c_{d,q}>0$.  This is the periodic analogue of the homogeneous kernel of order $q$.

\begin{proposition}[Periodic sub-Coulomb failure]\label{prop:noPLsubCo}
Assume $2-d<q<2$ and let $\mu\in L^\infty(\mathbb T^d)$ be a probability density satisfying $\operatorname*{ess\,inf}_{\mathbb T^d}\mu>0$.  Then there exists a sequence of probability densities $\rho_n$ on $\mathbb T^d$, with the same regularity as $\mu$ and in particular smooth if $\mu$ is smooth, such that
\begin{equation}\label{eq:PL-torus-subcoulomb-ratio}
\frac{\int_{\mathbb T^d}|\nabla K_q*(\rho_n-\mu)|^2\rho_n\ud x}{\frac12\iint_{(\mathbb T^d)^2}K_q(x-y)(\rho_n-\mu)(x)(\rho_n-\mu)(y)\ud x\ud y}
\longrightarrow0.
\end{equation}
In particular, no global PL inequality holds for this periodic kernel.
\end{proposition}

\begin{proof}
Choose $0<\varepsilon<\operatorname*{ess\,inf}\mu$ and, for $k\in\mathbb Z^d\setminus\{0\}$, set
\begin{equation}
\rho_k(x):=\mu(x)+\varepsilon\cos(2\pi k\cdot x).
\end{equation}
Then $\rho_k$ is a probability density and is nonnegative.  Let $\sigma_k:=\rho_k-\mu$.  Since $\widehat{\sigma_k}(\pm k)=\varepsilon/2$ and all other Fourier coefficients vanish, Plancherel gives
\begin{equation}\label{eq:PL-torus-energy}
\mathcal E_q(\sigma_k)
=\frac12\iint K_q(x-y)\sigma_k(x)\sigma_k(y)\ud x\ud y
=\frac{\varepsilon^2}{4}\widehat{K_q}(k).
\end{equation}
Moreover $\nabla K_q*\sigma_k$ has only the frequencies $\pm k$.  Hence $|\nabla K_q*\sigma_k|^2$ has only the frequencies $0$ and $\pm2k$, and
\begin{equation}\label{eq:PL-torus-diss}
\int_{\mathbb T^d}|\nabla K_q*\sigma_k|^2\rho_k\ud x
\le C\varepsilon^2 |2\pi k|^2\widehat{K_q}(k)^2.
\end{equation}
Indeed, the term containing the additional factor $\varepsilon\cos(2\pi k\cdot x)$ integrates to zero by frequency orthogonality, and the term weighted by $\mu$ is bounded using $\|\mu\|_{L^1}=1$ and $|\widehat\mu(\pm2k)|\le1$.  Dividing \eqref{eq:PL-torus-diss} by \eqref{eq:PL-torus-energy} and using \eqref{eq:periodic-riesz-fourier},
\begin{equation}
\frac{\int_{\mathbb T^d}|\nabla K_q*\sigma_k|^2\rho_k\ud x}{\mathcal E_q(\sigma_k)}
\le C |k|^2\widehat{K_q}(k)
\le C |k|^{2-d-q}\longrightarrow0,
\end{equation}
because $q>2-d$.
\end{proof}

\begin{remark}[Periodic Coulomb endpoint]\label{rem:PL-periodic-coulomb}\label{ssec:PL-periodic-coulomb}\label{prop:PL-coulomb-1d}\label{rem:PL-coulomb-lipschitz-note}\label{rem:PL-coulomb-lift}
The high-frequency construction in \Cref{prop:noPLsubCo} does not cover the Coulomb endpoint. In one dimension, periodic integration by parts gives an exact weighted electric-field identity, so every uniformly positive target satisfies a PL inequality, whereas the inequality may fail when the target vanishes at one point. In higher dimensions, targets that are sufficiently small perturbations of the uniform density satisfy a PL inequality; however, even a prescribed positive lower bound does not yield a PL constant that is uniform over the full target class. We refer to \cite{ChodronDeCourcelRosenzweigCoulombDiscrepancies} for these results and the associated convergence theory.
\end{remark}

\section{Long-time convergence}\label{sec:LT}
Long-time convergence asks whether the flow approaches the target rather than a spurious stationary state, and at what rate.  The asymptotic-criticality result \Cref{thm:conditional-critical-convergence} is unconditional for $1\le q<2$, while its $0<q<1$ branch and the subsequent target-identification result use uniform moment and $L^p$ bounds.  Here we prove the complementary negative statement \Cref{thm:noconvmmd}: on $\R^d$, no MMD decay modulus can be uniform over the full admissible class of initial data.  The argument is the dynamical counterpart of the far-field PL obstruction: a translated source retains discrepancy of order $R^q$, while a waiting-time estimate keeps it in a far-field neighborhood of its initial location for a time that diverges with $R$.

For the proof, fix a nonnegative probability density $\theta\in C_c^\infty(B_1(0))$. For $R\ge 8$ choose a point $x_R\in\R^d$ with $|x_R|=R$ and set
	\begin{equation}\label{eq:noconv-translated-initial-datum}
	\rho_{0,R}(x)\coloneqq \theta(x-x_R).
	\end{equation}
	Then $\rho_{0,R}$ is a smooth compactly supported probability density, its $L^p$ norm is independent of $R$, and it has finite moments of every order.  Let $\rho_{t,R}$ be the solution of \eqref{eq:wgfmmd} with initial datum $\rho_{0,R}$ and velocity $v_{t,R}=-\nabla K\ast(\rho_{t,R}-\mu)$.  Let $X_R(t,s,a)$ denote the characteristic flow generated by $v_{t,R}$.

By \Cref{lem:noconv-separated-mmd}, whenever $\supp\rho_{t,R}\subset B_{R/4}(x_R)$ and $R$ is sufficiently large,
\begin{equation}\label{eq:noconv-separated-mmd-dynamic}
cR^q\le\mmd_q^2(\rho_{t,R},\mu)\le CR^q.
\end{equation}

Next we prove that the translated bump cannot leave $B_{R/4}(x_R)$ before a time which diverges with $R$.
	\begin{lemma}[Support waiting time]\label{lem:noconv-waiting-time}
		There are $R_0\ge8$ and times $T_R\to\infty$ as $R\to\infty$, depending on $\theta,\mu,d,q,p$ but not otherwise on $R$, such that for every $R\ge R_0$,
		\begin{equation}\label{eq:noconv-support-contained}
		\supp\rho_{t,R}\subset B_{R/4}(x_R)
		\qquad\text{for all }0\le t\le T_R.
		\end{equation}
	\end{lemma}
	\begin{proof}
Let
\begin{equation}\label{eq:noconv-first-exit}
\tau_R:=\inf\{t\ge0:\supp\rho_{t,R}\not\subset B_{R/4}(x_R)\}.
\end{equation}
For $t<\tau_R$ and $x\in B_{R/4}(x_R)$, the velocity satisfies one of two bounds.  If $0<q\le1$, translation invariance of the initial $L^p$ norm and the estimate from \Cref{thm:wgfmmdwp} give $\|\rho_{t,R}\|_{L^p}\le Be^{At}$, uniformly in $R$.  Applying \Cref{lem:singular-convolution-bound} with $a=1-q$ (and the mass bound when $q=1$) yields
\begin{equation}\label{eq:noconv-velocity-small-q}
\|v_{t,R}\|_{L^\infty}
\le C\left(1+\|\rho_{t,R}\|_{L^p}^{(1-q)p_*/d}\right)
\le Ce^{At},
\end{equation}
where $A$ and $C$ have been enlarged if necessary.
If $1<q<2$, the support assumption gives $|x-y|\le R/2$ for $y\in\supp\rho_{t,R}$.  Since $q-1<1\le r$, one also has $\mathcal M_{q-1}(\mu)<\infty$ and, for $x\in B_{R/4}(x_R)$,
\begin{equation}
\int_{\R^d}|x-z|^{q-1}\ud\mu(z)
\le |x|^{q-1}+\mathcal M_{q-1}(\mu)
\le C_\mu R^{q-1}.
\end{equation}
Consequently,
\begin{equation}\label{eq:noconv-velocity-large-q}
|v_{t,R}(x)|\le CR^{q-1}.
\end{equation}
Choose $T_R$ by
\begin{equation}\label{eq:noconv-waiting-time-choice}
\int_0^{T_R}Ce^{As}\ud s=\frac R{16}\quad(0<q\le1),
\qquad
T_R=\frac{R}{16CR^{q-1}}\quad(1<q<2),
\end{equation}
with $CT_R=R/16$ if $A=0$.  In both cases $T_R\to\infty$.

The same first-exit argument now applies in both regimes.  For $a\in\supp\rho_{0,R}\subset B_1(x_R)$, a characteristic remaining in $B_{R/4}(x_R)$ satisfies, for $t\le T_R$,
\begin{equation}\label{eq:noconv-characteristic-displacement}
|X_R(t,0,a)-x_R|
\le1+\int_0^t\|v_{s,R}\|_{L^\infty(B_{R/4}(x_R))}\ud s
\le1+\frac R{16}<\frac R4.
\end{equation}
Thus no characteristic can reach the boundary at a time $\tau_R\le T_R$, and \eqref{eq:noconv-support-contained} follows.
\end{proof}

\begin{proof}[Proof of \Cref{thm:noconvmmd}]
	With the notation fixed above, suppose, toward a contradiction, that a function $\gamma$ satisfying \eqref{eq:impossibledecay} exists.  Increase $R_0$ so that \Cref{lem:noconv-separated-mmd,lem:noconv-waiting-time} apply for every $R\ge R_0$.  For such $R$, \Cref{lem:noconv-waiting-time} and \Cref{lem:noconv-separated-mmd} imply
	\begin{equation}
	cR^q\le \mmd_q^2(\rho_{T_R,R},\mu)\le \gamma(T_R)\mmd_q^2(\rho_{0,R},\mu)\le C R^q\gamma(T_R).
	\end{equation}
	Thus $\gamma(T_R)\ge c/C$ for all sufficiently large $R$.  Since $T_R\to\infty$, this contradicts $\lim_{t\to\infty}\gamma(t)=0$.
\end{proof}

\begin{remark}
	\Cref{thm:noconvmmd} does not rule out qualitative convergence, nor does it rule out local convergence estimates whose constants or waiting times depend on the initial datum.  It only excludes a global multiplicative rate in MMD with no dependence on the initial support radius or moment size. For compactly supported targets in the overlapping case $d=q=1$, a compatible Coulomb persistence result is proved in \cite[Theorem~1.12 and Corollary~1.13]{ChodronDeCourcelRosenzweigCoulombDiscrepancies}.
\end{remark}

\begin{remark}[Radial Coulomb PL inequality]\label{ssec:LT-radial}
Coulomb-specific radial reductions lead to substantially stronger conclusions than are presently available for the general energy kernels considered here. If the source and target are radial bounded densities, the target has connected support and is bounded above and below there, and the source support is contained in the target support, Chodron de Courcel and the first author prove a PL inequality and exponential convergence \cite[Theorem~1.8 and Corollary~1.11]{ChodronDeCourcelRosenzweigCoulombDiscrepancies}. The structural assumptions are shown to be sharp within the radial class; see \cite[Remarks~1.9--1.10]{ChodronDeCourcelRosenzweigCoulombDiscrepancies}.
\end{remark}

\section{Two-time-scale phenomenon: a one-dimensional example}\label{sec:1dexample}
The obstruction in \Cref{thm:noconvmmd} concerns uniformity over the initial support scale and does not rule out rapid relaxation after a datum-dependent waiting time.  In one dimension, the Coulomb/energy-distance particle flow makes this two-stage behavior explicit: for a uniformly positive target, particles outside its support first travel to it and then relax exponentially; if the target density vanishes, the post-entry relaxation can instead be polynomial.

\subsection{Uniformly positive target}\label{ssec:1dexampleunifpos}
Assume that $\mu$ is supported on $[0,1]$ and has density $\mu\ge\lambda>0$ there.  Let
\begin{equation}\label{eq:one-dimensional-cdf}
F(x):=\mu((-\infty,x]).
\end{equation}
For the diagonal-free Coulomb particle system
\begin{equation}\label{eq:Npart}
\dot x_i^t=(K'*\mu)(x_i^t)-\frac1N\sum_{j\in[N]\setminus\{i\}}K'(x_i^t-x_j^t),
\qquad i\in[N],
\qquad K(x)=-|x|,
\end{equation}
the self-interaction is omitted, in accordance with the convention in \eqref{eq:wgfmmdparticle}.  Start with $(x_i^0)_{i\in[N]}$ satisfying $x_1^0<\cdots<x_N^0$.  Since $-K'*\mu=2F-1$ and, as long as the ordering holds,
\begin{equation}\label{eq:one-dimensional-diagonal-free-count}
-\frac1N\sum_{j\in[N]\setminus\{i\}}K'(x_i^t-x_j^t)=\frac{2i-1-N}{N},
\end{equation}
one obtains the uncoupled equations
\begin{equation}\label{eq:dotxi}
\dot x_i^t=\frac{2i-1}{N}-2F(x_i^t).
\end{equation}
At a first contact of adjacent particles, the derivative of their gap would equal $2/N>0$; hence the ordering is preserved.

For each $i\in[N]$, there is a unique equilibrium $x_i^\infty\in(0,1)$ determined by
\begin{equation}\label{eq:xiinfty-general}
F(x_i^\infty)=\frac{2i-1}{2N}.
\end{equation}
A particle outside $[0,1]$ moves with constant velocity until its entrance time
\begin{equation}\label{eq:taui-general}
\tau_i:=
\begin{cases}
-\dfrac{Nx_i^0}{2i-1},&x_i^0<0,\\[0.4em]
0,&x_i^0\in[0,1],\\[0.4em]
\dfrac{N(x_i^0-1)}{2N-2i+1},&x_i^0>1.
\end{cases}
\end{equation}
Indeed, before entrance,
\begin{equation}\label{eq:one-dimensional-pre-entrance-trajectories}
x_i^t=
\begin{cases}
x_i^0+\dfrac{2i-1}{N}t,&x_i^0<0,\\[0.4em]
x_i^0-\dfrac{2N-2i+1}{N}t,&x_i^0>1.
\end{cases}
\end{equation}
After entrance, \eqref{eq:dotxi} and \eqref{eq:xiinfty-general} give
\begin{equation}\label{eq:one-dimensional-equilibrium-sign}
\dot x_i^t=2\bigl(F(x_i^\infty)-F(x_i^t)\bigr),
\end{equation}
so the velocity has the sign opposite to that of $x_i^t-x_i^\infty$.  Thus $x_i^t$ moves monotonically toward $x_i^\infty$.  If $I_i^t$ is the interval between these points, then
\begin{equation}\label{eq:one-dimensional-contraction}
\frac{\ud}{\ud t}|x_i^t-x_i^\infty|=-2|F(x_i^t)-F(x_i^\infty)|=-2\mu(I_i^t)
\le-2\lambda|x_i^t-x_i^\infty|,
\end{equation}
and therefore
\begin{equation}\label{eq:expconv-general}
|x_i^t-x_i^\infty|
\le|x_i^{\tau_i}-x_i^\infty|e^{-2\lambda(t-\tau_i)},
\qquad t\ge\tau_i.
\end{equation}
Thus the entrance time $\tau_{\mathrm{ent}}:=\max_{i\in[N]}\tau_i$ separates a transport phase from exponential relaxation at rate at least $2\lambda$.  For the uniform target, $x_i^\infty=(2i-1)/(2N)$ for $i\in[N]$, and
\begin{equation}\label{eq:minmmdunif}
\mmd_1^2\left(\frac1N\sum_{i=1}^N\delta_{\frac{2i-1}{2N}},\indic_{[0,1]}\right)=\frac1{12N^2}.
\end{equation}

\begin{remark}[Support normalization and noncompact targets]
The choice of $[0,1]$ is only a normalization.  If $\mu$ is supported on an interval $[a,b]$ and its density satisfies $\mu\ge\lambda>0$ there, the preceding argument applies with $0$ and $1$ replaced by $a$ and $b$.  Compact support is not needed for the cumulative distribution function (CDF) reduction \eqref{eq:dotxi}.  Subject to the standing moment assumptions, if $F$ is continuous and strictly increasing, each particle moves monotonically toward its quantile equilibrium.  For a noncompact target, however, a probability density cannot be bounded below uniformly on $\R$.  Quantitative estimates therefore combine tail control of the time needed to reach a neighborhood of equilibrium with local lower bounds for the density there; the local behavior of the density near the equilibrium determines the subsequent relaxation rate.
\end{remark}

\subsection{A degenerate target}\label{ssec:1dexampledegenerate}
The positivity assumption is essential for exponential relaxation.  Let
\begin{equation}\label{eq:degenerate-target}
\mu(x):=12\left(x-\frac12\right)^2\indic_{[0,1]}(x),
\end{equation}
so that
\begin{equation}\label{eq:Fdegenerate}
F(x)=\frac12+4\left(x-\frac12\right)^3,
\qquad x\in[0,1].
\end{equation}
Take $N\ge3$ odd, $m=(N+1)/2$, place every particle except the middle one at its equilibrium, and set $x_m^0=1/2+\varepsilon$ with $\varepsilon>0$ small enough to preserve the ordering.  Then $x_i^t\equiv x_i^\infty$ for $i\in[N]\setminus\{m\}$, while $z(t):=x_m^t-1/2$ solves
\begin{equation}\label{eq:degenerate-middle-particle}
\dot z=-8z^3,
\qquad z(0)=\varepsilon,
\end{equation}
and hence
\begin{equation}\label{eq:polyz}
z(t)=(\varepsilon^{-2}+16t)^{-1/2}\sim(16t)^{-1/2}.
\end{equation}
On the ordered chamber, the particle energy $\mathcal E_N$ defined in \eqref{eq:particle-energy-definition} satisfies $\partial_{x_i}\mathcal E_N=-\dot x_i/N$ for $i\in[N]$.  Integrating along the middle coordinate therefore gives
\begin{equation}\label{eq:polyenergy}
\mathcal E_N(X_N^t)-\mathcal E_N(X_N^\infty)
=\frac2N z(t)^4
=\frac2N(\varepsilon^{-2}+16t)^{-2}
\sim\frac1{128N}t^{-2}.
\end{equation}
Thus even particles initially inside the target support need not relax exponentially when the target density vanishes.

\begin{remark}[Continuum quantile dynamics]
Let $\rho_t$ be a solution of the one-dimensional continuum Coulomb flow, and let $Q_t,Q_\mu:(0,1)\to\R$ denote the quantile functions of $\rho_t$ and $\mu$, respectively.  The characteristic formulation gives, for almost every $\alpha\in(0,1)$,
\begin{equation}\label{eq:one-dimensional-continuum-quantile}
\partial_tQ_t(\alpha)
=2\bigl(\alpha-F(Q_t(\alpha))\bigr).
\end{equation}
This is the quantile equation studied in \cite{duong_steidl26}.  In particular, the particle equations \eqref{eq:dotxi} are obtained by sampling \eqref{eq:one-dimensional-continuum-quantile} at the midpoint ranks
\begin{equation}\label{eq:one-dimensional-continuum-midpoint-ranks}
\alpha_i=\frac{2i-1}{2N},
\qquad i\in[N].
\end{equation}

For the uniformly positive target above, every fixed quantile initially outside $[0,1]$ moves at constant velocity until it enters the target support and thereafter satisfies
\begin{equation}\label{eq:one-dimensional-continuum-contraction}
|Q_t(\alpha)-Q_\mu(\alpha)|
\le
|Q_{\tau(\alpha)}(\alpha)-Q_\mu(\alpha)|
e^{-2\lambda(t-\tau(\alpha))},
\qquad t\ge\tau(\alpha),
\end{equation}
where $\tau(\alpha)$ is its entrance time.  Unlike in the finite-particle system, however, these entrance times are generally not uniformly bounded in $\alpha$: if $\rho_0$ assigns positive mass to $\R\setminus[0,1]$, then the velocities of the extreme ranks approach zero and
\begin{equation}\label{eq:one-dimensional-continuum-unbounded-entrance}
\operatorname*{ess\,sup}_{\alpha\in(0,1)}\tau(\alpha)=\infty.
\end{equation}
Thus the continuum flow has the same two-stage behavior quantile by quantile, but generally has no finite time by which its entire support has entered $[0,1]$.  If instead $\supp\rho_0\subset[0,1]$, then no waiting time is needed.  Writing $G_t$ for the CDF of $\rho_t$ and $h_t:=G_t-F$, one has $h_t'=\rho_t-\mu$ almost everywhere and $h_t(\pm\infty)=0$.  Hence 
\begin{equation}\label{eq:one-dimensional-continuum-weight-identity}
\int_{\R}h_t^2(\rho_t-\mu)\ud x
=\int_{\R}h_t^2h_t'\ud x
=\frac13\bigl[h_t^3\bigr]_{-\infty}^{\infty}
=0.
\end{equation}
Thus the usual dissipation weighted by $\rho_t$ may equivalently be weighted by $\mu$, and the energy--dissipation identity reads
\begin{equation}\label{eq:one-dimensional-continuum-cdf-energy}
\begin{aligned}
\mmd_1^2(\rho_t,\mu)&=\int_{\R}|G_t-F|^2\ud x,\\
\frac{\ud}{\ud t}\mmd_1^2(\rho_t,\mu)
&=-4\int_0^1|G_t-F|^2\mu\,\ud x,
\end{aligned}
\end{equation}
gives
\begin{equation}\label{eq:one-dimensional-continuum-exponential-mmd}
\mmd_1^2(\rho_t,\mu)
\le e^{-4\lambda t}\mmd_1^2(\rho_0,\mu).
\end{equation}

The degeneracy mechanism also persists at the continuum level, although the middle-particle example does not transfer verbatim.  Indeed, for the target \eqref{eq:degenerate-target}, set
\begin{equation}\label{eq:one-dimensional-continuum-degenerate-errors}
a_\alpha:=Q_\mu(\alpha)-\frac12,
\qquad
e_\alpha(t):=Q_t(\alpha)-Q_\mu(\alpha).
\end{equation}
Since $\alpha=\frac12+4a_\alpha^3$, equation \eqref{eq:one-dimensional-continuum-quantile} yields, near the target,
\begin{equation}\label{eq:one-dimensional-continuum-degenerate-error-ode}
\dot e_\alpha
=-8e_\alpha\bigl(3a_\alpha^2+3a_\alpha e_\alpha+e_\alpha^2\bigr).
\end{equation}
The linearized relaxation rate is therefore $24a_\alpha^2$, which vanishes as $\alpha\to\frac12$; at the median label one recovers $\dot e=-8e^3$.  This explains the loss of a relaxation rate uniform over the quantiles.  Since changing a single quantile does not change the underlying measure, obtaining a continuum decay law requires perturbing a positive-measure family of nearby quantiles, which we do not pursue here.
\end{remark}

\section{Further questions and directions}\label{sec:further-questions}
The preceding sections establish global continuum and particle dynamics, a fixed-time mean-field limit, and several obstructions to uniform convergence.  The questions below concern what remains beyond this framework: long-time selection and rates, well-posedness at the finite critical exponent, the remaining rigidity and critical-point questions, and uniform-in-time mean-field approximation.

For $1\le q<2$, the energy--dissipation identity and \Cref{thm:conditional-critical-convergence} already imply that $d_{\mathrm{BL}}(\rho_t,\mathcal C_\mu)\to0$ as $t\to\infty$, but without the uniform bounds the $\omega$-limit points are not known to be absolutely continuous.  A first target-selection question is therefore when the moment and subcritical $L^p$ bounds propagate uniformly for all $t\ge0$; when $0<q<1$, the same bounds are also used here to pass to the critical-point equation.  Whenever the available uniform moment order meets the hypotheses of the rigidity theorem, every $\omega$-limit point equals the target; in the residual three-dimensional regime, the same conclusion holds at the natural moment level for radial source and target data.  Within this uniformly controlled class, the remaining target-identification problem concerns $d=3$ and $0<q<1$ with moment order below $q+1$ and without radial symmetry or a compact-positive-part condition on the limiting critical points.  Outside the present well-posedness range, when $d=1$ and $0<q<1$, non-minimizing critical points exist; one may instead ask which of them are dynamically accessible and how their basins of attraction are organized.

The failures of a global data-independent rate do not rule out local or data-dependent convergence estimates.  More precisely, one may seek a flow-invariant class, described by some combination of support, tail, moment, and two-sided density controls, on which a \L ojasiewicz-type dissipation inequality
\begin{equation}
\mathcal D_K(\rho\mid\mu)
\ge c\,\mathcal E_K(\rho-\mu)^{1+\alpha}
\end{equation}
holds.  Here $\alpha=0$ would yield exponential decay, while $\alpha>0$ would yield polynomial decay whenever the energy identity is available.  Small MMD alone need not preserve the geometric or density information required for such an estimate, so a central issue is to find natural conditions that propagate along the flow.  The compact-geometry results of \cite{chizat2026quantitativeconvergencewassersteingradient} provide one model for this local theory.  The radial Coulomb theory of \cite{ChodronDeCourcelRosenzweigCoulombDiscrepancies} further suggests asking how the order and geometry of target degeneracy determine the rate and whether source-support-inclusion estimates survive without radial symmetry.

A separate analytic problem occurs at the critical endpoint $p=p_c$ of the Euclidean well-posedness theory.
The critical Lorentz refinement $L^{p_c,1}$ should follow from adapting the well-posedness argument of \cite{chizat2026quantitativeconvergencewassersteingradient}, whose mechanism is not essentially periodic.  The sharper question is whether the flow is well posed when the source and target belong only to the critical Lebesgue space $L^p$.  At this endpoint the natural velocity control is Osgood rather than Lipschitz, while the target field $\nabla K*\mu$ can have unbounded negative divergence, so the present $L^p$ estimate does not close.  Does there exist a unique global solution in an appropriate critical solution class, or are stronger assumptions on the target necessary?

For the Euclidean energy kernel $K(z)=-|z|^q$, the principal remaining rigidity question is whether the natural $q$-moment assumption suffices when $d=3$ and $0<q<1$, or whether this regime admits an absolutely continuous non-minimizing Lagrangian critical point.  Any such counterexample must have nonradial discrepancy and unbounded support of its positive part.  More generally, one would like to determine under which regularity and support assumptions Lagrangian criticality in the sense of \Cref{critical:def:Lcrit} agrees with stationarity or with the stronger Wasserstein and JKO notions discussed above.  Beyond rigidity, one would like stability criteria and a description of the corresponding basins of attraction.  The examples in \Cref{prop:particle-saddle-configurations} give exact collision-free particle equilibria with either a singular line-supported limit or the full-dimensional uniform-ball limit.  More generally, which continuum Lagrangian critical points arise as limits of exact critical configurations of $\mathcal E_N$?  In particular, does every sufficiently regular continuum critical point admit such a recovery sequence, and what stability or nondegeneracy assumptions are necessary?

Finally, beyond the well-prepared regime of \Cref{cor:well-prepared-particle-uniform-time}, suppose that $\mmd_q(\rho_0^N,\rho_0)\to0$.  The estimate \eqref{eq:main} propagates the mean-field approximation on intervals $[0,T_N]$ with $T_N\to\infty$ chosen sufficiently slowly.  Particle energy dissipation then selects times $t_N\in[T_N/2,T_N]$ at which the mean-square particle velocity tends to zero.  This growing-window argument applies throughout the present mean-field range; the range $0<q<1$ enters only in the closure step.  For $1\le q<2$, \Cref{prop:particle-critical-closure} requires no additional microscopic nonconcentration hypothesis, whereas for $0<q<1$ it additionally requires a uniform discrete Riesz bound with inverse-power exponent $1-q$ at those same selected times.  Under the corresponding closure hypotheses, this gives particle and continuum subsequences with the same Lagrangian-critical limit.  Can the required nonconcentration bound when $0<q<1$ be propagated from natural initial data, or can the mean-field comparison be made uniform in time on a stable class?

Any such comparison must account for the positive optimal $N$-point quantization error of the target, whose sharp asymptotics are studied in \cite{HessChildsRosenzweigSerfaty2026Quantization}, rather than assuming that the particle dynamics can reach zero discrepancy at fixed $N$.  After separating this unavoidable floor, one may ask whether the mean-field and long-time limits commute and whether the empirical-measure images of the particle $\omega$-limit sets converge to the $\omega$-limit sets of the continuum flow.  Such a result would also require control of discrete metastability and an understanding of which continuum critical points admit recovery sequences of finite-particle equilibria.

\appendix
\section{Additional properties of the energy-kernel MMD}\label{app:completion-proof}
This appendix collects the metric properties of the energy MMD that are not needed before the PDE analysis.  We first record a Monte Carlo estimate, then compare MMD with moments and transport costs, identify two borderline failures, and finally determine the concrete Hilbert completion.

\subsection{Monte Carlo approximation}\label{ssec:app-monte-carlo-fourier}
Let $\mu\in\mathcal P_q(\R^d)$, let $(x_i)_{i\in[N]}$ be independent with law $\mu$, and set $\mu_N=N^{-1}\sum_{i=1}^N\delta_{x_i}$.

\begin{proposition}[Monte Carlo MMD scaling]\label{prop:empirical_mmd_bound}\label{cor:mmd_highprob}
For $0<q<2$,
\begin{equation}\label{eq:mmdmontecarlo}
\mathbb E\big[\mmd_q^2(\mu_N,\mu)\big]
=\frac1{2N}\iint_{(\R^d)^2}|x-y|^q\ud\mu(x)\ud\mu(y).
\end{equation}
Consequently, for every $\delta\in(0,1)$,
\begin{equation}\label{eq:mmd-high-probability}
\mmd_q^2(\mu_N,\mu)
\le\frac1{2\delta N}\iint_{(\R^d)^2}|x-y|^q\ud\mu(x)\ud\mu(y)
\end{equation}
with probability at least $1-\delta$.
\end{proposition}
The expectation identity is the standard diagonal/off-diagonal expansion of the three terms in the MMD energy; \eqref{eq:mmd-high-probability} is Markov's inequality.

\subsection{Moment and transport comparisons}\label{ssec:mmdmoment}\label{ssec:mmdwass}
We first show that MMD controls moments below order $q/2$, then compare MMD with transport costs in both directions, and finally extract the resulting topological consequences.

\begin{proposition}[Moment control by MMD]\label{prop:mmdtomoment}
Let $q\in(0,2)$ and $0<r<q/2$.  There is $C_{d,q,r}>0$ such that, for every $\mu,\rho\in\mathcal P_{q/2}(\R^d)$ and $x_0\in\R^d$,
\begin{equation}\label{eq:mmdtomoment}
\left|\int_{\R^d}|x-x_0|^r\ud(\mu-\rho)(x)\right|
\le C_{d,q,r}\mmd_q(\mu,\rho)^{2r/q}.
\end{equation}
In particular,
\begin{equation}\label{eq:mmdmomentcontrol}
\mathcal M_r(\rho)\le\mathcal M_r(\mu)+C_{d,q,r}\mmd_q(\mu,\rho)^{2r/q}.
\end{equation}
\end{proposition}

\begin{proof}
Set $\nu:=\mu-\rho$.  The Schoenberg--L\'evy formula with exponent $r$, followed by a split at frequency $R>0$, Cauchy--Schwarz, \Cref{lem:genSob}, and $|\widehat\nu|\le2$, gives
\begin{equation}\label{eq:moment-frequency-split}
\left|\int_{\R^d}|x-x_0|^r\ud\nu(x)\right|
\le C_{d,q,r}\left(R^{(q-2r)/2}\mmd_q(\mu,\rho)+R^{-r}\right).
\end{equation}
If the MMD vanishes, let $R\to\infty$; otherwise choose $R=\mmd_q(\mu,\rho)^{-2/q}$.  This proves \eqref{eq:mmdtomoment}, and \eqref{eq:mmdmomentcontrol} follows by taking $x_0=0$.
\end{proof}

For $r>0$, write
\begin{equation}\label{eq:transport-cost-definition}
\mathsf T_r(\rho,\mu):=\inf_{\pi\in\Gamma(\rho,\mu)}\int_{\R^d\times\R^d}|x-y|^r\ud\pi(x,y).
\end{equation}
For $r\ge1$, $W_r:=\mathsf T_r^{1/r}$ is the Wasserstein distance of order $r$; for $0<r<1$, $\mathsf T_r$ is the natural transport cost.

\begin{proposition}[Quantitative MMD/transport comparisons]\label{prop:mmdwass}
Let $q\in(0,2)$, set $s=(d+q)/2$, and let $\rho,\mu\in\mathcal P_{q/2}(\R^d)$.
\begin{enumerate}[label=\textup{(\arabic*)}]
\item One has
\begin{equation}\label{eq:wassmmd}
\mmd_q(\rho,\mu)\le C_{d,q}\mathsf T_{q/2}(\rho,\mu).
\end{equation}
\item If $\gamma>1$ and, for some $S>0$, $\mathcal M_\gamma(\rho)+\mathcal M_\gamma(\mu)\le S$, then
\begin{equation}\label{eq:mmd-to-W1}
W_1(\rho,\mu)\le C_{d,q,\gamma,S}\mmd_q(\rho,\mu)^{\alpha_{d,q,\gamma}},
\qquad
\alpha_{d,q,\gamma}:=\frac{2(\gamma-1)}{\gamma(d+q+2)-q}.
\end{equation}
Consequently, for $0<r\le1$,
\begin{equation}\label{eq:mmd-to-Tr}
\mathsf T_r(\rho,\mu)\le C_{d,q,r,\gamma,S}\mmd_q(\rho,\mu)^{r\alpha_{d,q,\gamma}}.
\end{equation}
\item If $p\ge1$, $\gamma>p$, and, for some $S>0$, $\mathcal M_\gamma(\rho)+\mathcal M_\gamma(\mu)\le S$, then
\begin{equation}\label{eq:mmd-to-Wp}
W_p(\rho,\mu)
\le C_{d,q,p,\gamma,S}\mmd_q(\rho,\mu)^{\frac{\alpha_{d,q,\gamma}(\gamma-p)}{p(\gamma-1)}}.
\end{equation}
\end{enumerate}
\end{proposition}

\begin{proof}
The feature map in \eqref{eq:feature-space-main} satisfies $\|\Phi_x-\Phi_y\|_H\le C_{d,q}|x-y|^{q/2}$.  Integrating against a coupling and taking the infimum proves \eqref{eq:wassmmd}.

To prove \eqref{eq:mmd-to-W1}, use Kantorovich--Rubinstein duality and test against a $1$-Lipschitz $\varphi$ with $\varphi(0)=0$.  If $\chi_R$ is a cutoff and $\eta_\varepsilon$ a mollifier, the moment tail, smoothing error, and $\dot H^s$--$\dot H^{-s}$ duality give
\begin{equation}\label{eq:transport-comparison-master}
\left|\int_{\R^d}\varphi\ud(\rho-\mu)\right|
\le C_{d,q,\gamma,S}\left(R^{1-\gamma}+\varepsilon+R^{d/2+1}\varepsilon^{-s}\mmd_q(\rho,\mu)\right).
\end{equation}
We have implicitly used that $\|\eta_\varepsilon*(\chi_R\varphi)\|_{\dot H^s}\le C R^{d/2+1}\varepsilon^{-s}$, which follows by differentiating at adjacent integer orders and interpolating.  Writing $D:=\mmd_q(\rho,\mu)$, if $D=0$, letting $\varepsilon\downarrow0$ and $R\to\infty$ in \eqref{eq:transport-comparison-master} gives $W_1(\rho,\mu)=0$.  Thus assume $D>0$ and balance the last two terms by taking
\begin{equation}\label{eq:transport-optimal-epsilon}
\varepsilon:=\bigl(R^{d/2+1}D\bigr)^{1/(s+1)}.
\end{equation}
This gives a contribution $R^{(d+2)/(d+q+2)}D^{2/(d+q+2)}$.  Balancing it with $R^{1-\gamma}$ amounts to choosing
\begin{equation}\label{eq:transport-optimal-radius}
R:=D^{-2/(\gamma(d+q+2)-q)},
\end{equation}
and yields the exponent $\alpha_{d,q,\gamma}$ in \eqref{eq:mmd-to-W1}; concavity gives \eqref{eq:mmd-to-Tr}.  Finally, for any coupling and $A>0$, H\"older's inequality and Young's product inequality give
\begin{equation}\label{eq:transport-moment-interpolation}
\int_{\R^d\times\R^d}|x-y|^p\ud\pi
\le A^{p-1}\int_{\R^d\times\R^d}|x-y|\ud\pi
+A^{p-\gamma}\int_{\R^d\times\R^d}|x-y|^\gamma\ud\pi.
\end{equation}
Taking an optimal coupling for $W_1$ and using the moment bound gives
\begin{equation}\label{eq:transport-Wp-intermediate}
W_p(\rho,\mu)^p\le C_{p,\gamma,S}\left(A^{p-1}W_1(\rho,\mu)+A^{p-\gamma}\right).
\end{equation}
The choice $A\asymp W_1(\rho,\mu)^{-1/(\gamma-1)}$ gives
\begin{equation}\label{eq:transport-Wp-W1-interpolation}
W_p(\rho,\mu)^p\le C_{p,\gamma,S}W_1(\rho,\mu)^{(\gamma-p)/(\gamma-1)},
\end{equation}
and \eqref{eq:mmd-to-Wp} follows from \eqref{eq:mmd-to-W1}.
\end{proof}

\begin{proposition}[MMD and transport convergence]\label{prop:mmd-topology}
Let $q\in(0,2)$ and $\mu_n,\mu\in\mathcal P_{q/2}(\R^d)$.
\begin{enumerate}[label=\textup{(\arabic*)}]
\item If $\mathsf T_{q/2}(\mu_n,\mu)\to0$, then $\mmd_q(\mu_n,\mu)\to0$.
\item If $\mmd_q(\mu_n,\mu)\to0$, then $\mathsf T_r(\mu_n,\mu)\to0$ for every $0<r<q/2$.
\end{enumerate}
\end{proposition}

\begin{proof}
The first assertion is \eqref{eq:wassmmd}.  For the second, \Cref{prop:mmdtomoment} gives convergence and uniform boundedness of every $r$-moment with $0<r<q/2$, while \Cref{lem:genSob} gives distributional convergence against Schwartz functions.  Tightness upgrades this to weak convergence, and weak convergence together with convergence of the $r$-moments is equivalent to convergence in $\mathsf T_r$.
\end{proof}

\subsection{Borderline examples}\label{ssec:app-borderline-examples}
The preceding comparison is sharp at the endpoint $q/2$, and $\mathcal P_{q/2}(\R^d)$ is not complete for the MMD metric.  The following two examples exhibit these failures using the summable sequence
\begin{equation}\label{eq:dyadic-kernel-sequence}
a_m:=f_q(2^{m/2}),\qquad
f_q(s):=s^q+s^{-q}-(s-s^{-1})^q,\qquad m\ge0.
\end{equation}
 As $s\to\infty$,
\begin{equation}\label{eq:dyadic-kernel-asymptotic}
f_q(s)=s^{-q}+q s^{q-2}+O(s^{q-4}).
\end{equation}
Thus the leading term is $s^{-q}$ for $0<q<1$, $2s^{-1}$ for $q=1$, and $q s^{q-2}$ for $1<q<2$.  In particular,
\begin{equation}\label{eq:dyadic-kernel-geometric-bound}
a_m\le C_q2^{-m\min\{q,2-q\}/2},
\end{equation}
so $a\in\ell^1(\mathbb N_0)$.

\begin{example}[Failure at the endpoint transport exponent]\label{ex:MMDtoWass}
Let $d=1$, $\mu=\delta_0$, and choose a fixed $c_q>0$ sufficiently small that $c_q\sum_{k=1}^\infty2^{-kq/2}\le1$.  Set $C_n:=c_qn^{-3/4}$ and
\begin{equation}\label{eq:borderline-transport-example}
\mu_n:=\sum_{k=1}^n C_n2^{-kq/2}\delta_{2^k}
+\left(1-\sum_{k=1}^n C_n2^{-kq/2}\right)\delta_0.
\end{equation}
Then $\mu_n\rightharpoonup\delta_0$, and the centered-kernel calculation gives
\begin{equation}\label{eq:borderline-transport-scaling}
\mmd_q^2(\mu_n,\delta_0)
=\frac{C_n^2}{2}\sum_{j,k=1}^n a_{|j-k|}\asymp nC_n^2\longrightarrow0,
\qquad
\mathcal M_{q/2}(\mu_n)=nC_n\longrightarrow\infty.
\end{equation}
Thus MMD convergence does not imply convergence in the $q/2$-transport topology.
\end{example}

\begin{example}[MMD-Cauchy sequence with a weak limit outside $\mathcal P_{q/2}$]\label{ex:mmd-incomplete-example}
Again let $d=1$.
Let
\begin{equation}\label{eq:incomplete-limit-measure}
\mu:=A^{-1}\sum_{k=1}^\infty\frac{2^{-kq/2}}{k}\delta_{2^k},
\qquad
A:=\sum_{k=1}^\infty\frac{2^{-kq/2}}{k},
\end{equation}
and move the tail mass to the origin:
\begin{equation}\label{eq:incomplete-truncations}
\mu_n:=A^{-1}\left(\sum_{k=1}^n\frac{2^{-kq/2}}{k}\delta_{2^k}
+\sum_{k=n+1}^\infty\frac{2^{-kq/2}}{k}\delta_0\right).
\end{equation}
Each $\mu_n$ has finite support and hence belongs to $\mathcal P_q(\R)$.  Moreover, for every $\phi\in C_b(\R)$,
\begin{equation}\label{eq:incomplete-weak-convergence}
\left|\int_{\R}\phi\ud(\mu_n-\mu)\right|
\le\frac{2\|\phi\|_{L^\infty}}{A}
\sum_{k=n+1}^\infty\frac{2^{-kq/2}}{k}
\longrightarrow0,
\end{equation}
so $\mu_n\rightharpoonup\mu$.  The limit belongs to $\mathcal P_r$ for every $0<r<q/2$ but not to $\mathcal P_{q/2}$.  For $n>m$, the same dyadic-kernel identity and the $\ell^1$ convolution bound give
\begin{equation}\label{eq:incomplete-cauchy-bound}
\mmd_q^2(\mu_n,\mu_m)
\le C_{q,A}\sum_{k=m+1}^n\frac1{k^2}
\le\frac{C_{q,A}}{m}.
\end{equation}
Hence $(\mu_n)$ is MMD-Cauchy although its weak limit is outside $\mathcal P_{q/2}$.
\end{example}

\subsection{Concrete Hilbert completion}\label{ssec:mmdcompletion}
The preceding example shows that the completion must contain probability measures outside $\mathcal P_{q/2}(\R^d)$.  Guided by \Cref{lem:genSob}, we identify the additional elements with probability measures $\mu$ for which $\widehat\mu-1$ belongs to the weighted complex Hilbert space $H_q$.

Retaining the feature map $\Phi_x$ from \eqref{eq:feature-space-main}, set
\begin{equation}\label{eq:completion-Hq-definition}
H_q:=L^2\bigl(\R^d,C_{d,q}|2\pi\xi|^{-(d+q)}\ud\xi;\mathbb C\bigr),
\end{equation}
and
\begin{equation}\label{eq:completion-space-definition}
\widetilde{\mathcal P}_q(\R^d):=\{\mu\in\mathcal P(\R^d):\widehat\mu-1\in H_q\}.
\end{equation}

\begin{proposition}[Characterization of the MMD completion]\label{prop:completion-full-range}\label{cor:completion-full-range}
The MMD completion of $\mathcal P_{q/2}(\R^d)$ is identified, through the isometric embedding $\mu\mapsto\widehat\mu-1\in H_q$, with $\widetilde{\mathcal P}_q(\R^d)$.  Moreover, for every $\mu\in\mathcal P(\R^d)$, the compactly supported truncations
\begin{equation}\label{eq:completion-truncations}
\mu_R:=\mu\lfloor_{B_R}+\mu(B_R^c)\delta_0
\end{equation}
belong to $\mathcal P_{q/2}(\R^d)$, and
\begin{equation}\label{eq:completion-truncation-convergence}
\mu\in\widetilde{\mathcal P}_q(\R^d)
\quad\Longleftrightarrow\quad
\|\widehat\mu_R-\widehat\mu\|_{H_q}\longrightarrow0.
\end{equation}
\end{proposition}

The proof has two parts.  First, compact truncations are dense in $\widetilde{\mathcal P}_q$, using positivity when $q\le1$ and dyadic square-tail control together with annular almost orthogonality \mbox{when $q>1$}.  Second, every $H_q$-limit of functions $\widehat\mu_j-1$, with $\mu_j\in\mathcal P_{q/2}(\R^d)$, is of the form $\widehat\mu-1$ for some $\mu\in\widetilde{\mathcal P}_q(\R^d)$.  With the notation above, we have
\begin{equation}\label{eq:centered-kernel-fourier-feature}
\Real\langle\Phi_x,\Phi_y\rangle_{H_q}=k_q(x,y).
\end{equation}
By \eqref{eq:centered-positive-kernel-intro}, when $0<q\le1$, subadditivity gives
\begin{equation}\label{eq:kq-positive-bounded}
0\le k_q(x,y)\le |x|^q\wedge|y|^q.
\end{equation}

To turn this pointwise positivity into an $H_q$-estimate for compact truncations, we first identify the relevant Hilbert-space pairing.

\begin{lemma}[Weak pairing identity in $H_q$]\label{lem:completion-weak-pairing}
Let $0<q\le1$ and fix $\mu\in\widetilde{\mathcal P}_q(\R^d)$.  Put
$g:=\widehat\mu-1\in H_q$.  If $\eta$ is a finite nonnegative Borel measure with compact
support, then
\begin{equation}
\Psi_\eta:=\int_{\R^d}\Phi_x\ud\eta(x)
\end{equation}
is a Bochner integral in $H_q$, and
\begin{equation}\label{eq:completion-weak-pairing}
\Real\langle \Psi_\eta,g\rangle_{H_q}
=\iint_{\R^d\times\R^d}k_q(x,y)\ud\eta(x)\ud\mu(y).
\end{equation}
In particular, for every fixed $x\in\R^d$,
\begin{equation}\label{eq:completion-duality}
\Real\langle \Phi_x,g\rangle_{H_q}
=\int_{\R^d}k_q(x,y)\ud\mu(y).
\end{equation}
\end{lemma}

\begin{proof}
The Bochner integral defining $\Psi_\eta$ is well-defined because $\eta$ has compact support and
$\|\Phi_x\|_{H_q}=C_{d,q}^{1/2}|x|^{q/2}$.  The right-hand side of
\eqref{eq:completion-weak-pairing} is finite by \eqref{eq:kq-positive-bounded}, since
$k_q(x,y)\le |x|^q$ and $\eta$ has compact support.

Define
\begin{equation}
\kappa_\eta(y):=\int_{\R^d}k_q(x,y)\ud\eta(x).
\end{equation}
Then $\kappa_\eta$ is bounded and continuous, and $\kappa_\eta(0)=0$.  More precisely, in $\mathcal S'(\R^d)$,
\begin{equation}\label{eq:kappa-eta-fourier-distribution}
\widehat{\kappa_\eta}
=C_{d,q}|2\pi\xi|^{-(d+q)}\Psi_\eta(\xi)
+\frac12\left(\int_{\R^d}|x|^q\ud\eta(x)\right)\delta_0.
\end{equation}
When $q=1$, the first term on the right-hand side is understood in the symmetric principal-value sense.  To see that \eqref{eq:kappa-eta-fourier-distribution} holds, we argue as follows.  For $\eta=\delta_x$, it is the Fourier form of the centered Schoenberg--L\'evy identity
\eqref{eq:centered-kernel-fourier-feature}; the constant term $\frac12|x|^q$ in $k_q(x,\cdot)$ produces $\frac12|x|^q\delta_0$.  By linearity, the identity therefore holds for every finite atomic measure $\eta$.

For a general compactly supported $\eta$, fix a compact set
$K\supset\supp\eta$.  Choose finite Borel partitions
$K=\bigsqcup_j A_{n,j}$ with
$\delta_n:=\max_j\diam(A_{n,j})\to0$, choose $x_{n,j}\in A_{n,j}$,
and set $\eta_n:=\sum_j\eta(A_{n,j})\delta_{x_{n,j}}$.  Using the
feature-map estimate
$\|\Phi_x-\Phi_y\|_{H_q}\le C_{d,q}|x-y|^{q/2}$ and Minkowski's
inequality, we obtain
\begin{align}
\|\Psi_{\eta_n}-\Psi_\eta\|_{H_q}
&\le \sum_j\int_{A_{n,j}}
   \|\Phi_{x_{n,j}}-\Phi_x\|_{H_q}\ud\eta(x)\notag\\
&\le C_{d,q}\eta(K)\delta_n^{q/2}\longrightarrow0.
\end{align}
Thus $\Psi_{\eta_n}\to\Psi_\eta$ strongly in $H_q$.  Uniform continuity of $x\mapsto|x|^q$ on the fixed compact set gives
\begin{equation}
\int_{\R^d}|x|^q\ud\eta_n(x)
\longrightarrow
\int_{\R^d}|x|^q\ud\eta(x).
\end{equation}
Moreover, $\kappa_{\eta_n}\to\kappa_\eta$ uniformly on compact sets, and \eqref{eq:kq-positive-bounded} gives a uniform bound on $\kappa_{\eta_n}$.  Passing to the limit in $\mathcal S'(\R^d)$ proves \eqref{eq:kappa-eta-fourier-distribution}.  Its $\delta_0$ term does not contribute when paired with $\mu-\delta_0$, since $\widehat\mu(0)-1=0$.

We now pair \eqref{eq:kappa-eta-fourier-distribution} with the zero-mass finite measure
$\mu-\delta_0$, whose Fourier transform is $g$.  To make the weak pairing explicit, convolve
$\kappa_\eta$ and $\mu-\delta_0$ with a Schwartz approximate identity.  For the smoothed objects,
Plancherel gives
\begin{equation}
\int \kappa_{\eta,\varepsilon}\ud(\mu-\delta_0)_\varepsilon
=
\Real\int C_{d,q}|2\pi\xi|^{-(d+q)}\Psi_\eta(\xi)\overline{g(\xi)}
|\widehat\theta_\varepsilon(\xi)|^2\ud\xi.
\end{equation}
The Fourier side converges to $\Real\langle\Psi_\eta,g\rangle_{H_q}$ by Cauchy--Schwarz in $H_q$.
The physical side converges to $\int\kappa_\eta\ud(\mu-\delta_0)$, because
$\kappa_\eta$ is bounded and uniformly continuous and $\mu-\delta_0$ is a finite measure.  Since
$\kappa_\eta(0)=0$, this physical limit is $\int\kappa_\eta\ud\mu$.  Therefore
\begin{equation}
\Real\langle\Psi_\eta,g\rangle_{H_q}
=\int_{\R^d}\kappa_\eta(y)\ud\mu(y)
=\iint_{\R^d\times\R^d}k_q(x,y)\ud\eta(x)\ud\mu(y),
\end{equation}
which proves \eqref{eq:completion-weak-pairing}.  Taking $\eta=\delta_x$ proves
\eqref{eq:completion-duality}.
\end{proof}

\begin{lemma}[Truncation approximation for $0<q\le1$]\label{prop:completion-q-le-one}
Let $0<q\le1$ and let $\mu\in\widetilde{\mathcal P}_q(\R^d)$.  For the truncations $\mu_R$ defined in \eqref{eq:completion-truncations}, one has
\begin{equation}
\mu_R\in\mathcal P_{q/2}(\R^d)
\end{equation}
and
\begin{equation}
\lim_{R\to\infty}\|\widehat\mu_R-\widehat\mu\|_{H_q}=0.
\end{equation}
Consequently, $\widehat\mu-1$ belongs to the $H_q$-closure of
$\{\widehat\nu-1:\nu\in\mathcal P_{q/2}(\R^d)\}$ for every
$\mu\in\widetilde{\mathcal P}_q(\R^d)$.
\end{lemma}

\begin{proof}
Set
\begin{equation}
g:=\widehat\mu-1\in H_q,
\qquad
g_R:=\widehat\mu_R-1=\int_{|x|\le R}\Phi_x\ud\mu(x).
\end{equation}
The measure $\mu_R$ is compactly supported and therefore belongs to $\mathcal P_{q/2}$.  Observe that
\begin{equation}\label{eq:completion-IR}
\|g_R\|_{H_q}^2
=\iint_{|x|\le R,\,|y|\le R}k_q(x,y)\ud\mu(x)\ud\mu(y).
\end{equation}
The function $R\mapsto\|g_R\|_{H_q}^2$ is increasing because $k_q\ge0$.

We claim that $\|g_R\|_{H_q}^2$ is bounded in $R$. Indeed, apply \Cref{lem:completion-weak-pairing} with $\eta=\mu\lfloor_{B_R}$.  Since
$\Psi_\eta=g_R$, we obtain
\begin{equation}
\Real\langle g_R,g\rangle_{H_q}
=\iint_{|x|\le R}k_q(x,y)\ud\mu(x)\ud\mu(y)
\ge \|g_R\|_{H_q}^2.
\end{equation}
Therefore
\begin{equation}
\|g_R\|_{H_q}^2
\le \Real\langle g_R,g\rangle_{H_q}
\le \|g_R\|_{H_q}\|g\|_{H_q},
\end{equation}
and hence either $g_R=0$ or $\|g_R\|_{H_q}\le\|g\|_{H_q}$.  Thus $R\mapsto\|g_R\|_{H_q}^2$ is increasing and bounded, so $\|g_R\|_{H_q}^2$ has a finite limit as $R\to\infty$.

If $0<R<S$, then by \eqref{eq:completion-IR} and $k_q\ge0$,
\begin{equation}
\|g_S-g_R\|_{H_q}^2
=\iint_{R<|x|\le S,\,R<|y|\le S}k_q(x,y)\ud\mu(x)\ud\mu(y)
\le \|g_S\|_{H_q}^2-\|g_R\|_{H_q}^2.
\end{equation}
Since $\|g_R\|_{H_q}^2$ converges, it follows that $(g_R)_R$ is Cauchy
in $H_q$; let $h$ be its limit.  Weighted Cauchy--Schwarz gives a
continuous embedding $H_q\hookrightarrow\mathcal S'(\R^d)$, so
$g_R\to h$ in $\mathcal S'(\R^d)$.  On the other hand,
$\mu_R\rightharpoonup\mu$ narrowly, hence $\mu_R\to\mu$ in
$\mathcal S'(\R^d)$.  Continuity of the Fourier transform on
$\mathcal S'(\R^d)$ therefore gives $g_R\to g$ in $\mathcal S'(\R^d)$.
By uniqueness of distributional limits, $h=g$.  Therefore $g_R\to g$
in $H_q$, which is the asserted convergence.
\end{proof}

\begin{remark}[Limitation of the direct truncation argument]\label{rmk:completion-q-gt-one}
The proof of \Cref{prop:completion-q-le-one} uses the pointwise positivity of
$k_q(x,y)=\frac12(|x|^q+|y|^q-|x-y|^q)$, which follows from subadditivity only when
$0<q\le1$.  For $1<q<2$,
\begin{equation}
k_q(x,-x)=\frac12(2-2^q)|x|^q<0.
\end{equation}
Thus the monotone truncation argument has no sign content beyond $q=1$.  The completion statement
nevertheless remains true in the full range $0<q<2$; for $1<q<2$, the proof below replaces
monotonicity by dyadic square-tail control and annular almost orthogonality.
\end{remark}

\begin{lemma}[Dyadic square-tail bound]\label{lem:completion-dyadic-square-tail}
Let $1<q<2$ and let $\mu\in\widetilde{\mathcal P}_q(\R^d)$.  For $n\ge0$ set
\begin{equation}\label{eq:completion-spatial-annuli}
A_n:=\{x\in\R^d:2^n<|x|\le2^{n+1}\},
\qquad
\beta_n:=\mu(A_n).
\end{equation}
Then
\begin{equation}\label{eq:completion-square-tail}
\sum_{n=0}^\infty 2^{nq}\beta_n^2<\infty.
\end{equation}
\end{lemma}

\begin{proof}
The idea is to detect the mass $\beta_n$ on the spatial annulus $A_n$ by integrating the nonnegative quantity $1-\Real\widehat\mu$ against the $H_q$ weight on the reciprocal frequency annulus $\Omega_n$.  Positivity prevents cancellation, and Cauchy--Schwarz converts the resulting lower bound into the square-tail estimate.

Write
\begin{equation}
\nu(\xi):=\Real(1-\widehat\mu(\xi))
=\int_{\R^d}\bigl(1-\cos(2\pi\xi\cdot x)\bigr)\ud\mu(x)\ge0.
\end{equation}
For $n\ge0$ define the frequency annulus
\begin{equation}
\Omega_n:=\{\xi\in\R^d:2^{-n-1}\le |\xi|\le2^{-n}\}.
\end{equation}
We claim there is a constant $c=c(d,q)>0$ such that
\begin{equation}\label{eq:completion-osc-lower}
\int_{\Omega_n}\bigl(1-\cos(2\pi\xi\cdot x)\bigr)C_{d,q}|2\pi\xi|^{-(d+q)}\ud\xi
\ge c2^{nq}
\qquad
\forall n\ge0,\quad \forall x\in A_n.
\end{equation}
Indeed, after the changes of variables $\xi=2^{-n}\zeta$ and $x=2^n z$, the left-hand side equals
$2^{nq}$ times
\begin{equation}
C_{d,q}\int_{1/2\le |\zeta|\le1}
\bigl(1-\cos(2\pi\zeta\cdot z)\bigr)|2\pi\zeta|^{-(d+q)}\ud\zeta,
\end{equation}
where $1<|z|\le2$.  The last integral is a continuous positive function of $z$ on the compact
annulus $\{1\le |z|\le2\}$; positivity follows because the nonnegative integrand cannot vanish
identically unless $z=0$.

Using \eqref{eq:completion-osc-lower} and $\nu\ge0$,
\begin{equation}
\int_{\Omega_n}\nu(\xi)C_{d,q}|2\pi\xi|^{-(d+q)}\ud\xi
\ge \int_{A_n}\int_{\Omega_n}
\bigl(1-\cos(2\pi\xi\cdot x)\bigr)C_{d,q}|2\pi\xi|^{-(d+q)}\ud\xi\ud\mu(x)
\ge c2^{nq}\beta_n.
\end{equation}
Moreover, $\int_{\Omega_n}C_{d,q}|2\pi\xi|^{-(d+q)}\ud\xi\le C2^{nq}$.  Hence Cauchy--Schwarz gives
\begin{equation}
\int_{\Omega_n}|\widehat\mu(\xi)-1|^2C_{d,q}|2\pi\xi|^{-(d+q)}\ud\xi
\ge \int_{\Omega_n}\nu(\xi)^2C_{d,q}|2\pi\xi|^{-(d+q)}\ud\xi
\ge c2^{nq}\beta_n^2.
\end{equation}
The annuli $\Omega_n$ are disjoint and $\widehat\mu-1\in H_q$, so summing in $n$ proves
\eqref{eq:completion-square-tail}.
\end{proof}

\begin{lemma}[Annular almost orthogonality]\label{lem:completion-annular-almost-orthogonality}
Let $1<q<2$.  For $n\ge0$, let $\eta_n$ be a finite nonnegative measure supported in
the annulus $A_n$ from \eqref{eq:completion-spatial-annuli}, set $\beta_n:=\eta_n(A_n)$, and define
\begin{equation}
f_n:=\int_{A_n}\Phi_x\ud\eta_n(x)\in H_q.
\end{equation}
Then, for all $0\le n\le m$,
\begin{equation}\label{eq:completion-annular-cross}
\bigl|\Real\langle f_n,f_m\rangle_{H_q}\bigr|
\le C_{d,q}2^{-(m-n)(1-q/2)}
\bigl(2^{nq/2}\beta_n\bigr)
\bigl(2^{mq/2}\beta_m\bigr).
\end{equation}
\end{lemma}

\begin{proof}
If $m=n$, then Minkowski's inequality and
$\|\Phi_x\|_{H_q}=C_{d,q}^{1/2}|x|^{q/2}$ give
$\|f_n\|_{H_q}\le C_{d,q}2^{nq/2}\beta_n$, which proves
\eqref{eq:completion-annular-cross} after enlarging the constant.  If $0<m-n\le2$, then
Cauchy--Schwarz and the diagonal case give the desired bound, again after enlarging the constant.  Assume
$m-n\ge3$.  Then $|x|\le |y|/2$ for $x\in A_n$ and $y\in A_m$.  By the mean-value theorem,
\begin{equation}
\bigl||y|^q-|x-y|^q\bigr|\le C_q |x|\,|y|^{q-1},
\end{equation}
and, since $q>1$ and $|x|\le |y|$,
\begin{equation}
|x|^q\le |x|\,|y|^{q-1}.
\end{equation}
Consequently,
\begin{equation}\label{eq:completion-kq-separated}
|k_q(x,y)|\le C_q |x|\,|y|^{q-1}
\le C_q2^n2^{m(q-1)}.
\end{equation}
Using
\begin{equation}
\Real\langle f_n,f_m\rangle_{H_q}
=\iint_{A_n\times A_m}k_q(x,y)\ud\eta_n(x)\ud\eta_m(y),
\end{equation}
we obtain
\begin{equation}
\bigl|\Real\langle f_n,f_m\rangle_{H_q}\bigr|
\le C_q2^n2^{m(q-1)}\beta_n\beta_m
=C_q2^{-(m-n)(1-q/2)}
\bigl(2^{nq/2}\beta_n\bigr)
\bigl(2^{mq/2}\beta_m\bigr).
\end{equation}
This proves \eqref{eq:completion-annular-cross}.
\end{proof}

We can now combine the dyadic square-tail bound with annular almost orthogonality to control the truncation tails in $H_q$.

\begin{lemma}[Truncation approximation for $1<q<2$]\label{prop:completion-q-gt-one}
Let $1<q<2$ and let $\mu\in\widetilde{\mathcal P}_q(\R^d)$.  For the truncations $\mu_R$ defined in \eqref{eq:completion-truncations}, one has
\begin{equation}
\mu_R\in\mathcal P_{q/2}(\R^d)
\end{equation}
and
\begin{equation}
\lim_{R\to\infty}\|\widehat\mu_R-\widehat\mu\|_{H_q}=0.
\end{equation}
Consequently, $\widehat\mu-1$ belongs to the $H_q$-closure of
$\{\widehat\nu-1:\nu\in\mathcal P_{q/2}(\R^d)\}$ for every
$\mu\in\widetilde{\mathcal P}_q(\R^d)$.
\end{lemma}

\begin{proof}
With $A_n$ and $\beta_n$ as in \eqref{eq:completion-spatial-annuli}, set
\begin{equation}
f_n:=\int_{A_n}\Phi_x\ud\mu(x).
\end{equation}
By \Cref{lem:completion-dyadic-square-tail}, the sequence
$a_n:=2^{nq/2}\beta_n$ belongs to $\ell^2(\mathbb N_0)$.  We claim that the series
$\sum_{n\ge0}f_n$ converges in $H_q$.  Indeed, expanding the square and applying
\eqref{eq:completion-annular-cross} to the diagonal and off-diagonal terms gives, for $0\le N\le M$,
\begin{equation}
\left\|\sum_{n=N}^M f_n\right\|_{H_q}^2
\le C_{d,q}\left(\sum_{n=N}^M a_n^2
+\sum_{N\le n<m\le M}2^{-(m-n)(1-q/2)}a_na_m\right)
\le C_{d,q}\sum_{n=N}^M a_n^2.
\end{equation}
The last step is Schur's inequality for the convolution kernel
$j\mapsto2^{-j(1-q/2)}\in\ell^1(\mathbb N)$.
Since $a\in\ell^2$, the right-hand side tends to zero as $N\to\infty$, uniformly in $M\ge N$.

Hence the series converges in $H_q$, and its tails
\begin{equation}
T_N:=\sum_{n=N}^\infty f_n
\end{equation}
satisfy $\|T_N\|_{H_q}\to0$.  For each fixed $\xi$, the finite partial
sums over $N\le n\le M$ converge by bounded convergence to
\begin{equation}
\int_{|x|>2^N}(e^{-2\pi i\xi\cdot x}-1)\ud\mu(x),
\end{equation}
and a subsequence of the same partial sums converges
$\mathcal L^d$-a.e.\ to $T_N$.  Hence
\begin{equation}
T_N(\xi)=\int_{|x|>2^N}(e^{-2\pi i\xi\cdot x}-1)\ud\mu(x)
\quad\text{for $\mathcal L^d$-a.e. }\xi.
\end{equation}
Therefore $\|\widehat\mu_{2^N}-\widehat\mu\|_{H_q}=\|T_N\|_{H_q}\to0$.

It remains only to pass from dyadic radii to arbitrary $R\to\infty$.  If $2^N\le R<2^{N+1}$, then
\begin{equation}
\widehat\mu_R-\widehat\mu
=-\int_{|x|>R}\Phi_x\ud\mu(x)
=-\int_{A_N\cap\{|x|>R\}}\Phi_x\ud\mu(x)-T_{N+1}.
\end{equation}
The first term has $H_q$ norm at most $C_{d,q}2^{Nq/2}\beta_N=C_{d,q}a_N$, which tends to zero,
and $\|T_{N+1}\|_{H_q}\to0$.  This proves the asserted convergence.
\end{proof}

The two preceding truncation arguments show the set inclusion
$\{\widehat\mu-1:\mu\in\widetilde{\mathcal P}_q(\R^d)\}
\subset
\overline{\{\widehat\nu-1:\nu\in\mathcal P_{q/2}(\R^d)\}}^{H_q}$.
For the reverse inclusion, it
remains to show that every $H_q$-limit of a sequence $\widehat\mu_j-1$, with
$\mu_j\in\mathcal P_{q/2}(\R^d)$, is of the form $\widehat\mu-1$ for some
$\mu\in\widetilde{\mathcal P}_q(\R^d)$.

\begin{lemma}[Identification of probability-measure limits in $H_q$]\label{lem:completion-closed-probability-part}
Let $0<q<2$.  Suppose that $\mu_j\in\mathcal P_{q/2}(\R^d)$ and
\begin{equation}
\widehat\mu_j-1\to h\qquad\text{in }H_q.
\end{equation}
Then there exists $\mu\in\widetilde{\mathcal P}_q(\R^d)$ such that
\begin{equation}
h=\widehat\mu-1\qquad\mathcal L^d\text{-a.e.}
\end{equation}
Consequently,
\begin{equation}
\overline{
\bigl\{\widehat\mu-1:\mu\in\mathcal P_{q/2}(\R^d)\bigr\}
}^{H_q}
\subset
\bigl\{\widehat\mu-1:\mu\in\widetilde{\mathcal P}_q(\R^d)\bigr\}.
\end{equation}
\end{lemma}

\begin{proof}
Fix $r\in(0,q/2)$.  Since $\widehat\mu_j-1$ is Cauchy in $H_q$, the quantities
\begin{equation}
\mmd_q(\mu_j,\mu_1)=\|(\widehat\mu_j-1)-(\widehat\mu_1-1)\|_{H_q}
\end{equation}
are bounded.  By
\Cref{prop:mmdtomoment},
\begin{equation}
\sup_j\mathcal M_r(\mu_j)<\infty.
\end{equation}
Thus $\{\mu_j\}$ is tight.  Passing to a subsequence, $\mu_j\rightharpoonup\mu$ narrowly for some
$\mu\in\mathcal P(\R^d)$, and therefore $\widehat\mu_j(\xi)-1\to\widehat\mu(\xi)-1$ for every
$\xi\in\R^d$.

After passing to a further subsequence,
$\widehat\mu_j-1\to h$ $\mathcal L^d$-a.e.
Hence $h=\widehat\mu-1$ $\mathcal L^d$-a.e.
Since $h\in H_q$, this implies $\mu\in\widetilde{\mathcal P}_q(\R^d)$.
\end{proof}

We now combine the preceding truncation and closedness results to prove
\Cref{prop:completion-full-range}.

\begin{proof}[Proof of \Cref{prop:completion-full-range}]
The inclusion $\mathcal P_{q/2}\subset\widetilde{\mathcal P}_q$ follows exactly as in the proof of
\Cref{lem:genSob}: since $\|\Phi_x\|_{H_q}=C_{d,q}^{1/2}|x|^{q/2}$, the Bochner integral
$\int\Phi_x\ud\mu(x)$ is well-defined in $H_q$ whenever $\mu\in\mathcal P_{q/2}$, and the resulting
$H_q$ element has continuous representative $\widehat\mu-1$.

The density of the compactly supported truncations in
$\widetilde{\mathcal P}_q(\R^d)$ is \Cref{prop:completion-q-le-one} for $0<q\le1$ and
\Cref{prop:completion-q-gt-one} for $1<q<2$.  The reverse implication in
\eqref{eq:completion-truncation-convergence} is immediate: if the displayed $H_q$ norm tends to zero, then $\widehat\mu-\widehat\mu_R\in H_q$ for all sufficiently large $R$; since $\widehat\mu_R-1\in H_q$, it follows that $\widehat\mu-1\in H_q$.  The reverse inclusion in the completion identification follows from
\Cref{lem:completion-closed-probability-part}.  Since \Cref{lem:genSob} identifies $\mmd_q(\mu,\nu)$
with $\|\widehat\mu-\widehat\nu\|_{H_q}$, the stated completion
identification follows.
\end{proof}

\printbibliography
\end{document}